\documentclass[reqno]{article}

\usepackage{amssymb,amsmath,amsthm}
\usepackage{newpxtext,newpxmath} % commentare per tornare ai fonts abituali
\usepackage{mathtools}
\usepackage{amsfonts}
\usepackage{a4wide}
\usepackage[normalem]{ulem}
\usepackage{mathrsfs,enumerate}
\usepackage[bookmarks=false]{hyperref}

\usepackage{color}
\usepackage{pdfsync}

\usepackage[
color,
final]
{showkeys}   
 
\definecolor{refkeybis}{gray}{.65}%     per avere le labels stampate chiare
\definecolor{labelkeybis}{gray}{.65}%   basta modificare .50:= intensita grigio
{\makeatletter
\def\SK@refcolor{\color{refkeybis}}%
\def\SK@labelcolor{\color{labelkeybis}}}

\numberwithin{equation}{section}  % per cambiare la numerazione
\newtheorem{theorem}{Theorem}[section]  % per cambiare la numerazione
\newtheorem{lemma}[theorem]{Lemma}
\newtheorem{corollary}[theorem]{Corollary}
\newtheorem{proposition}[theorem]{Proposition}

\newtheorem{remark}[theorem]{Remark}

\newtheorem{definition}[theorem]{Definition}

\newcommand{\D}{\mathbb{D}}
\newcommand{\R}{\mathbb{R}}
\newcommand{\Q}{\mathbb{Q}}
\newcommand{\N}{\mathbb{N}}

\newcommand{\X}{\mathbb{X}}
\newcommand{\DD}{\mathscr{D}}

\newcommand{\mm}{{\mbox{\boldmath$m$}}}

\renewcommand{\tt}{{\mbox{\boldmath$t$}}}

\newcommand{\tauV}{{\kern-3pt\tau}}

\newcommand{\mmu}{{\mbox{\boldmath$\mu$}}}

\newcommand{\sfd}{{\sf d}}

\newcommand{\sfv}{{\sf v}}
\newcommand{\sfx}{{\sf x}}
\newcommand{\sfy}{{\sf y}}

\newcommand{\sfS}{{\sf S}}

\newcommand{\sfU}{{\sf U}}

\newcommand{\rmc}{{\mathrm c}}

\newcommand{\rme}{{\mathrm e}}

\newcommand{\rmC}{{\mathrm C}}

\newcommand{\rmJ}{{\mathrm J}}

\newcommand{\Kliminf}{K\kern-3pt-\kern-2pt\mathop{\rm lim\,inf}\limits}  % Kuratowski liminf di insiemi
\newcommand{\argmin}{\mathop{\rm argmin}\limits}   % argmin
\newcommand{\restr}[1]{\lower3pt\hbox{$|_{#1}$}}
\newcommand{\Restriction}[1]{\lower3pt\hbox{$|_{#1}$}}  %\Restriction{arg}  restrizione ad arg
\newcommand{\Leb}[1]{{\mathscr L}^{#1}}      % Misura di Lebesgue
\newcommand{\down}{\downarrow}              %frecce in su e in giu nei limiti
\newcommand{\up}{\uparrow}
\newcommand{\eps}{\varepsilon}  
\newcommand{\Probabilities}[1]{\mathscr P(#1)}          % misure di probabilita'
\newcommand{\ProbabilitiesTwo}[1]{\mathscr P_2(#1)}     % misure di probabilita' con momento quadratico finito
\newcommand{\DomainSlope}[1]{\mathrm{Dom}(|\partial#1|)}                 % Dominio della pendenza
\newcommand{\MetricSlope}[2]{|\partial#1|(#2)}              %Pendenza
\newcommand{\lMetricSlope}[2]{|\partial#1|_\ell(#2)}              %Pendenza di #1 valutata in #2                                
\newcommand{\MetricSlopeSquare}[2]{|\partial#1|^2(#2)}      %Pendenza
\newcommand{\lMetricSlopeSquare}[2]{|\partial#1|_\ell^2(#2)}      %Pendenza al quadrato                                
\newcommand{\PlainSlope}[1]{|\partial#1|}                   %Slope
\newcommand{\lPlainSlope}[1]{|\partial#1|_\ell}                                                
\newcommand{\Pc}[2]{\overline{#1}\kern-2pt^{\vphantom 0}_{#2}}
\newcommand{\Pcshift}[3]{\overline{#1}\kern-1pt^{#3}_{#2}}
\newcommand{\Pch}[2]{\overline{#1}^{\kern1pt h}_{\kern-2pt#2}}
\newcommand{\Pck}[2]{\overline{#1}^{\kern1pt k}_{\kern-2pt#2}}
\newcommand{\Pcinfty}[2]{\overline{#1}^{\kern1pt \infty}_{\kern-2pt#2}}
\newcommand{\Pchn}[2]{\overline{#1}^{\kern1pt h_n}_{\kern-2pt#2}}
\newcommand{\Pchnk}[2]{\overline{#1}^{\kern1pt h'_{n_k}}_{\kern-2pt#2}}
\newcommand{\Pchk}[2]{\overline{#1}^{\kern1pt h_{k}}_{\kern-2pt#2}}

\newcommand{\Pl}[2]{{#1}\kern-2pt^{\vphantom 0}_{#2}}
\newcommand{\Pcb}[2]{\underline{\phantom{u}}\kern-6pt #1_{#2}}

\newcommand{\nchi}{{\raise.3ex\hbox{$\chi$}}}
\newcommand{\MetricDerivative}[1]{|\dot{#1}|}                         %Derivata metrica

\newcommand{\MD}[2]{|\dot{#1}_{#2}|}
\newcommand{\MDS}[2]{|\dot{#1}_{#2}|^2}

\newcommand{\AmbientSpace}{X}  % lo spazio metrico Ambiente
\newcommand{\Geodesic}{\sfx}
\newcommand{\Gdc}{\Geodesic}

\newcommand{\DistName}{{\mathsf d}}          %Distanza
\newcommand{\lDistName}{\DistName_\ell}          %Distanza
\newcommand{\lDDistName}[1]{\DistName_{#1,\ell}}          %Distanza
\newcommand{\WDistName}{\DistName_W}          %Distanza
\newcommand{\Dist}[2]{\DistName(#1,#2)}
\newcommand{\lDist}[2]{\lDistName(#1,#2)}

\newcommand{\lDDist}[3]{\lDDistName#1(#2,#3)}

\newcommand{\DistSquare}[2]{\DistName^2(#1,#2)}

\newcommand{\WDistSquare}[2]{\WDistName^2(#1,#2)}

\newcommand{\GeoCon}[3]{\mathrm{Geo}_{#1}[#2\kern-1pt\to\kern-1pt#3]}
\newcommand{\GeoConPhi}[4]{\mathrm{Geo}^{\phi,#2}_{#1}[#3\kern-1pt\to\kern-1pt#4]}
\newcommand{\GeoConArg}[5]{\mathrm{Geo}^{#5,#2}_{#1}[#3\kern-1pt\to\kern-1pt#4]}

\newcommand{\GDom}[1]{[0,1]}

\newcommand{\GeoSpace}[1]{\mathrm{Geo}({#1})}

\renewcommand{\d}{{\mathrm d}}
\newcommand{\dt}{{\d t}}

\newcommand{\topref}[2]{\stackrel{\eqref{#1}}#2}

\newcommand{\DClass}[2]{[#2]_{#1,\ell}}
\newcommand{\System}{(\AmbientSpace,\DistName,\phi)}
\newcommand{\Syst}{\AmbientSpace,\DistName,\phi}

\newcommand{\DirSuperDer}[3]{{#1}'(#2;#3)}

\newcommand{\DirDistDer}[3]{[\dot{#1},#3]_{#2}}

\newcommand{\EExp}[2]{{\sf E}_{#1}(#2)}
\newcommand{\taueta}{{\tau,\eta}}
\newcommand{\TauEta}{{\tau,\eta}}
\newcommand{\discrMM}[2]{{\sf{MM}}_{#1}(#2)}
\newcommand{\GMM}[1]{{\sf{GMM}}(#1)}
\newcommand{\contMM}[1]{{\sf{MM}}(#1)}

\newcommand{\lambdam}{{\lambda\kern-2pt^-}}

\newcommand{\Urd}{{\frac{ \d}{\d t}\!\!}^+}

\newcommand{\lrd}{{\tfrac{ \d}{\dt}\!}_+}
\newcommand{\Lrd}{{\frac{ \d}{\dt}\!}_+}
\newcommand{\sfE}[2]{{\sf E}_{#1}(#2)}

\newcommand{\GSlope}[3]{\mathfrak {L}_{#1}[#2](#3)}
\newcommand{\GSlopeE}[2]{\mathfrak {L}_{#1}[#2]}
\newcommand{\GSlopeSquare}[3]{\mathfrak {L}_{#1}^2[#2](#3)}

\newcommand{\h}{k}

\newcommand{\n}{{n}}

\newcommand{\EVIshort}[1]{\ensuremath{\EVIname_{#1}}}
\newcommand{\EVI}[4]{\ensuremath{\EVIname_{#4}(#1,#2,#3)}}
\newcommand{\EVIname}{\ensuremath{\mathrm{EVI}}}
\newcommand{\Dom}[1]{\ensuremath{\mathrm{Dom}(#1)}}
\newcommand{\FlowName}{\ensuremath{\mathsf S}}
\newcommand{\Flow}[2]{\FlowName_{#1}(#2)}

\newcommand{\forevery}{\ensuremath{\text{for every }}}
\newcommand{\ooint}{o}
\newcommand{\Resolvent}[2]{\rmJ_{#1}[#2]}
\newcommand{\PhiF}{\mathscr F}
\newcommand{\PhiE}{\mathscr E}

\newcommand{\ICA}{{(\text{IC}1)}}
\newcommand{\ICB}{{(\text{IC}2)}}

\newcommand{\AC}{\mathrm{AC}}

\newcommand{\nc}{\normalcolor}

\renewcommand{\mm}{\mathfrak m}

\begin{document}
 \title{Gradient flows and Evolution Variational Inequalities in metric spaces. 
   I: structural properties}
 \author{Matteo Muratori \thanks{Partially supported by GNAMPA project 2017 \emph{Equazioni Diffusive Non-lineari in Contesti Non-Euclidei e Disuguaglianze Funzionali Associate}, by GNAMPA project 2018 \emph{Problemi Analitici e Geometrici Associati a EDP Non-Lineari Ellittiche e Paraboliche}, and by MIUR PRIN 2015 project
\emph{Calculus of Variations}. Email: \textsf{matteo.muratori@polimi.it}} \\
\emph{Dipartimento di Matematica, Politecnico di Milano}
\and
Giuseppe Savar\'e
\thanks{Partially supported by 
Cariplo Foundation and Regione Lombardia via project \emph{Variational
Evolution Problems and Optimal Transport}, by MIUR PRIN 2015 project
\emph{Calculus of Variations}, and by IMATI-CNR. Email: \textsf{giuseppe.savare@unipv.it}}\\
\emph{Dipartimento di Matematica,
Universit\`a di Pavia}}
%\date{12 ottobre 2004}
\date{}
\maketitle
\begin{abstract}

  This is the first of a series of papers devoted to a thorough analysis
  of the class of gradient flows in a metric space $(X,\DistName)$, that can be
  characterized by Evolution Variational Inequalities.
  We present new results
  concerning the structural properties of solutions to
  the $\mathrm{EVI}$ formulation, such as contraction, regularity,
  asymptotic expansion,
  precise energy identity, stability, asymptotic behaviour and
  their link with the geodesic convexity of the driving functional.
  
  Under the crucial assumption of the existence of an $\mathrm{EVI}$ gradient flow,
  we will also prove two main results:
  \begin{itemize}
    \renewcommand{\labelitemi}{$-$}
  \item the equivalence with the De Giorgi variational
    characterization of curves of maximal slope;
  \item the convergence of
    the Minimizing Movement-JKO scheme to the $\mathrm{EVI}$ gradient flow,
    with an explicit and
    uniform error estimate of order $1/2$ with respect to the step
    size, independent of any geometric hypothesis (as upper or lower
    curvature bounds) on $\DistName$.
  \end{itemize}
  In order to avoid any compactness assumption, we will also introduce
  a suitable relaxation of the Minimizing Movement algorithm obtained
  by
  the Ekeland variational principle, and we will prove its uniform convergence as well.
\end{abstract}
\tableofcontents
%\part{Gradient Flows and Evolution Variational Inequalities}
\section{Introduction}

This is the first of a series of papers devoted to a thorough analysis
of the class of gradient flows in metric spaces that can be
characterized by Evolution Variational Inequalities (EVI, in short).

Gradient flows govern a wide range of important evolution
problems. Perhaps the most popular and well-known theory, with
relevant applications to various classes of Partial Differential Equations, concerns
the evolution in a Hilbert space $X$ driven by a lower semicontinuous and convex (or
$\lambda$-convex) functional $\phi:X\to (-\infty,+\infty]$
with non empty proper domain $\Dom\phi:=\{w\in X:\phi(w)<\infty\}$.
The evolution can be described by a locally Lipschitz curve
$u:(0,+\infty)\to \Dom\phi$ solving the differential inclusion
\begin{equation}
  \label{eq:49}
  u'(t)\in -\partial\phi(u(t))\quad \text{for
    $\Leb1$-a.e.~$t>0$},\qquad \lim_{t\down0}u(t)=u_0\in \overline{\Dom\phi},
\end{equation}
where $\partial\phi$ denotes the Fr\'echet subdifferential of convex analysis.
Existence, uniqueness, well-posedness, approximation
and regularity properties of solutions to \eqref{eq:49}
have been deeply studied by the 
the pioneering papers of \textsc{Komura, Crandall-Pazy, Dorroh, Kato, Br\'ezis},
see e.g.~the book \cite[Chapters III,IV]{Brezis73} and the references therein.

In the Hilbert framework solutions to \eqref{eq:49} enjoy many 
important properties: 
they give rise to a continuous semigroup
$(\FlowName_t)_{t>0}$ of $\lambda$-contracting maps
$\FlowName_t:\overline{\Dom\phi}\to \Dom\phi$, in the sense that
$u(t):=\FlowName_t(u_0)$ is the unique solution to \eqref{eq:49} and
\begin{equation}
  \label{eq:52}
  \|\FlowName_t(u_0)-\FlowName_t(v_0)\|_X\le \rme^{-\lambda
    t}\|u_0-v_0\|_X\quad
  \text{for every }u_0,v_0\in \Dom\phi.
\end{equation}
Among the many important properties of $(\FlowName_t)_{t>0} $,
we recall that the energy map $t\mapsto
\phi(u(t))$ is absolutely continuous and
satisfies the energy-dissipation identity (here in a
differential form, where $\|\partial\phi(\cdot)\|_X$ denotes the
minimal norm of the elements in $\partial\phi(\cdot)$)
\begin{equation}
  \label{eq:38}
  \frac\d{\d t}\phi(u(t))=-\|u'(t)\|_X^2=-\|\partial\phi(u(t))\|_X^2\quad
  \text{for $\Leb 1$-a.e.~$t>0$};
\end{equation}
moreover every solution $u$ arises from the 
locally-uniform limit of the discrete approximations obtained by the Implicit Euler Method
\begin{subequations}
  \begin{gather}
      \label{eq:50}
    U^n_\tau=\mathrm J_\tau(U^{n-1}_\tau), \quad n\in \N,\quad
    U^0_\tau:=u_0,
    \intertext{where the resolvent map $\mathrm
      J_\tau:X\to X$
      is defined by}
    \label{eq:78}
    U=\mathrm J_\tau (V)\quad\Leftrightarrow\quad
    \frac{U-V}\tau\in
    -\partial\phi(U).
    % \frac{U^n_\tau-U^{n-1}_\tau}\tau\in
    % -\partial\phi(U^n_\tau)
  \end{gather}
\end{subequations}
In fact \eqref{eq:50} recursively defines a family of sequences
$(U^n_\tau)_{n\in \N}$ depending on a sufficiently small time step
$\tau>0$ and inducing a piecewise constant interpolant
\begin{equation}
  \label{eq:51}
  \Pc U\tau(t):=U^n_\tau=\mathrm J_\tau^n(u_0)\quad\text{whenever $t\in ((n-1)\tau,n\tau]$},
\end{equation}
converging to the continuous solution $u$ of \eqref{eq:49} as
$\tau\downarrow0$.

Under the initial impulse of \textsc{De Giorgi,
  Degiovanni, Marino, Tosques} \cite{DeGiorgi-Marino-Tosques80,Degiovanni-Marino-Tosques85,Marino-Saccon-Tosques89}, the abstract theory has been extended towards two main
directions: a relaxation of the convexity assumptions on $\phi$
(see e.g.~\cite{Rossi-Savare06,Mielke-Rossi-Savare13})
and a broadening of the structure of the ambient space, from
Hilbert to Banach spaces
(for the theory of doubly nonlinear evolution equations, see e.g.~\cite{Barbu76,Colli92})
or to more general metric and topological spaces \cite{DeGiorgi93}. 
It is remarkable that the original approach by De Giorgi and his
collaborators encompasses both these directions.

Here we focus on the second direction and we consider the metric side of
the theory, referring also to \cite{Daneri-Savare14,Santambrogio17}
for recent overviews.
There are (at least) three different
formulations of gradient flows in a metric space $(X,\DistName)$:
the first one, introduced by \cite{DeGiorgi-Marino-Tosques80},
got inspiration from \eqref{eq:38}, which is in fact an equivalent
characterization of \eqref{eq:49} in Hilbert spaces.
One can give a metric meaning to the (scalar) velocity of a curve $u:[0,\infty)\to X$ by the so
called \emph{metric derivative}
\begin{equation*}
  \label{eq:53}
  |\dot u|(t):=\lim_{h\to0}\frac{\Dist{u(t+h)}{u(t)}}{|h|},
\end{equation*}
and to the norm of the (minimal selection in the) subdifferential by
the \emph{metric slope}
\begin{equation*}
  \label{eq:54}
  |\partial\phi|(u):=\limsup_{v\to u}\frac{(\phi(u)-\phi(v))_+}{\Dist uv}.
\end{equation*}
A \emph{curve of maximal slope} is an absolutely continuous curve
$u:[0,+\infty)\to \AmbientSpace$ satisfying the energy-dissipation identity
\begin{equation}
  \label{eq:55}
    \frac\d{\d t}\phi(u(t))=-|\dot u|^2 (t)=-|\partial\phi|^2 (u(t))\quad
  \text{for $\Leb 1$-a.e.~$t>0$},
\end{equation}
which is the metric formulation of \eqref{eq:38}.

A second approach is related to a variational formulation to
\eqref{eq:78} (since the latter does not make sense in a pure metric setting), which can be considered as the first-order optimality
condition for the variational problem
\begin{equation}
  \label{eq:56}
  U\in \mathrm J_\tau(V)\quad\Leftrightarrow\quad
  U\in \argmin_{W\in \AmbientSpace}\frac1{2\tau}\DistSquare W{V}+\phi(W),
\end{equation}
whose minimizers define a multivalued map still denoted by $\mathrm
J_\tau$.
A recursive selection of $U^n_\tau$ among the minimizers $\mathrm J_\tau(U^{n-1})$ of \eqref{eq:56} 
yields a powerful algorithm to approximate gradient flows.
Pointwise limits up to subsequences
of the interpolants $\Pc U\tau$ defined by \eqref{eq:51} as $\tau\down0$ 
are denoted by $\GMM{\Syst;u_0}$ and called \emph{Generalized Minimizing Movements}, 
a particular important
case of a general framework introduced in \cite{DeGiorgi93}.
Let us remark that the
extensive application of \eqref{eq:56} in the area of Optimal Transport has been
independently initiated by the celebrated paper of
\textsc{Jordan-Kinderlehrer-Otto}
\cite{Jordan-Kinderlehrer-Otto98}; their approach has been so
influential, that the realization of \eqref{eq:56} in
Kantorovich-Rubinstein-Wasserstein spaces (see the examples below and
Section \ref{subsubsec:Wasserstein}) is commonly called \emph{JKO scheme.}

It is not difficult to prove existence of generalized minimizing
movements, under suitable compactness conditions
(see \cite{Ambrosio95,Ambrosio-Gigli-Savare08}); if moreover the slope
of $\phi$ is a lower semicontinuous upper gradient (see \cite[Definitions 1.2.1, 1.2.2]{Ambrosio-Gigli-Savare08}) then generalized
minimizing movements are also curves of maximal slope, according to
\eqref{eq:55}. 

Even if $\phi$ is geodesically $\lambda$-convex (the natural extension
of convexity in metric spaces), however, the evolution does not enjoy
all the nice semigroup properties of the Hilbertian case: it can be
easily checked even in finite-dimensional spaces, endowed with a non-quadratic norm \cite{Ohta-Sturm12}. For this to hold, we need to ask something more. Indeed, the most powerful (and demanding) notion of gradient flow can be
obtained by a metric formulation of \eqref{eq:49}, which may be
written as a system of \emph{evolution variational inequalities}, observing
that 
\begin{equation*}
  \label{eq:57}
  \begin{aligned}
    \frac 12\frac \d{\d t}\|u(t)-w\|_X^2&= \langle
    u'(t),u(t)-w\rangle_X\\&\le \phi(w)-\phi(u(t))-\frac
    \lambda2\|u(t)-w\|_X^2 \quad\text{for every }w\in \Dom\phi,
  \end{aligned}
\end{equation*}
thanks to the $\lambda$-convexity of $\phi$ and the
properties of the subdifferential.

In metric spaces one may similarly look for curves
$u:[0,\infty)\to\Dom\phi$ satisfying 
\begin{equation}
  \label{eq:58}
  \frac 12\frac \d{\d t}\DistSquare{u(t)}w
  +\frac\lambda2\DistSquare {u(t)}w\le \phi(w)-\phi(u(t))
  \quad\text{for every }w\in \Dom\phi;
\end{equation}
if a solution exists for every initial datum $u_0\in \Dom\phi$,
it gives rise to a $\lambda$-contracting semigroup
$(\FlowName_t)_{t\ge0}$ in $\Dom\phi$ satisfying the analogous
stability property of \eqref{eq:52}.
Let us recall that \eqref{eq:58}, which resemble previous formulations of
\cite{Benilan72,Brezis72,Baiocchi89} in Hilbert and Banach spaces,
has been introduced in the metric framework by
\cite{Ambrosio-Gigli-Savare08} and it 
is called the \EVIshort\lambda\ formulation of the gradient
flow driven by $\phi$.

In fact \eqref{eq:58} is the most restrictive definition of gradient flow
for $\lambda$-convex functionals.
Differently from the Hilbertian case, 
in arbitrary metric spaces $\AmbientSpace$ some ``Riemannian-like''
structure for $X$ should also be required
\cite{Ohta-Sturm12,VonRenesse-Tolle12}.
Among the most interesting examples of spaces and $\lambda$-convex
functionals
where an \EVIshort\lambda\ gradient flow exists we can quote:
\renewcommand{\labelitemi}{$\triangleright$}
\begin{itemize}\itemsep-2pt
\item Hilbert spaces,
\item complete and smooth Riemannian manifolds,
\item Hadamard non-positively curved (NPC) spaces, or more generally,
  metric spaces with an
  upper curvature bound \cite{Mayer98,Jost98,Ambrosio-Gigli-Savare08,Bacak14},
\item PC spaces, or more generally Alexandrov spaces with a uniform
  lower curvature bound
  \cite{Burago-Gromov-Perelman92,Burago-Burago-Ivanov01,Plaut02,Perelman-Petrunin,Petrunin07,Savare07,Ohta09,Muratori-SavareIII},
\item the Wasserstein space $(\ProbabilitiesTwo X,\WDistName)$, where $X$ is 
  a complete and smooth Riemannian manifold, a compact Alexandrov
  space or a Hilbert space
  \cite{Ambrosio-Gigli-Savare08,Savare07,Erbar10,Villani09,Ambrosio-Savare-Zambotti09,Ohta09,GigliKuwadaOhta10}.
\item RCD(K,$\infty$) metric measure spaces
  \cite{Ambrosio-Gigli-Savare14b,Erbar-Kuwada-Sturm15,Ambrosio-Mondino-Savare15},
\item $(\ProbabilitiesTwo X,\WDistName)$, where $X$ is an RCD(K,$\infty$)
  metric-measure space
  and $\phi$ is the relative entropy functional
  \cite{Sturm18,Muratori-SavareII},
\end{itemize}
and there is an intensive research to construct EVI gradient flows
in ad-hoc geometric settings
for reaction-diffusion equations
\cite{Kondratyev-Monsaingeon-Vorotnkikov16,Chizat-DiMarino17,
  Liero-Mielke-Savare18,Laschos-Mielke18} and systems
\cite{Mielke11,Glitzky-Mielke13,Liero-Mielke13},
nonlinear viscoelasticity \cite{Mielke-Ortner-Sengul14},
Markov chains \cite{Maas11,Erbar-Maas12,Mielke13},
jump processes \cite{Erbar14},
configuration
and Wiener spaces
\cite{Erbar-Huesmann15,Ambrosio-Erbar-Savare16};
we refer to Section \ref{subsec:examples} for a more detailed
discussion of the main classes of metric structures.

We think that all these examples and the developments of
the metric theory justify a systematic study of the
\EVIshort\lambda-formulation of gradient flows.
This is in fact the main contribution of our investigation.
In the present paper we will assume that an \EVIshort\lambda-flow
exists and we will deeply examine its main structural properties,
independently of the construction method.
In the following companion papers, we will study the natural stability
properties related to perturbation of the functional $\phi$
and of the distance $\DistName$
under suitable variational notions of convergence (as $\Gamma$,
Mosco, or Gromov-Hausdorff convergence) and
the generation results under the weakest assumptions on the metric and
on the functional. In all our analyses we will make a considerable
effort to avoid ad hoc hypotheses, in particular concerning
compactness, in order to cover also important infinite-dimensional examples.

\subsection*{Plan of the paper and main results}
Let us now quickly describe the structure and the main results of the present paper.

In the preliminary  \textbf{Section \ref{sec:preliminaries}} we will briefly
recall the main metric notions we will extensively deal with.
Since we want to avoid compactness assumptions, \emph{a particular attention is
devoted to approximate conditions}: length properties
(Definition \ref{def:length} and Lemma \ref{le:complete-length}),
a new relaxed formulation of the Moreau-Yosida regularization
inspired by the Ekeland variational principle \cite{Ekeland74}
(Section \ref{subsec:MYE}),
and an approximate notion of $\lambda$-convexity
(Definition \ref{def:app_lambda_convexity}).
Theorem
\ref{le:slope_by_Ekeland} provides a very useful \emph{approximation by
points with finite slope} and a crucial estimate, which lies at the core
of the relaxed version of the Minimizing Movement scheme
of Section \ref{sec:uniform_estimate}.
Theorem \ref{le:bfb} shows that \emph{approximately convex
(resp.~$\lambda$-convex) functions are linearly (resp.~quadratically) bounded from below}, a property
which is well known in Banach spaces. Both these results only depend
on the completeness of the sublevels of $\phi$.

\textbf{Section \ref{sec: gf}} contains the main structural properties
of solutions to the Evolution Variational Inequalities
\EVIshort\lambda;
first of all, in Section \ref{subsec:EVIdefinition}
we will consider various formulations of \eqref{eq:58},
showing their equivalence.
It is clear that classical formulations in terms of pointwise
derivatives are quite useful to obtain contraction and
regularity estimates, whereas integral or distributional
characterizations
are important when stability issues are involved.

Our first main result is \emph{Theorem \ref{thm:main1}},
which collects \emph{all the fundamental properties of solutions}
of the \EVIshort\lambda-formulation,
mainly inspired by the Hilbert framework.
Some results (as $\lambda$-contraction, regularization or asymptotic
behaviour) had previously been obtained in
\cite{Ambrosio-Gigli-Savare08}, occasionally using stronger structural
properties on $\phi$. Here we will show that
all the relevant properties are consequence of the Evolution
Variational Inequalities; in particular we will obtain
precise pointwise versions of the energy identity, refined
regularization estimates, and sharp asymptotic expansions, which will
play a crucial role in the analysis of the curves of maximal slopes
and of the minimizing movements in Sections \ref{sec:max-slope} and
\ref{sec:uniform_estimate}.
Section \ref{subsec:lconvexity} collects two basic properties
involving
\EVIshort\lambda-flows and $\lambda$-convexity:
first of all, we extend the result by \cite{Daneri-Savare08} showing
that
the existence of an \EVIshort\lambda-flow entails approximate
$\lambda$-convexity if $\Dom\phi$ is a length (also called intrinsic) subset.
Differently from the geodesic property, the length assumption
should not be considered restrictive, since a result of
\cite{Ambrosio-Erbar-Savare16}
shows that an \EVIshort\lambda-flow in $(X,\DistName)$
is also an \EVIshort\lambda-flow if we replace $\DistName$ by the
length distance $\lDistName$; Theorem \ref{thm:Daneri} then implies
that $\phi$ is always $\lambda$-convex in the modified intrinsic geometry.
We will also show that the constants $\lambda$ of the
$\EVIshort\lambda$-formulation and of the $\lambda$-convexity
coincide.

In the last two sections of the paper we compare
the notions of $\EVIshort\lambda$-flow, of curves of maximal slope and
of generalized Minimizing Movements. Our main assumption is that an
$\EVIshort\lambda$-flow exists; in this case, 
in \textbf{Section \ref{sec:max-slope}} we prove that any curve of
maximal slope, even in a weaker integral sense than \eqref{eq:55}, is a solution
to the \EVIshort\lambda formulation; in particular, we get also
uniqueness of curves of maximal slope.

In \textbf{Section \ref{sec:uniform_estimate}} we will study the
Minimizing Movement scheme, in a new and suitably relaxed form where
we use approximate minimizers obtained at each step via the Ekeland
variational principle. In this way, we can establish the existence of a discrete
approximating sequence without invoking compactness arguments,
and we will prove that every sequence of discrete Minimizing Movements
converges to the solution of the \EVIshort\lambda-flow.
As a byproduct of our analysis, we will also prove explicit error
estimates between discrete Minimizing Movements
and continuous solutions, assuming different regularity on the initial
data and allowing for a
further relaxation of the minimality condition:
when $u_0$ has finite slope, in the case of a convex functional with
no need for the Ekeland regularization, we will obtain
the uniform error estimate of order $1/2$
\begin{equation}
  \label{eq:59}
  \Dist{\Flow t{u_0}}{\mathrm J^n_{\tau}(u_0)}\le \frac {t}{\sqrt
    n}|\partial\phi|(u_0),\quad
  \tau=t/n,
\end{equation}
which reproduces in our metric setting (with a better constant) the celebrated Crandall-Liggett
estimates for contraction semigroups in Banach spaces
\cite{Crandall-Liggett71}, previously obtained in specific metric
settings under much stronger assumptions on $\DistName$
\cite{Mayer98,Ambrosio-Gigli-Savare08,Clement-Desch10,Craig16,Bacak14},
the latter implying suitable contraction properties of $\mathrm J_\tau$.

These estimates are completely new; their strength relies on the fact
that they do not depend on any upper or lower curvature bound on the
distance $\DistName$, but
only on the existence of an $\EVIshort\lambda$-flow.
In particular, they can be applied to the setting of
$\mathrm{RCD}(K,\infty)$ metric spaces or to Wasserstein spaces,
showing that the JKO-Minimizing Movement Scheme is always at least of
order $1/2$ whenever it is applied to an $\EVIshort\lambda$-gradient flow.
Apart from their intrinsic interest, these estimates will also be
crucial in the study of the convergence of $\EVIshort\lambda$-flows, which will be investigated in the next paper \cite{Muratori-SavareII}. For these reasons, we tried to reach the greatest level of generality. 

\subsection*{List of main notations}

\halign{$#$\hfil\ &#\hfil
\cr
(\AmbientSpace,\DistName)&the reference metric space\cr
\phi&a l.s.c.~functional on $\AmbientSpace$ with values in $(-\infty,+\infty]$\cr
\AC^p(I;\AmbientSpace)&$p$-absolutely continuous curves $\sfx:I\to X$, Def.~\ref{def:ACcurves}\cr
|\dot \sfx|,\ \mathrm{Length}[\sfx]&metric velocity and length of an
$\AC$ curve $\sfx$, \eqref{eq:cap1:65} and \eqref{eq:9} \cr
\lDistName,\lDDistName D&length distance induced by $\DistName$ (in a
subset $D$), \eqref{eq:41}\cr
\GeoSpace D,\ \GeoCon D{x_0}{x_1}&geodesics in a subset $D$, Def.~\ref{geo-geo}\cr
\Dom\phi& domain of a proper functional $\phi$, \eqref{eq:61}\cr
\MetricSlope\phi x,\ \GSlope\lambda\phi x&metric and global slopes of
$\phi$, Def.~\ref{def:metric_slope}\cr
\Resolvent\taueta x&Moreau-Yosida-Ekeland resolvent, Def.~\ref{def:MYE-resolvent}\cr
\phi'(\sfx_0,\sfx)&directional derivative of $\phi$ along a geodesic
$\sfx$, Def.~\ref{def:metric_slope2}\cr
\X=\System&a metric-functional system, \eqref{eq:system}\cr
\EVIshort\lambda&Evolution Variational Inequalities characterizing the
flow, Def.~\ref{def:GFlow}\cr
\mathsf E_\lambda(t)& the primitive of the function $ \rme^{\lambda t} $ s.t.~$\mathsf E_\lambda(0)=0$,
\eqref{eq:cap2:11}\cr
(U^n_\taueta)_{n\in \N}&a seq.~in $X$ generated by the 
Minimizing Movement algorithm, Def.~\ref{def:MMS}\cr
\taueta &step size and Ekeland relaxation parameter of a
Min.~Mov., Def.~\ref{def:MMS}\cr
\Pc U\taueta(t)& piecewise-constant interp.~of a discrete
Minimizing Movement, \eqref{eq:cap1:66}\cr
\contMM{\Syst;u_0}&Minimizing
Movements starting from $u_0$, Def.~\ref{def:MM}\cr
 \GMM{\Syst;u_0}&Generalized Minimizing
Movements starting from $u_0$, Def.~\ref{def:MM}\cr
}
\vspace{6pt}

\section{Preliminaries}
\label{sec:preliminaries}
Let us first briefly recall some basic definitions and tools we will extensively use in the forthcoming sections, referring to \cite{Ambrosio-Gigli-Savare08} for a more detailed introduction to the whole subject.
Throughout the present paper we will refer to a metric space
$(\AmbientSpace,\DistName)$.
We will often use the symbol $D$ to denote a distinguished subset of
$\AmbientSpace$, which inherits the distance $\DistName$ from $\AmbientSpace$.

\subsection{Absolutely continuous curves, length subsets and
  geodesics}
\label{subsec:curves}
%% which for the moment we do not assume to be complete.
\begin{definition}[Absolutely continuous curves and metric derivative]
  \label{def:ACcurves}
  Let $I\subset \R$ be an interval and $p\in [1,+\infty]$. A curve 
  $\sfx:I\to X$ belongs to $\AC^p(I;\AmbientSpace)$
  if there exists $m\in L^p(I)$ such that
  \begin{equation}
    \label{eq:cap1:1}
    \Dist{\sfx_s}{\sfx_t}\le \int_s^t m(r)\,\d r\quad
    \forevery s,t\in I \ \text{with} \ s\le t.
  \end{equation}
  The \emph{metric derivative} of $\sfx$ is defined, where it exists, as
  \begin{equation}
    \label{eq:cap1:65}
    \MetricDerivative \sfx( t):=
    \lim_{h\to0}\frac{\Dist{\sfx_{t+h}}{\sfx_t}}{|h|}.
  \end{equation}
\end{definition}
The proof of the following result can be found, e.g., in \cite[Theorem
1.1.2]{Ambrosio-Gigli-Savare08}. 
% A key point lies in the fact that, since the image of $ \sfx  $ is separable, there holds 
% \begin{equation}\label{eq: img-sep}
% \Dist{\sfx_s}{\sfx_t} = \sup_{n \in \N} \left| \Dist{\sfx_s}{v_n} - \Dist{\sfx_t}{v_n}  \right| \quad \forall s,t \in I
% \end{equation}
% for some sequence $ v_n  $ dense in $ \sfx(I) $.

\begin{theorem}[Absolutely continuous curves and metric derivative]
  \label{thm:metric_derivative}
  If $\sfx\in \AC^p(I;\AmbientSpace)$, then $\sfx$ is uniformly continuous, its metric derivative exists at $\Leb 1$-a.e.\ $t\in I$, belongs to $L^p(I)$ and provides the minimal function $m$ satisfying
  \eqref{eq:cap1:1}, i.e.\ $\MetricDerivative \sfx$ complies with \eqref{eq:cap1:1} and
  every function $m$ as in \eqref{eq:cap1:1} satisfies $m(t)\ge
  \MetricDerivative\sfx(t)$ for $\Leb 1$-a.e.\ $t\in I$.
\end{theorem}
Still by means of similar techniques to those used in the proof of \cite[Theorem 1.1.2]{Ambrosio-Gigli-Savare08}, 
it is not difficult to show that a curve belongs to $\AC^1(I;\AmbientSpace)$ (resp.\ $\AC^\infty(I;\AmbientSpace)$)
if and only if it is \emph{absolutely continuous} (resp.\
\emph{Lipschitz continuous}). Hence from now on $ \AC^1 = \AC $. The length of a curve $\sfx\in \AC(I;\AmbientSpace)$ is
\begin{equation}
  \label{eq:9}
  \operatorname{Length}[\sfx]:=\int_I \MetricDerivative\sfx(t)\,\d t.
\end{equation}
A well-known reparametrization result (see e.g.~\cite[Lemma
1.1.4]{Ambrosio-Gigli-Savare08}) shows that
for every $\sfx\in \AC([a,b];\AmbientSpace)$ with length $L$, there exist
unique maps $\sfy\in \AC^\infty([0,1];\AmbientSpace)$
and $\sigma:[a,b]\to[0,1]$ continuous and (weakly) increasing such
that
\begin{equation}
  \label{eq:43}
  \sfx=\sfy\circ\sigma,\quad
  \MetricDerivative \sfy=L\quad\text{$\Leb 1$-a.e.~in $[0,1]$}.
\end{equation}
Given a subset $D\subset X$,
we can then define the \emph{length distance} induced by $\DistName$
in $D$:
\begin{equation}
  \label{eq:41}
  \lDDist D{x_0}{x_1}:=\inf\Big\{  \operatorname{Length}[\sfx]:
  \sfx\in \AC([0,1];D),\ \sfx(i)=x_i,\ i=0,1\Big\},
\end{equation}
and we will simply use the symbol $\lDistName$ when $D=\AmbientSpace$.

If $D$ is \emph{Lipschitz connected}, i.e.~each 
couple of points $x_0,x_1\in D$ can be connected by a curve
$\sfx\in \AC([0,1];D)$, then $(D,\lDDistName D)$
is a metric space. In general, we adopt the usual convention to set $\lDDist D{x_0}{x_1}=+\infty$ if
there are no absolutely continuous curves connecting $x_0$ to $x_1$ in $D$.
In this case 
$\lDDistName D:D\times D\to[0,+\infty]$
is an \emph{extended} distance on $D$, i.e.~it satisfies all the axioms of a
distance function, possibly assuming the value $+\infty$.
By partitioning $D$ in the equivalence classes with
respect to the relation $\sim_{D,\ell}$ defined by 
\begin{equation}
  x\sim_{D,\ell} y\quad \Leftrightarrow\quad
  \lDDist Dxy<\infty,\label{eq:lequivalence}
\end{equation}
each class
\begin{equation*}
  \label{eq:42}
  \DClass D{\bar x}:=\Big\{x\in D:
  \lDDist Dx{\bar x}<\infty\Big\},\quad \bar x\in D,
\end{equation*}
endowed with the distance $\lDDistName D$ becomes a metric space in the usual sense
and it is closed and open  in $(D,\lDDistName D)$. 
Since $\DistName\le \lDDistName D$, the topology induced by
$\lDDistName D$
is stronger than the original topology induced by $\DistName$;
moreover, if
$(D,\DistName)$ is complete, then $(D,\lDDistName D)$ is
complete as well.

\nc
% Recall that a curve $\sfx:I\to\AmbientSpace$
% is \emph{strongly $\Leb 1$-measurable}
% if it is $\Leb 1$ measurable (with respect to
% the Borel subsets of $\AmbientSpace$) and
% its range is essentially separable, i.e.\ there exists a
% $\Leb 1$-negligible subset $\mathscr N\subset I$
% such that $\sfx(I\setminus \mathscr N)$ is separable.
% When $p\in (1,+\infty)$ we have the following equivalent characterization
% of curves in $\AC^p(I;\AmbientSpace)$
% \cite[Lemma 1]{Lisini07}.
% \begin{proposition}
%   \label{prop:equivalent_\AC}
%   Suppose that $\AmbientSpace$ is complete,
%   and let $\sfx:I\to\AmbientSpace$ be a strongly $\Leb 1$-measurable curve.
%   If 
%   \begin{equation}
%     \label{eq:cap1:88}
%     \limsup_{h\down0}\int_{a}^{b}
%     \frac{\DistPower p{\sfx_{t+h}}{\sfx_t}}{h^p}\,\d t\le C
%     <+\infty\quad
%     \forevery\, (a,b)\Subset I,
%   \end{equation}
%   then there exists a unique curve $\tilde\sfx \in
%   \AC^p(I;\AmbientSpace)$
%   such that $\tilde\sfx_t=\sfx_t$ for $\Leb 1$-a.e.\ $t\in I$.
% \end{proposition}
\begin{definition}[Geodesics and geodesic subsets]\label{geo-geo}
  A \emph{constant speed, minimal geodesic} (in short,
  \emph{geodesic}) in $D\subset \AmbientSpace$ is
  a (Lipschitz) curve $\Geodesic:[0,1]\to D$ such
  that
  \begin{equation*}
    \label{eq:cap1:2}
    \frac{\Dist{\Geodesic_s}{\Geodesic_t}}{|s-t|}
    =\Dist{\Gdc_0}{\Gdc_1}=:|\dot\Gdc|\quad
    \text{for every } 0\le s<t\le 1.
  \end{equation*}
  We denote by $\GeoSpace D\subset \AC^{\infty}([0,1];X)$
  the collection of all of the geodesics in $D$ and by $\GeoCon D{x_0}{x_1}$ the
  (possibly empty) collection of the geodesics in $D$ satisfying
  $\Geodesic_i=x_i$, $i=0,1$. 
  If $\GeoCon D{x_0}{x_1} \not\equiv \emptyset $ for every $ x_0,x_1
  \in D $ we say that $ D $ is a \emph{geodesic subset}. 
% A subset $C\subset\AmbientSpace$ is \emph{geodesic (or strictly intrinsic)}
%   if for every couple $x_0,x_1\in C$ the set
%   $\GeoCon\AmbientSpace{x_0}{x_1}$ is not
%   empty.
\end{definition} 
So geodesics are absolutely continuous curves along
which the space swept between any two points ($ \sfx_0 $ and $
\sfx_1 $ are enough actually) is equal to their distance. By the arc-length reparametrization \eqref{eq:43}, it is not restrictive to assume
that they have constant speed. In particular, if $\AmbientSpace$ is a
geodesic space then $\DistName=\lDistName$ and the $ \inf $ in \eqref{eq:41} is attained.
%\Comment{aggiunta giusto per chi (come me...) non ha troppa familiarita' con le geodetiche}
% We also recall the notion of \emph{intrinsic distance} associated to $\DistName$:
% \begin{definition}[Intrinsic distance]
%   \label{def:intrinsic-distance}
%   For every couple of point $x_0,x_1\in \AmbientSpace$ we set
%   \begin{equation}
%     \label{eq:25}
%     \IntrinsicDist{x_0}{x_1}:=\inf\Big\{\int_0^1|\dot \sfx_t|\,\dt:\sfx\in AC^1(0,1;\AmbientSpace),\
%     \sfx_i=x_i,\ i=0,1\Big\}.
%   \end{equation}
% \end{definition}
% It is easy to check that $\IntrinsicDistName:\AmbientSpace\times\AmbientSpace\to[0,+\infty]$
% (which can also take the value $+\infty$)
% satisfies the axioms of (pseudo-)distances in metric space. $(X,\IntrinsicDistName)$ is a complete (pesudo-)metric space
% with
% \begin{equation}
%   \label{eq:26}
%   \Dist{x_0}{x_1}\le \IntrinsicDist{x_0}{x_1}\quad\text{for every }x_0,x_1\in \AmbientSpace.
% \end{equation}
% $\AmbientSpace$ is called \emph{intrinsic} if $\IntrinsicDistName\equiv\DistName$.
% \begin{remark}
%   \label{rem:intrinsic-AC}
%   Any curve $\sfx\in AC(a,b;\AmbientSpace)$ is also absolutely continuous with respect to
%   the intrinsic distance $\IntrinsicDistName$ (the converse implication is obvious) and
%   \begin{equation}
%     \label{eq:27}
%     |\dot \sfx_t|=\lim_{h\to0}\frac{\IntrinsicDist{\sfx_{t+h}}{\sfx_t}}{|h|}\quad\text{for $\Leb 1$-a.e.\ $t\in (a,b)$.}
%   \end{equation}
%   In particular, the collection of geodesics of $\DistName$ and of $\IntrinsicDistName$ is the same.
% \end{remark}
% \GGG

Since in general metric spaces the existence of geodesics is not for
granted, 
one can consider weaker notions; the
first one is the length (or intrinsic) property.
\begin{definition}[Length (or intrinsic) property]
  \label{def:length} 
    $D$ is a length (or intrinsic) subset of $X$ 
    if for every $x_0,x_1\in D$ and every $d>\Dist{x_0}{x_1}$
    there exists a curve $\sfx\in \AC([0,1];D)$ connecting $x_0$ to
  $x_1$ such that
  $\operatorname{length}[\sfx]\le d$.  
\end{definition}
\noindent
Note that $D$ is intrinsic if and only if $\lDDistName D=\lDistName=\DistName$ on $D\times D$. 

\medskip

A second, even weaker property, is related to
the existence of
$\eps$-approximate intermediate points. \nc
\begin{definition}[Approximate length subsets]
  \label{def:applength}
  Let $x_0,x_1\in X$, $\vartheta\in (0,1)$, $\eps\in (0,1)$. We say that $x$ is a $(\vartheta,\eps)$-intermediate point ($\eps$-midpoint if
  $\vartheta=1/2$) between $x_0$ and $x_1$ if 
  \begin{equation}
    \label{eq:89} 
    \frac{\DistSquare{x_0}{x}}\vartheta+
    \frac{\DistSquare{x}{x_1}}{1-\vartheta}\le \DistSquare{x_0}{x_1}(1+\eps^2\vartheta(1-\vartheta)).
  \end{equation}
  We say that $D\subset X$ is an \emph{approximate length subset} if for every $x_0,x_1\in
  D$ and every $\eps \in (0,1) $ there exists an $\eps$-midpoint $x\in D$
  between $x_0$ and $x_1$. 
\end{definition}

Note that in the case of $\eps$-midpoints \eqref{eq:89} reads
\begin{equation*}
  \label{eq:91}
  \DistSquare{x_0}{x}
  +\DistSquare{x}{x_1}\le \frac 12\DistSquare{x_0}{x_1}(1+ \eps^2/4).
\end{equation*}
The above definition comes from the fact that geodesic points $
\Geodesic_\vartheta $ can be characterized as minimizers of the
functional in the l.h.s.~of \eqref{eq:89} for which the minimum coincides with $\DistSquare{x_0}{x_1}$, see
\cite[Lemma 2.4.8]{Burago-Burago-Ivanov01};
by means of $ \varepsilon $ we admit an arbitrarily small error with respect to exact minima. It is immediate to check that if 
  \begin{equation}
    \label{eq:88}
    \Dist{x_0}{x}\le \vartheta \Dist{x_0}{x_1}(1+\delta),\quad
    \Dist{x}{x_1}\le (1-\vartheta)
    \Dist{x_0}{x_1}(1+\delta)
  \end{equation}
with (for instance) $
    \delta\le \tfrac13 {\eps^2\vartheta(1-\vartheta)} $, then \eqref{eq:89} holds. On the other hand, recalling the elementary identity
\begin{displaymath}
  (\vartheta a+(1-\vartheta)b)^2=\vartheta a^2+(1-\vartheta)b^2-\vartheta(1-\vartheta)(a-b)^2,
\end{displaymath}
$ \ell :=\Dist{x_0}{x}+\Dist{x}{x_1}$ satisfies 
\begin{displaymath}
  \DistSquare{x_0}{x_1}\le 
  \ell^2=\frac{\DistSquare{x_0}{x}}\vartheta+\frac{\DistSquare{x}{x_1}}{1-\vartheta}-
  \vartheta(1-\vartheta)\left(\frac{\Dist{x_0}{x}}\vartheta-\frac{\Dist{x}{x_1}}{1-\vartheta}\right)^2 ,
\end{displaymath}
so that if \eqref{eq:89} holds we get
\begin{equation}\label{eq:89bis}
  \left|\frac{\Dist{x_0}{x}}\vartheta-\frac{\Dist{x}{x_1}}{1-\vartheta}\right|\le 
  \Dist{x_0}{x_1}\eps \qquad \text{and} \qquad \ell \le \Dist{x_0}{x_1} \sqrt{1+\eps^2 \vartheta(1-\vartheta) } ;
\end{equation}
from \eqref{eq:89bis} it is not difficult to deduce that \eqref{eq:88} is satisfied e.g.~with $\delta=\eps$. 

    We collect in the following Lemma a list of useful properties;
    to our purposes, let us denote by
        $\mathbb D:=\{k2^{-n}: \, k,n\in \N,\ 0\le k\le 2^n\}$ the set
    of all dyadic points in $[0,1]$.

\begin{lemma}
  \label{le:complete-length}
  \ 
  \begin{enumerate}[\rm L1:]
  \item If $D\subset X$ is a length subset then it is also an
    approximate length subset.
  \item \label{L2}
    $ D \subset X $ is an approximate length subset if and only if
    $\overline D$ is.
  \item \label{L1}
    If $D$ is an approximate length
    subset of $\AmbientSpace$ then for every $x_0,x_1\in D$ and every $L>\Dist{x_0}{x_1}$ there
    exists $\{\sfx_\vartheta\}_{\vartheta\in \mathbb D}\subset D$ such
    that $ \sfx_0 = x_0 $, $ \sfx_1=x_1 $ and
    \begin{equation*}
      \label{eq:92}
      \Dist{\sfx_{\vartheta'}}{\sfx_{\vartheta''}}\le
      L|\vartheta'-\vartheta''| \quad \forevery \vartheta',\vartheta''
      \in \mathbb D.
      % , \qquad 
      % L := \Dist{x_0}{x_1}(1+\delta) .
    \end{equation*}
    In particular, for every $\vartheta,\eps\in (0,1)$ there is
    a $(\vartheta,\eps)$-intermediate point $x\in D$ between $x_0$ and $x_1$.
  \item
    If $D$ is a complete and approximate length subset of
    $\AmbientSpace$ then it is a length subset of $\AmbientSpace$.
    % \item $D$ is an approximate length subset if and only if for all
    %   $ x,y \in D $ there holds
    %   \begin{equation*}
    %     \label{eq:93}
    %     \begin{aligned}
    %       \Dist{x}{y}=\sup_{\eps>0} \,
    %       \inf\Big\{&\sum_{k=1}^N\Dist{x_k}{x_{k-1}} : \ N \in \N ,
    %       \ \{ x_k \} \subset D, \\& \ x_0=x,\ x_N=y,\ \max_{k} \,
    %       \Dist{x_k}{x_{k-1}}\le \eps\Big\} .
    %     \end{aligned}
    %   \end{equation*}
  \end{enumerate}
  \end{lemma}
  \begin{proof}
    L1 is a consequence of the discussion
    following Definition \ref{def:applength}.
    The density of $D$ in $\bar D$
    and the triangle inequality yield L2.

    L3 can be proved by adapting the arguments of
    \cite[Section 2.4.4]{Burago-Burago-Ivanov01}: we briefly
    sketch the main steps.

    Let $\delta>0$ be such that $L=\Dist{x_0}{x_1}(1+\delta)$ and
    let us fix a sequence $(\eps_n)_{n\in \N}\subset (0,1)$
    satisfying $\sum_{n=0}^\infty\eps_n\le\log(1+\delta)$. We can parametrize all the points in $\mathbb D$ by two integers
    $n\in \N$ and $h\in [0,2^{n}]$ by the surjective (but not
    injective) map
    $(h,n)\mapsto \theta_{h,n}=h2^{-n}$.
    We call $\D_n:=\{\theta_{h,n}: h\in [0,2^{n}]\}$.
    It is clear that
    for any odd integer $h=2k+1$
    $\theta_{h,n+1}$ is
    the midpoint between $\theta_{k,n}$ and $\theta_{k+1,n}$,
    belonging to $\D_n$,
    whereas for an even integer $h=2k$, the point $\theta_{h,n+1}$ belongs to $\D_n$.

    We can construct a map $\vartheta\mapsto \sfx(\vartheta) = \sfx_\vartheta $ in $\D$
    satisfying
    \begin{equation}
      \label{eq:44}
      \Dist{\sfx(\vartheta)}{\sfx(\vartheta')}\le
      \Dist{x_0}{x_1}|\vartheta-\vartheta'|\exp\Big(\sum_{m=0}^n\eps_m\Big)\quad
      \text{for every }\vartheta,\vartheta'\in \D_n,
    \end{equation}
    by induction with respect to $n$:
    note that \eqref{eq:44} surely holds, by the triangle inequality, if
    \begin{displaymath}
      \Dist{\sfx((k+1)2^{-n})}{\sfx(k2^{-n})}\le
      \frac1{2^n}\Dist{x_0}{x_1}\exp\Big(\sum_{m=0}^n\eps_m\Big)\quad
      \text{for every integer }k\in [0,2^n].
    \end{displaymath}
    The case $n=0$ is trivial; assuming that
    $\sfx$ has already been defined on $\D_n$
    we can extend it to every point $\theta_{2k+1,n+1}\in
    \D_{n+1}\setminus \D_n$ by
    choosing an $\eps_{n+1}$-midpoint between
    $\sfx(\theta_{k,n})$ and $\sfx(\theta_{k+1,n})$; it is clear that (recall \eqref{eq:88})
    \begin{align*}
      \Dist{\sfx(\theta_{2k+1,n+1})}{\sfx(\theta_{2k,n+1})}&\le
      \frac 12\Dist{\sfx(\theta_{k,n})}{\sfx(\theta_{k+1,n})}(1+\eps_{n+1})
                                                             \\&\le 
      \frac1{2^{n+1}}\Dist{x_0}{x_1}
      \exp\Big(\sum_{m=0}^n\eps_m\Big)\exp(\eps_{n+1})
      \\&=
      \frac1{2^{n+1}}\Dist{x_0}{x_1}
      \exp\Big(\sum_{m=0}^{n+1}\eps_m\Big) ,
    \end{align*}
    and an analogous estimate holds for 
    $\Dist{\sfx(\theta_{2k-1,n+1})}{\sfx(\theta_{2k,n+1})}$. 
    % \begin{displaymath}
    %   \sum_{k=2h+1}^{2h'}\Dist{\sfx(\theta_{k,n+1})}{\sfx(\theta_{k-1,n+1})}
    %   \le
    %   \sum_{j=h+1}^{h'}\Dist{\sfx(\theta_{j,n})}{\sfx(\theta_{j-1,n})}
    %   (1+\delta 2^{-(n+1)})
    %   \le (1+-2^{-n}\delta\Dist{x_0}{x_1})
    %   (1+\delta 2^{-(n+1)})|\theta-\vartheta'|
    % \end{displaymath}
    %
    %
    % $\vartheta\in \mathbb D\setminus \{0,1\}$ can be written in a
    % unique way as $(2h+1)/2^n$ of some integer $h\in [0,2^{n-1})$.
    % Suppose that In particular, for every $ x_0,x_1 \in D $,
    % $ \vartheta \in (0,1) $ and $ \varepsilon \in (0,1) $ there exists
    % a $ (\vartheta,\eps) $-intermediate point $ x \in D $ between
    % $ x_0 $ and $x_1$.  Moreover, if $ D $ is a \emph{complete} subset
    % of $X$ then \eqref{eq:92} yields the existence of a Lipschitz
    % curve $ x_\vartheta: [0,1] \to D $ satisfying \eqref{eq:92} for
    % all $ \vartheta \in [0,1] $.
    
    L4 follows immediately from L3 and the completeness of $D$.
  \end{proof}
  
In analogy with \eqref{eq:41}, it is not difficult to show (for instance by using L\ref{L1}) that $D$ is an approximate length subset if and only if for all $ x,y \in D $ and $ \eps>0 $ there holds 
    \begin{equation}
      \label{eq:93}
      %\begin{aligned}
        \Dist{x}{y}=
        \inf\left\{ \sum_{k=1}^N\Dist{x_k}{x_{k-1}} : \ N \in \N , \ \{
          x_k \} \subset D,        
          \ x_0=x,\ x_N=y,\ \max_{k} \,
          \Dist{x_k}{x_{k-1}}\le \eps\right\} .
      %\end{aligned}
    \end{equation}
  
\subsection{Moreau-Yosida regularizations, slopes and Ekeland's variational principle}
\label{subsec:MYE}
On $X$ we will be considering \emph{proper} functionals $\phi:X \to (-\infty,+\infty]$, where  
\begin{equation}
\Dom\phi:=
\big\{x \in X: \, \phi(x) < + \infty\big\} \label{eq:61}
\end{equation}
denotes the (non-empty) domain of $\phi$. We say that $\phi$ is \emph{quadratically bounded from below} if
there exist $ o \in \AmbientSpace$, $\phi_o,\kappa_o\in \R$ such that
\begin{equation} 
  \label{eq:cap1:13}
  \phi(x)+\frac {\kappa_o}{2}\DistSquare x{\ooint }\ge \phi_o\quad
  \text{for every } x\in \AmbientSpace.
\end{equation}
Similarly, we say that $ \phi $ is \emph{linearly bounded from below} if there exist $ o \in \AmbientSpace$, $\phi_o,\ell_o\in \R$ such that
\begin{equation} 
  \label{eq:cap1:13-lin}
  \phi(x)+ {\ell_o}\Dist x{o}\ge \phi_o\quad
  \text{for every } x\in \AmbientSpace ,
\end{equation}
which in particular implies that it is quadratically bounded from below for all $ \kappa_o>0 $.

The (quadratic) 
\emph{Moreau-Yosida} regularizations of $\phi$ (we refer e.g.~to \cite[Section 3.1]{Ambrosio-Gigli-Savare08} or \cite[Chapter 9]{DalMaso93}) are the functionals $\phi_\tau:\AmbientSpace\to\R$ defined by 
\begin{equation}
  \label{eq:1}
  \phi_\tau(x):=\inf_{y\in \AmbientSpace} \phi(y) + \frac 1{2\tau}\DistSquare
  yx \quad \forevery x \in \AmbientSpace,
\end{equation}
and we set
\begin{equation}
    \label{eq:1-bis}
  \tau_o:=\sup\Big\{\tau>0 : \, \Dom{\phi_\tau}\neq\emptyset\Big\}.
\end{equation}
Note that $\tau_o>0$ if and only if $\phi$ is quadratically bounded
from below and $\Dom{\phi_\tau}=X$ for every $\tau\in (0,\tau_o)$.
If $\phi$ satisfies \eqref{eq:cap1:13} then 
\begin{displaymath}
  \tau_o\ge \kappa_o^{-1}\quad\text{if }\kappa_o>0;\qquad
  \tau_o=+\infty \quad \text{if }\kappa_o\le 0.
\end{displaymath}
Because we will mainly deal with functionals that are quadratically
bounded from below (see for instance Theorem \ref{le:bfb} and Theorem 
\ref{thm:main1} below), the regularizations given by \eqref{eq:1} come
naturally into play; they are also strictly related to the Minimizing Movement approach to
gradient flows, see the Introduction or Section \ref{sec:uniform_estimate}. However, it is
also possible to deal with more general regularizations: this will be addressed in detail in \cite{Muratori-SavareII}.

% and which is bounded from below at least in a neighborhood
% of some point:
% \begin{equation}
%   \label{eq:cap1:56}
%   \exists\, u_o\in \Dom\phi ,\ r_o>0:\quad
%   \inf \big\{\phi(v):v\in \AmbientSpace,\ \Dist v{u_o}\le r_o\big\}>-\infty.
% \end{equation}
% We are assuming that $\phi$ is $\lambda$-convex.
\begin{definition}[Metric slopes]
  \label{def:metric_slope}
  The \emph{metric slope} of $\phi$ at $x\in \AmbientSpace$ is
  \begin{equation*}
    \label{eq:cap1:71}
    \MetricSlope\phi x:=
    \begin{cases}
      +\infty&\text{if }x\not\in \Dom\phi,\\
      0&\text{if $x\in \Dom\phi$ is isolated,}\\
      \displaystyle\limsup\limits_{y\to x}\frac{\left(\phi(x)-\phi(y)\right)^+}
    {\Dist xy}&\text{otherwise.}
     \end{cases}
   \end{equation*}
  As usual, we set 
  $\DomainSlope{\phi}:=\big\{x \in \AmbientSpace \! : 
    \MetricSlope\phi x<+\infty\big\}$.
  For $\lambda\in \R$ and $x\in \AmbientSpace$ we then introduce the global $\lambda$-slope
  \begin{equation*}
    \label{eq:cap1:10}
    \GSlope\lambda\phi x:=\sup_{y\neq x}\frac{\left(
      \phi(x)-\phi(y)+\tfrac \lambda2\DistSquare xy\right)^+}{\Dist
      xy} ,
  \end{equation*}
understood to be $ +\infty $ if $ x \not\in \Dom\phi $. Similarly, we
set $\Dom{\GSlopeE\lambda\phi}:=\big\{x\in \AmbientSpace\!:
\GSlope\lambda\phi x<+\infty\big\}$. 

Finally, we will denote by $\lPlainSlope\phi$ the metric slope of
$\phi$ evaluated w.r.t.~the length distance $\lDistName$
\eqref{eq:41}, i.e.~in the extended metric space $(\AmbientSpace,\lDistName)$.
\end{definition} 

Note that for every $S\in [0,+\infty)$ and $x\in \AmbientSpace$ there holds 
\begin{equation}
  \label{eq:cap1:83}
  \MetricSlope\phi x\le S\quad
  \Leftrightarrow\quad
  x\in \Dom\phi,\quad
  \phi(y)\ge \phi(x)-S\Dist yx+o(\Dist yx)\quad\text{as }\Dist yx\to0,
\end{equation} 
whereas 
\begin{equation}
  \label{eq:cap1:89}
  \GSlope\lambda\phi x\le S\quad
  \Leftrightarrow\quad
  \phi(y)\ge \phi(x)-S\Dist yx+\frac\lambda 2\DistSquare yx
  \quad\forevery y\in \AmbientSpace.
\end{equation} 
In particular,
\begin{equation} 
  \label{eq:cap1:12}
  \MetricSlope\phi x\le \GSlope\lambda\phi x\quad
  \forevery x \in \AmbientSpace , \ \lambda\in \R.
\end{equation} 
We will mostly deal with proper functionals $ \phi $ that are \emph{lower semicontinuous} (l.s.c.\ for short): under such an assumption, it is easy to check that the global $\lambda$-slope is also lower semicontinuous. Contrarily, the metric slope is in general not lower semicontinuous, though it can be seen as a pointwise limit of lower semicontinuous functionals: 
\begin{equation}\label{eq: loc-slope}
  \MetricSlope \phi x = \lim_{n\to\infty} \sup\limits_{y \in \Dom\phi
    \setminus \{ x \} \atop \, \Dist x y < \frac1n }\frac{\left(\phi(x)-\phi(y)\right)^+}{\Dist xy} ,
\end{equation}
where the supremum is understood to be eventually $0$ if $x$ is
isolated in $\Dom\phi$ and $ +\infty $ if $ x \not \in \Dom \phi $. 
%\Comment{ho aggiunto questo fatto, puo' essere utile per questioni di misurabilita' ad esempio}

It turns out that the metric slope is always a \emph{weak upper
  gradient} for $\phi$, whereas the global $\lambda$-slope is a
\emph{strong upper gradient} for $\phi$ provided the latter is lower
semicontinuous: see \cite[Definitions 1.2.1, 1.2.2 and Theorem
1.2.5]{Ambrosio-Gigli-Savare08} for more details. 
%\Gsout{These are very important properties, of which however we shall not make an explicit use here.}
% \begin{definition}[\Reg\lambda functionals]
%   \label{def:Reglambda}
%   We say that $\phi$ is \Reg\lambda if
%   \begin{equation}
%     \label{eq:15}
%     \MetricSlope\phi x= \GSlope\lambda\phi x\quad
%     \forevery\, x\in \Dom\phi .
%   \end{equation}
% \end{definition}
%
%
% It is easy to check that if there exists at least one point
% $\ooint $ such that $\GSlope\lambda\phi{\ooint }=S<+\infty$ then
% $\phi$ is quadratically bounded from below,
% since
% \begin{equation}
%   \label{eq:cap1:14}
%   \phi(x)\ge \Big(\phi(\ooint )-\frac
%   {S^2}{2\eps}\Big)+\frac{\lambda-\eps}2
%   \DistSquare x{\ooint }\quad
%   \forevery\, x\in \AmbientSpace,\quad\eps>0,
% \end{equation}
% i.e.\ it satisfies \eqref{eq:cap1:13} with $\phi_o:= \frac
% {S^2}{2\eps}-\phi(\ooint )$ and $\tau_o\le \frac1{(\lambda-\eps)^-}$.

When $\phi$ satisfies suitable coercivity assumptions
(e.g.~its sublevels are locally compact), it is not difficult to 
check that $\DomainSlope\phi $ is dense in $\Dom\phi $, and in
particular is not empty. This can be proved \cite[Lemma 3.1.3]{Ambrosio-Gigli-Savare08}
by studying the properties of the minimizers $y\in \AmbientSpace$ of the variational problem \eqref{eq:1}.
More generally we will see that, thanks to \emph{Ekeland's variational principle}, 
the density of $\DomainSlope\phi$ in $ \Dom \phi $ holds if the sublevels of $\phi$ are merely
\emph{complete} (which in particular implies that $ \phi $ is l.s.c.). As a byproduct, if in addition $\phi$ is quadratically bounded from
below, $\GSlopeE\lambda\phi $ is also a proper functional for all $\lambda \le -\kappa_o$ (it is enough that $ x \mapsto \phi(x)-\tfrac{\lambda}{2}\DistSquare{x}{o} $ is linearly bounded from below actually).

Let us first recall the celebrated variational principle of Ekeland \cite{Ekeland74,Ekeland79}. 
\begin{theorem}[Ekeland's variational principle] 
  \label{thm:Ekeland}
  Let $\varphi:\AmbientSpace\to
  (-\infty,+\infty]$ be a functional bounded from below,
  with complete sublevels $\{x\in \AmbientSpace\!:\varphi(x)\le c\}$ for all
  $c\in \R$. % of a \Gsout{proper}
  % \footnote{Sembra conseguenza della \eqref{eq:cap1:82}}
  % functional $\varphi:\AmbientSpace\to
  % (-\infty,+\infty]$ are complete and $\inf_X\varphi>-\infty$., 
  For every $x\in \Dom\varphi$ % \in \AmbientSpace, \eps>0$
  % satisfy 
  % \begin{equation}
  %   \label{eq:cap1:82}
  %   \varphi(x_0)< \inf_{x \in \AmbientSpace} \varphi(x) +\eps.
  % \end{equation} 
  % Then
  and for every $\eta>0$ there exists a point
  $x_\eta\in \Dom\varphi$ such that
%  \footnote{Ho semplificato un po', mettendo solo quello che veramente
%    usiamo;
%    la stima con $\eps$ \`e una conseguenza e per noi \`e inutile. 
%  Infine ho riscritto le disuguaglianze come devono essere pensate....}
  \begin{subequations}
    \label{subeq:Ekeland}
    \begin{align}
      \label{eq:21}
      \varphi(x_\eta)&<\varphi(y)+\eta\,\Dist y{x_\eta} \quad \forevery y \in \AmbientSpace\setminus\{x_\eta\},\\
      % \Dist {x_0}{x_\eta}& < \eps\eta^{-1}, \\
      \label{eq:20}\varphi(x_\eta) +\eta\,\Dist {x}{x_\eta}&\le \varphi (x).
   \end{align}
  \end{subequations}
\end{theorem}

In the next definition we introduce a relaxed version of the variational problem that underlies \eqref{eq:1}, suitable to deal with possible lack of coercivity.  
\begin{definition}[Moreau-Yosida-Ekeland resolvent]
  \label{def:MYE-resolvent}
  Let $\phi:\AmbientSpace\to(-\infty,+\infty]$ be a proper functional. For every $ x \in \AmbientSpace $, $\tau>0$ and $\eta\ge0$ 
%  \GGG we denote by $\mathrm E_{\tau,\xi}[x,x']$ 
%  the set of points $z\in \Dom\phi$ 
%  characterized by the following two conditions:
%  \begin{subequations}
%      \begin{align}
%    \label{eq:cap1new:90bis}
%    \frac1{2\tau}\DistSquare x{z}+\phi(z)
%        &\le\frac1{2\tau}\DistSquare xy+ \phi(y)
%          +\xi\,\Dist y{z}
%    \quad
%    \text{for every } y\in \AmbientSpace ,\\
%    \label{eq:cap1new:90}
%        \frac1{2\tau}\DistSquare x{z}+\phi(z)+\xi\Dist z{x'}
%    &\le 
%        \frac1{2\tau}\DistSquare x{x'}+\phi(x').
%  \end{align}
%  \end{subequations}
%  \nc
 we denote by $\Resolvent\taueta x$ the (possibly empty) set of points
  $x_\taueta\in \Dom\phi$ characterized by the following two conditions:
  \begin{subequations}
      \begin{align}
    \label{eq:cap1:90bis}
    \frac1{2\tau}\DistSquare x{x_{\tau,\eta}}+\phi(x_{\tau,\eta})
    &\le\frac1{2\tau}\DistSquare xy+ \phi(y)
    +\frac\eta2\,\Dist x{x_\taueta}\,\Dist y{x_{\tau,\eta}}
    \quad
    \text{for every } y\in \AmbientSpace ,\\
    \label{eq:cap1:90}
    \frac1{2\tau}\DistSquare x{x_{\tau,\eta}}+\phi(x_{\tau,\eta})
    &\le \phi(x).
  \end{align}
  \end{subequations}
\end{definition}
Note that either $ \phi $ is \emph{not} quadratically
  bounded from below and therefore the infimum in \eqref{eq:1} is $
  -\infty $ for all $ \tau>0 $, or $ \phi $ is quadratically bounded
  from below and for every family $x_{\taueta}\in
  \Resolvent{\taueta}x$, $\eta>0$, $ \tau\in(0,\tau_o) $ and
  {$\bar x\in X$} the following properties hold:
  \begin{equation}
    \label{eq:85}
    \DistSquare{x_\taueta}x\le \frac{2\tau\tau_o}{\tau_o-\tau}
    \big(\phi(x)-\phi_{\tau_o}(x)\big),\qquad
    \lim_{\tau\downarrow0}\Dist{x_\taueta}x=0 \quad \text{provided } x \in \Dom{\phi},
  \end{equation}
\begin{equation*}
%\label{eq:84}
\lim_{\eta\downarrow0} \frac1{2\tau}\DistSquare
x{x_{\tau,\eta}} + \phi(x_\taueta)= \phi_\tau(x);\qquad
\phi_\tau(x)=
\frac1{2\tau}\DistSquare x{\bar x}+\phi(\bar x) \ \ \
\Leftrightarrow \ \ \ \bar x \in \Resolvent{\tau,0}x.
\end{equation*}
In particular, $\Resolvent\taueta x$ is a useful substitute (for $\eta>0$) of $\Resolvent{\tau,0}x$ when this set is empty, i.e.~when the infimum in \eqref{eq:1} is not attained. More importantly, the slopes of its points can be estimated in a quantitative way. 

\begin{theorem}[Ekeland resolvent and slopes] 
  \label{le:slope_by_Ekeland}
  Let $\phi:\AmbientSpace\to(-\infty,+\infty]$ be a proper functional and $ x \in \AmbientSpace $, $ \tau>0 $, $ \eta\ge 0 $. Then every
$x_\taueta\in \Resolvent\taueta x$ belongs to $\DomainSlope\phi$ and satisfies the bound
  \begin{equation}
    \label{eq:cap1:4}
    \MetricSlope\phi{x_\taueta}\le
    \GSlope{- 1  /\tau}\phi{x_\taueta}\le 
    (1+\tfrac 12\eta\tau)\frac{\Dist x{x_\taueta}}\tau.
  \end{equation}
  If moreover $\phi$ has complete sublevels (thus in particular it is
  l.s.c.) and is quadratically bounded from below
  according to \eqref{eq:cap1:13},
  then for every choice of $x\in \Dom\phi $,
  $\tau\in (0,\tau_o)$ and $\eta>0$, the set 
  $\Resolvent\taueta x$ is not empty.
  In particular, $\DomainSlope\phi $ is
  dense in $\Dom\phi $.  
\end{theorem}
\begin{proof}
  Let us first check \eqref{eq:cap1:4}: if $x_\taueta\in \Resolvent\taueta x$ then
  \eqref{eq:cap1:90bis} yields, for every $y\in \AmbientSpace$,
  \begin{align*}
    \phi(y)&\ge\phi(x_\taueta)-
    \frac\eta2\,\Dist x{x_\taueta}\Dist y{x_\taueta}+  \frac1{2\tau}
    \Big(\DistSquare x{x_\taueta}-
    \DistSquare xy\Big)\\
    &\ge
    \phi(x_\taueta)-
    \frac\eta2\Dist x{x_\taueta}\Dist y{x_\taueta}-\Dist y{x_\taueta}  \frac1{2\tau}
    \Big(\Dist x{x_\taueta}+
    \Dist xy\Big)\\
    &=
    \phi(x_\taueta)-
    \Big(\frac \eta2+\tau^{-1}\Big)\Dist x{x_\taueta}\Dist y{x_\taueta}-
    \Dist y{x_\taueta} \frac1{2\tau}
    \Big(\Dist xy-\Dist x{x_\taueta}\Big)\\
    &\ge
    \phi(x_\taueta)-
    \Big(\frac \eta2+\tau^{-1}\Big)\Dist x{x_\taueta}\Dist y{x_\taueta}-
    \frac1{2\tau}\DistSquare y{x_\taueta}.
 \end{align*} 
  Recalling \eqref{eq:cap1:89} and \eqref{eq:cap1:12} we obtain \eqref{eq:cap1:4}.

  Let us now prove that if $\phi$ satisfies \eqref{eq:cap1:13}, its sublevels are complete, 
  $x\in \Dom\phi$, $\tau<\tau_o$ and $\eta>0$,
  then $\Resolvent\taueta x$ is not empty. 
  We want to apply Ekeland's variational principle
  (Theorem \ref{thm:Ekeland}) to the function
  \begin{displaymath}
    \varphi(y):=\frac 1{2\tau}\DistSquare xy+\phi(y)\quad y\in \AmbientSpace,
  \end{displaymath} 
  choosing $x_0=x$. Note that $\varphi$ is bounded from below by \eqref{eq:cap1:13}.
  In case
  %$\phi(x)=\phi_\tau(x)$, then
  $x$ itself is a minimizer of $\varphi$, we can simply pick $x_{\tau,\eta}=x$.
  % , and it is easy to
%   check by \eqref{eq:cap1:83} that
%   $\MetricSlope\phi x=0$.
  Otherwise, $\phi(x)=\varphi(x)>\inf_{y\in \AmbientSpace}\varphi(y)$ and we can set
  $\eps=2(\varphi(x)-\inf_{y\in\AmbientSpace}
  \varphi(y))>0$.  
  % since $\phi(x)=\psi(x)$ and $\inf_\AmbientSpace \psi=\phi_\tau(x)$, $\eps$ is strictly positive.
  Upon choosing a vanishing sequence $\eta_n\downarrow0$,
  Ekeland's variational principle provides a sequence $ \{ x_n \} \subset \Dom\phi $ such that
  \begin{equation}
    \label{eq:19}
    \varphi(x_n)=\frac 1{2\tau}\DistSquare x{x_n}+\phi(x_n)
    \topref{eq:20} \le \varphi(x)=\phi(x)
  \end{equation}
  and
  \begin{equation}
    \label{eq:23}
    \varphi(x_n)\topref{eq:21}\le \varphi(y) + \eta_n \, \Dist y{x_n} \quad
    \forevery y\in \AmbientSpace .
  \end{equation} 
  Since $x$ is not a minimizer of $\varphi$ and $ \phi $ is l.s.c., there exists $ n_0 \in \N $ such that the quantity $\delta:=\inf_{n \ge n_0}\Dist x{x_n}$
  is strictly positive and therefore there exists in turn an integer $\bar n \ge n_0 $ such that
  $2\eta_{\bar n}\le \delta \eta$. 
  %\Comment{prima c'era $ n_0=0 $} 
  By choosing $x_\taueta = x_{\bar n}$, inequalities
  \eqref{eq:cap1:90} and \eqref{eq:cap1:90bis} follow from
  \eqref{eq:19} and \eqref{eq:23}, respectively. The last assertion is just a consequence of \eqref{eq:85}.
\end{proof} 

\subsection[${\lambda}$-convex and approximately $\lambda$-convex functionals]{$\boldsymbol{\lambda}$-convex and approximately $\boldsymbol \lambda$-convex functionals}

We will study a class of gradient flows 
which is strictly related to suitable convexity properties of the
driving functional. We therefore recall the meaning of convexity on geodesic subsets and introduce a slightly weaker property on (approximate) length subsets.  

\begin{definition}[$\lambda$-convexity]
  \label{def:lambda_convexity}
  A functional $\phi:\AmbientSpace\to(-\infty,+\infty]$
  is \emph{$\lambda$-convex} in $D\subset \AmbientSpace$ 
  for some $\lambda\in \R$ if
  for every
  couple of points $\Gdc_0,\Gdc_1\in D\cap\Dom\phi$ there
  exists a geodesic $\Gdc\in \GeoCon D{\Gdc_0}{\Gdc_1}$
  such that
  \begin{equation}
    \label{eq:cap1:60}
    \phi(\Gdc_t)\le (1-t)\phi(\Gdc_0)+t\phi(\Gdc_1)-\frac\lambda2
    t(1-t)\DistSquare{\Gdc_0}{\Gdc_1} \quad \forevery
    t\in [0,1].
  \end{equation}
  We denote by $\GeoConPhi D\lambda{\Gdc_0}{\Gdc_1}\subset
  \GeoCon D{\Gdc_0}{\Gdc_1}$
  the subset of geodesics enjoying \eqref{eq:cap1:60}. The functional
  $\phi$ is \emph{strongly $\lambda$-convex} in $D$ if
  it is $\lambda$-convex and \eqref{eq:cap1:60}
  holds for \emph{every} geodesic in $\GeoCon D{\Gdc_0}{\Gdc_1}$,
  i.e.~$\emptyset\neq\GeoConPhi D\lambda{\Gdc_0}{\Gdc_1}=
  \GeoCon D{\Gdc_0}{\Gdc_1}$.
\end{definition} 

\begin{definition}[Approximate $\lambda$-convexity]
  \label{def:app_lambda_convexity}
  A functional $\phi:\AmbientSpace\to(-\infty,+\infty]$ 
  is \emph{approximately $\lambda$-convex} in $D\subset \AmbientSpace$ 
  for some $\lambda\in \R$ if
  for every
  couple of points $\Gdc_0,\Gdc_1\in D\cap\Dom\phi$ and for every
  $\vartheta , \eps \in (0,1)$ there 
  exists a $(\vartheta,\eps)$-intermediate point
  $\Gdc_{\vartheta,\eps}\in D$ between $\Gdc_0 $ and $\Gdc_1$ 
  such that 
  \begin{equation}
    \label{eq:cap1:60-bis}
    \phi(\Gdc_{\vartheta,\eps})\le (1-\vartheta)\phi(\Gdc_0)+\vartheta\phi(\Gdc_1)-\frac{\lambda-\eps}2
    \vartheta(1-\vartheta)\DistSquare{\Gdc_0}{\Gdc_1} .
  \end{equation}
\end{definition}
% Note that $ D \subset X $ must be an approximate length subset.  

When $D = \AmbientSpace$ (or, equivalently, $D=\Dom\phi$) we just say that $\phi$ is (approximately) $\lambda$-convex, and if $ \lambda = 0 $ we say (approximately) ``convex'' rather than (approximately) ``$ 0 $-convex''. 

\smallskip

Note that an approximately convex function may be far from being
convex. For example, consider the Dirichlet function in $ [0,1] $
which takes  the values $0$ in $ \Q $ and $1$ in $ [0,1] \setminus \Q
$: the latter is clearly non-convex but turns out to be approximately
convex. Nevertheless, if $D$ is complete and $ \phi $ is l.s.c.~such kind of situations cannot occur (see Remark \ref{acag}).

\begin{lemma}
  \label{le:pedante}
  Let $D\subset X$ be a set whose intersection with the sublevels of $\phi$,
  $\{\phi\le c\}$ with $c\in \R$, are complete.
  If $\phi$ is approximately $\lambda$-convex in $D$
  then $D$ is a length subset of $\AmbientSpace$.
\end{lemma}
\begin{proof}
  We argue as in the proof of Lemma \ref{le:complete-length};
  we only have to check that the points
  constructed by the induction argument are contained in a common sublevel
  of $\phi$. In fact, we will show that one can pick
  every point $\sfx(k2^{-n})$ so that
  \begin{equation}
  \phi(\sfx(k2^{-n}))
  \le \varphi_n:=\varphi_0+a 
  \sum_{h=0}^{n-1} 4^{-h}\le \varphi_0+4a/3,
  \label{eq:45}
\end{equation}
where 
$\varphi_0:=\max(\phi(x_0),\phi(x_1))$,
$a:=\frac 12(\lambda-1)_-\DistSquare{x_0}{x_1}$.
We argue by induction: the case $n=0$ is trivial;
assuming \eqref{eq:45}, it is sufficient to check the induction step
for every $\vartheta\in \D_{n+1}$ of the form
$\vartheta=\sfx((2k+1)2^{-(n+1)})$ for some
integer $k\in [0,2^n-1]$. If we pick $ \sfx((2k+1)2^{-(n+1)}) $ as an $ \varepsilon_{n+1} $-midpoint between $ \sfx(k2^{-n}) $ and $ \sfx((k+1)2^{-n}) $ satisfying \eqref{eq:cap1:60-bis}, we deduce that
\begin{align*}
    \phi(\sfx((2k+1)2^{-(n+1)}))&\le \varphi_n+\frac 18(\lambda-1)_-
    \DistSquare{\sfx(k2^{-n})}{\sfx((k+1)2^{-n})}
    \\&\le
    \varphi_n+\frac {(\lambda-1)_- }{8\cdot 4^{n}}
    \DistSquare{x_0}{x_1}(1+\delta)^2\le 
    \varphi_n+\frac {(\lambda-1)_- }{2\cdot 4^{n}}
    \DistSquare{x_0}{x_1} \le \varphi_{n+1},  
  \end{align*}
  where we used the fact that $1+\delta\le 2$.
\end{proof}

\begin{remark}[Approximate convexity along approximate geodesics]\label{acag}
\upshape
Under the same hypotheses as in Lemma \ref{le:pedante}, one can show that for any given $ \varepsilon \in (0,1) $ there exists a Lipschitz curve $ \vartheta \in [0,1] \mapsto \sfx_{\vartheta,\varepsilon} \in D \cap \Dom{\phi} $ such that \eqref{eq:cap1:60-bis} holds and 
\begin{equation*}\label{pedante2}
      \Dist{\sfx_{\vartheta',\varepsilon}}{\sfx_{\vartheta'',\varepsilon}}\le
     \Dist{x_0}{x_1} \left(1+\varepsilon \right)\left|\vartheta'-\vartheta''\right| \quad \forevery \vartheta',\vartheta''
      \in [0,1].
\end{equation*}
In order to prove it, one has to adapt the proof of Lemma \ref{le:pedante} so as to make sure that \eqref{eq:cap1:60-bis} is satisfied at every induction step. We omit details, but let us mention that for this procedure to work it is necessary to require that each $ \varepsilon_{n} $, at fixed $n$, is chosen depending also on $k$ (just as a consequence of the fact that the $ \varepsilon $-error allowed in \eqref{eq:cap1:60-bis} is weighted by $ \vartheta(1-\vartheta) $). 
\end{remark}

% and approximate $ \lambda $-convex, then
% one could prove that every couple of point $x_0,x_1\in D$
% can be connected by approximate geodesics (in the sense discussed
% below \eqref{eq:92})
% along which an approximate convexity inequality holds, \nc
% a property that is much closer to standard $ \lambda $-convexity as in Definition \ref{def:lambda_convexity}.
%
% \smallskip
%
\nc
For $ \lambda $-convex functionals there is a way of characterizing $ \MetricSlope{\phi}{x} $ as modulus of the slope along the \emph{direction of maximal slope}, to some extent. 

\begin{definition}[Directional derivatives]
  \label{def:metric_slope2}
  If $\Gdc$ is a geodesic starting at $\Gdc_0\in \Dom\phi $ we set
  \begin{equation*}
    \label{eq:cap1:70a}
    \DirSuperDer\phi{\Gdc_0}\Gdc:=\liminf_{t\down0}\frac{\phi(\Gdc_t)-\phi(\Gdc_0)}t.
  \end{equation*}
\end{definition}
It is not difficult to check that, if in addition $\Gdc\in \GeoConPhi \AmbientSpace\lambda{\Gdc_0}{\Gdc_1}$, then 
\begin{equation}
  \label{eq:cap1:61}
  \phi(\Gdc_1)\ge \phi(\Gdc_0)+
  \DirSuperDer\phi{\Gdc_0}\Gdc
  +\frac\lambda2\DistSquare{\Gdc_0}{\Gdc_1}.
\end{equation}
Indeed, \eqref{eq:cap1:60} can be rewritten as
\begin{equation*}\label{eq: conv-t}
  \phi(\Gdc_1)\ge \phi(\Gdc_0)+
  \frac{\phi(\Gdc_t)-\phi(\Gdc_0)}t
  +\frac\lambda2(1-t)\DistSquare{\Gdc_0}{\Gdc_1};
\end{equation*}
passing to the limit as $t\down0$ we obtain \eqref{eq:cap1:61}.
If $\phi$ is $\lambda$-convex it follows that \cite[Theorem 2.4.9]{Ambrosio-Gigli-Savare08}
\begin{equation*}
  \label{eq:cap1:72}
  \MetricSlope\phi x=\sup_{y\in \Dom\phi, \,
    \Gdc\in \GeoConPhi\AmbientSpace\lambda xy}
  \frac{\left(\DirSuperDer\phi x\Gdc\right)^-}{|\dot\Gdc|}=
  \GSlope\lambda\phi x ,
 %  \sup_{y\neq x}
%   \Big(\frac{\phi(x)-\phi(y)}
%   {\Dist xy}+\frac \lambda2\Dist xy\Big)^+,
\end{equation*}
where the supremum is understood to be zero if $ x \in \Dom\phi $ is isolated in $ \Dom\phi $, whereas it is set to $ +\infty $ if $ x \not\in \Dom\phi $. Hence the slope of $\phi$ is lower semicontinuous if $ \phi $ is, and it can be characterized by the global lower bound
\eqref{eq:cap1:89}. In fact the identity between metric slope and global $ \lambda $-slope also holds if $ \phi $ is approximately $\lambda$-convex. 
\begin{proposition}[Approximate $ \lambda $-convexity and slopes]\label{lem:app-conv-slope}
  If $\phi : \AmbientSpace \to (-\infty,+\infty] $ is an approximately $\lambda$-convex functional then  
  \begin{equation} 
  \label{eq:cap1:72bis}
  \MetricSlope\phi x =
  \GSlope\lambda\phi x \quad \forevery x \in \AmbientSpace.
 %  \sup_{y\neq x} 
%   \Big(\frac{\phi(x)-\phi(y)}
%   {\Dist xy}+\frac \lambda2\Dist xy\Big)^+,
\end{equation}
If the sublevels of $\phi$ are complete, we also have
  \begin{equation} 
  \label{eq:cap1:72tris}
  \MetricSlope\phi x =\lMetricSlope\phi x
  \quad \forevery x \in \AmbientSpace .
 %  \sup_{y\neq x} 
%   \Big(\frac{\phi(x)-\phi(y)}
%   {\Dist xy}+\frac \lambda2\Dist xy\Big)^+,
\end{equation}
\end{proposition}
\begin{proof}
We can suppose with no loss of generality that $ x=\Geodesic_0 \in \DomainSlope{\phi} $ and $ y = \Geodesic_1 \in \Dom{\phi} $. By Definitions \ref{def:applength} (along with subsequent discussions) and \ref{def:app_lambda_convexity}, for every $\vartheta,\eps \in (0,1)$ we can find a $(\vartheta,\eps)$-intermediate point $\Geodesic_{\vartheta,\eps}$ between $ \Geodesic_0 $ and $\Geodesic_1$ such that 
\begin{equation}\label{eq:conv-t-bis}  
    \begin{gathered}
      \phi(\Geodesic_1)\ge \phi(\Geodesic_0)+
      \frac{\phi(\Geodesic_{\vartheta,\eps})-\phi(\Geodesic_0)}\vartheta
      +\frac{\lambda-\eps}2(1-\vartheta)\DistSquare{\Geodesic_0}{\Geodesic_1},\\
      \vartheta(1-\eps)\Dist{\Geodesic_0}{\Geodesic_1}\le \Dist{\Geodesic_0}{\Geodesic_{\vartheta,\eps}}\le
      \vartheta(1+\eps)\Dist{\Geodesic_0}{\Geodesic_1} , 
    \end{gathered}
\end{equation}
whence
\begin{displaymath}
    \phi(\Geodesic_1)\ge \phi(\Geodesic_0) -
      \frac{\left( \phi(\Geodesic_0) - \phi(\Geodesic_{\vartheta,\eps}) \right)^+ }{\Dist{\Geodesic_0}{\Geodesic_{\vartheta,\eps}}}
      \Dist{\Geodesic_0}{\Geodesic_1}(1 + \eps)
      +\frac{\lambda-\eps}2(1-\vartheta)\DistSquare{\Geodesic_0}{\Geodesic_1};
\end{displaymath}
by passing to the limit as $\vartheta\downarrow0$, we end up with 
\begin{displaymath}
    \phi(\Geodesic_1)\ge \phi(\Geodesic_0)-\MetricSlope\phi{\Geodesic_0}
      \Dist{\Geodesic_0}{\Geodesic_1}(1+\eps)
      +\frac{\lambda-\eps}2 \DistSquare{\Geodesic_0}{\Geodesic_1}.
\end{displaymath} 
Identity \eqref{eq:cap1:72bis} follows by finally letting
$\eps\down0$, recalling \eqref{eq:cap1:89}--\eqref{eq:cap1:12}.

In order to prove \eqref{eq:cap1:72tris} it is sufficient to observe
that $\Dom\phi$ is a length subset of $\AmbientSpace$, thanks to
Lemma \ref{le:pedante}, so that
$\DistName=\lDistName=\lDDistName{\Dom\phi}$
on $\Dom\phi$
and therefore $\PlainSlope\phi=\lPlainSlope\phi$ by the characterization
\eqref{eq: loc-slope} of the metric slope.
\end{proof}
% In particular, any $\lambda$-convex functional
% is \Reg\lambda according to definition \ref{def:Reglambda}.

Another important property of approximately $ \lambda $-convex
functionals $ \phi $ is their quadratic boundedness from below,
provided sublevels are complete. This is obvious if $\lambda>0$ as
long as standard convexity is concerned, since in such case $\phi$ is
even bounded from below, see \cite[Lemma
2.4.8]{Ambrosio-Gigli-Savare08}.
We will also show that $x\mapsto \phi(x)-\frac \lambda2\DistSquare x{o}$ 
is linearly bounded from below; in the case $\lambda=0$ 
we thus find a metric analog of the well-known property of convex
l.s.c.~functionals in Banach spaces.
%\Comment{ho solo cercato di rendere meno ambigua la frase che c'era prima}
% Let us first set
% \begin{equation}
%     \label{eq:cap1:74}
%     \tau_\lambda=\frac 1{\lambda^-}\quad\text{if }\lambda<0,\quad
%     \tau_\lambda=+\infty\quad\text{if }\lambda\ge0,
%   \end{equation}
%   so that $\lambda^- \tau<1$ if $\tau\in (0,\tau_\lambda)$.
\begin{theorem}[Approximate $\lambda$-convexity and quadratic boundedness] 
  \label{le:bfb}
  If
  $\phi : \AmbientSpace \to (-\infty,+\infty] $ 
  is approximately $\lambda$-convex and has complete sublevels, then
  it is quadratically bounded from below. More precisely, for every $o\in \Dom\phi$ and 
  $\kappa_o>-\lambda$ it satisfies the lower bound \eqref{eq:cap1:13} 
%   \begin{equation}
%     \label{eq:70}
%     \phi(x)+\frac {\kappa_o}{2}\DistSquare xo\ge \phi_o
%     \quad
%     \text{for every }x\in \AmbientSpace,
%   \end{equation}
  with
  \begin{equation}
    \label{eq:cap1:73}
    \phi_o:=\phi(o)-\frac{\big(\phi(o)-m_o+ \frac{\lambda^+}{2} \big)^2}
    {2(\lambda+\kappa_o)}
    -\frac{\lambda+\kappa_o}2,\quad
    m_o:=\inf\left\{\phi(x): x \in \AmbientSpace, \ \Dist x{o}\le 1 \right\}
    >-\infty.
  \end{equation} 
  %  it satisfies
%   \begin{equation}
%   \label{eq:cap1:75}
%   \phi_\tau(x)=
%   \inf_{y\in \Dom\phi }\frac 1{2\tau}\DistSquare xy+\phi(y)>-\infty\quad
%   \forevery x\in \AmbientSpace,\ \tau\in (0,\tau_o).
% \end{equation}
Furthermore, the functional $ x \mapsto \phi(x)-\frac{\lambda}{2} \DistSquare x o $ is linearly bounded from below and satisfies the lower bound \eqref{eq:cap1:13-lin} with 
  \begin{equation}
    \label{eq:cap1:73-lin}
        \ell_o:=\phi(o)-m_o+\frac{\lambda}{2} , \quad 
    \phi_o := m_o-\frac{\lambdam}{2} .
  \end{equation}
\end{theorem}
\begin{proof} 
  Let us first prove that $m_o$ given by  \eqref{eq:cap1:73} is finite: 
  % By  \cite[Lemma 2.4.8]{Ambrosio-Gigli-Savare05}
%   it is sufficient to show that $\phi$ 
%   is bounded from below on a neighborhood of a point $o\in \Dom\phi$, i.e.
  we argue by contradiction assuming that
  $m_o=-\infty$.
Let $ Z $ be the complete metric space given by 
$$ 
Z:=\big\{x\in \AmbientSpace: \, \phi(x) \le \phi(o) , \ \Dist x{o}\le 1 \big\} . 
$$ 
For all $ \eps \in (0,1) $, we can pick a general sequence $ \{ y_n \} \subset Z $ such that $\phi(y_{n})\le-4^n $ and define inductively another sequence $ \{ x_n \} $ by setting $ x_{n+1}:=\Geodesic_{\vartheta,\eps} $, where $ \vartheta=2^{-n} $ and $ \Geodesic_{\vartheta,\eps} $ is a $ (\vartheta,\eps) $-intermediate point between $ x_n $ and $ y_{n} $ fulfilling \eqref{eq:cap1:60-bis} with $ \Geodesic_0=x_n $ and $ \Geodesic_1=y_{n} $. The sequence is supposed to start from $ x_3=o $, so that $ \Dist{x_n}{o} \le 1 $ for all $ n \ge 3 $ and (note that $  \Dist{x_n}{y_n} \le 2 $)  
  \begin{equation}\label{recurs}
    \Dist{x_n}{x_{n+1}} \topref{eq:88} \le 2^{-{n}+1} (1+\eps) , \qquad
    \phi(x_{n+1})-\phi(x_n)
    \topref{eq:cap1:60-bis} \le
    2^{-n}\big[ \phi(y_{n})-\phi(x_n)+ 2 (\lambda^-+\eps) \big] .
  \end{equation} 
From the right-hand inequality of \eqref{recurs} it is not difficult to check that $ \phi(x_n)\to-\infty $ and therefore $ x_n \in Z $ eventually: because $ \{ x_n \} $ is a Cauchy sequence (hence converging to some
  $x_\infty\in Z$) and $ \phi $ is l.s.c., it follows that $ \phi(x_\infty)\le-\infty $, which is impossible since $-\infty$ is not in the range of $\phi$. 

In order to show the bound \eqref{eq:cap1:13}, it is not restrictive to assume $ x \in \Dom \phi $. The latter is then a simple consequence of \eqref{eq:conv-t-bis}, understood with $ \Geodesic_0=o $ and $ \Geodesic_1=x $, which yields for every $ \vartheta,\eps \in (0,1) $ 
  \begin{equation}\label{conv-bound}
      \phi(x)\ge \phi(o) -
      \vartheta^{-1} \big(\phi(o)- \phi(\Geodesic_{\vartheta,\eps})\big)
      +\frac{\lambda-\eps}2(1-\vartheta)\DistSquare{o}{x}.
  \end{equation} 
When $\Dist ox>1$ we can pick $\vartheta^{-1}:=\Dist ox (1+\eps)$, so that $\phi(\Gdc_{\vartheta,\eps})\ge m_o$; 
  in any case, we obtain for every $x\in \AmbientSpace$ 
  \begin{equation*}
    \phi(x) \ge \phi(o)-(\Dist ox\lor 1)(1+\eps)\left(\phi(o)-m_o + \lambda^+/2\right)+
    \frac {\lambda - \eps }2\DistSquare ox ,
  \end{equation*}
and by letting $ \eps \down 0 $ 
  \begin{equation*}\label{eq:lbb}
  \begin{aligned}
    \phi(x) \ge & \phi(o)-(\Dist ox\lor 1)\left(\phi(o)-m_o + \lambda^+/2\right)+
    \frac {\lambda }2\DistSquare ox \\
     \ge   &  \phi(o)-\frac1 {2\delta} \left(\phi(o)-m_o+\lambda^+/2 \right)^2-\frac \delta 2+\frac{\lambda-\delta}2\DistSquare ox
\end{aligned}
  \end{equation*}   
for every $ \delta > 0  $. By choosing $\delta:=\lambda+\kappa_o>0$ we get \eqref{eq:cap1:73}. Linear boundedness from below with \eqref{eq:cap1:73-lin} follows straightforwardly from \eqref{conv-bound} arguing similarly to above.
\end{proof}

\section{The metric approach to gradient flows}\label{sec: gf}
% In the following we will always assume that
% \begin{assumption}
%   \label{ass:phi}
%   $\phi:\AmbientSpace\to(-\infty,+\infty]$
%   is a function quadratically bounded from below
%   \eqref{eq:cap1:13} with \emph{complete} sublevels.
% \end{assumption}
The goal of this section is to study the properties of
$\lambda$-gradient flows, which will be associated with a
metric-functional system
\begin{equation}
  \label{eq:system}
  \begin{gathered}
    \X=\System, \ \text{where $(\AmbientSpace,\DistName)$ is a metric space and}\\
    \text{$\phi:X\to (-\infty,+\infty]$ is a proper l.s.c.~functional.}
  \end{gathered}
\end{equation}
In the following, $\X$ and $\System$ will always refer to \eqref{eq:system}. 

\subsection[Evolution Variational Inequalities
  (${\EVIname}$)]{Evolution Variational Inequalities
  ($\boldsymbol{\EVIname}$)}
\label{subsec:EVIdefinition}
The next (quite demanding) definition is modeled
after the case of $\lambda$-convex functionals in Euclidean-like spaces,
and has first been introduced \cite[Chapter 4]{Ambrosio-Gigli-Savare08}.
\begin{definition}[\EVIname\ and Gradient Flow]
  \label{def:GFlow}
  % and let $u_0\in \overline{\Dom\phi }$.
  Let $\lambda\in \R$ and $\X=\System$ be a metric-functional system as
  in \eqref{eq:system}. 
  A solution of the \emph{Evolution Variational Inequality} $\EVI
  X\DistName\phi\lambda$ 
  is a continuous curve
  $ u: t\in (0,+\infty)\mapsto u_t\in \Dom\phi $ such that
  %$\lim_{t\downarrow0}u_t=u_0$ and
  \begin{equation} 
    \label{eq:EVI}
    \frac 12 \Urd\DistSquare {u_t}v+
    \frac \lambda 2\DistSquare{u_t}v
    \le \phi(v)-\phi(u_t)
   \quad
    \forevery t\in (0,+\infty), \
    v\in \Dom\phi .
    \tag{\EVIshort\lambda}
  \end{equation}
  An \emph{\EVIshort\lambda-Gradient Flow} of $\phi$ in $D\subset \overline{\Dom\phi}$ is a family of continuous maps
  $\FlowName_t:D\to D$, $t\ge0$, such that for every $u_0\in D$
  \begin{subequations}
    \begin{equation}
      \Flow{t+h}{u_0}=\Flow h{\Flow t{u_0}}\quad \forevery t,h\ge 0,
      \qquad
    \lim_{t\downarrow0}\Flow t{u_0}=\Flow 0{u_0}={u_0},
    \label{eq:18}
  \end{equation}
  \begin{equation}
    \label{eq:28}
    \text{the curve }t\mapsto \Flow t{u_0} \ \text{is  a solution of \EVI \AmbientSpace\DistName\phi\lambda}.
  \end{equation}
  \end{subequations}
\end{definition}
\begin{remark}[Upper, lower and distributional derivatives]
  \label{rem:ulder}
  \upshape
  Thanks to Lemma \ref{le:monotonicity} from Appendix \ref{app}, for \emph{continuous} curves $u $ the upper right derivative in \eqref{eq:EVI} can be replaced by the lower right derivative or the distributional one; in other words, \eqref{eq:EVI} is equivalent to
  \begin{equation} 
    \label{eq:EVIbis}
    \frac 12 \Lrd\DistSquare {u_t}v+
    \frac \lambda 2\DistSquare{u_t}v
    \le \phi(v)-\phi(u_t)
   \quad
   \forevery t\in (0,+\infty), \ v\in \Dom\phi,
    \tag{$ \EVIname_\lambda^\ell $}    
 \end{equation}
 or to
  \begin{equation}
    \label{eq:EVItris}
    \frac 12 \frac \d\dt \DistSquare {u_t}v+
    \frac \lambda 2\DistSquare{u_t}v
    \le \phi(v)-\phi(u_t)
   \quad
   \text{in }\DD'((0,+\infty)), \quad 
   \forevery v\in \Dom\phi 
    \tag{$\EVIname_\lambda^\DD $}    
 \end{equation} 
along with the further requirement $\phi\circ u\in L^1_{\rm loc}((0,+\infty))$. If $u$ is also \emph{absolutely} continuous,
  a pointwise derivative almost everywhere (compare with \cite[Theorem 4.0.4]{Ambrosio-Gigli-Savare08})
  is sufficient: indeed, in such case, $ u $ satisfies $\EVI \AmbientSpace\DistName\phi\lambda$ in the integral form \eqref{eq:cap1:78} below (the fact that $\phi\circ u \in L^1_{\rm loc}((0,+\infty)) $ easily follows from the lower semicontinuity of $ \phi $ and \eqref{eq:EVI} itself).
\end{remark} 

The next result shows two different characterizations
of solutions to \eqref{eq:EVI} (see also
\cite{Daneri-Savare08,Clement-Desch10b}).
At the core of its proof is the fact that $t\mapsto \phi(u_t)$ is a
nonincreasing function along any solution to \eqref{eq:EVI}. To our purposes, from here on it is convenient to set
\begin{equation} 
  \label{eq:cap2:11}
  \sfE\lambda t:=\int_0^t \rme^{\lambda r}\,\d r=
  \begin{cases}
    \frac{\rme^{\lambda t}-1}\lambda&\text{if }\lambda\neq0,\\
    t&\text{if }\lambda=0.
  \end{cases}
\end{equation}

\begin{theorem}[Integral characterizations of $\EVIname$]
  \label{thm:uniqueness}
  A curve $u:(0,+\infty)\to
  \overline{\Dom\phi}$ is a solution of $\EVI X\DistName\phi\lambda$
  according to Definition
  \ref{def:GFlow} if and only if it satisfies one of the following
  equivalent formulations.
  \begin{description}
  \item[\ICA] For every $ v \in \Dom \phi $ the maps $t\mapsto \phi(u_t)$, $ t\mapsto \DistSquare{u_t}v$ belong to
    $L^1_{\rm loc}((0,+\infty))$ and for all $s,t\in (0,+\infty)$ with $s<t$ there holds 
  \begin{equation} 
    \label{eq:cap1:78}
    \frac 12\DistSquare{u_t}v-
    \frac 12\DistSquare{u_s}v+\int_s^t\Big(\phi(u_r)+
    \frac {\lambda} 2\DistSquare{u_r}v\Big)\,\d r\le
    (t-s)\phi(v).
    \tag{$\EVIname_\lambda^{\prime}$}
    % \quad\forevery\, v\in \Dom\phi .
  \end{equation}
\item[\ICB] %$u_t\in \Dom\phi $ for every $t>0$ and
  For every $s,t\in (0,+\infty)$ with $s<t$ and $v\in \Dom\phi$ there holds
  \begin{equation}
    \label{eq:cap1:78bis}
    \frac {\rme^{\lambda(t-s)}}2\DistSquare{u_t}v-
    \frac 12\DistSquare{u_s}v\le
    \sfE\lambda{t-s}\Big(\phi(v)-\phi(u_t)\Big).
    \tag{$\EVIname_\lambda^{\prime\prime}$}
    %\quad\forevery\, v\in \Dom\phi .
  \end{equation}
  \end{description}
Furthermore, the map $t\mapsto\phi(u_t)$ is nonincreasing (in particular it is right continuous).
  \end{theorem}
\begin{proof}
  We split the proof in various steps. 
  
  \medskip
  
  \noindent\emph{\eqref{eq:EVI} $\Rightarrow$ \eqref{eq:cap1:78}: solutions of \eqref{eq:EVI} according to Definition \ref{def:GFlow}
    satisfy \textnormal{\ICA}}. It is enough to apply Lemma \ref{le:monotonicity} to the functions
  \begin{displaymath}
    \zeta(t)=\frac 12\DistSquare{u_t}v,\quad
    \eta(t)=
    \frac \lambda 2\DistSquare{u_t}v+\phi(u_t)-\phi(v).
  \end{displaymath}

 \medskip

\noindent\emph{\eqref{eq:cap1:78bis} $\Rightarrow$ \eqref{eq:EVI}: solutions of \textnormal{\ICB}\ satisfy \eqref{eq:EVI} according to Definition \ref{def:GFlow} (and $t\mapsto \phi(u_t)$ is nonincreasing)}. First of all note that \eqref{eq:cap1:78bis} yields $ u_t \in \Dom \phi $ for all $ t>0 $. Hence by choosing $v=u_s$ in \eqref{eq:cap1:78bis} 
  we immediately get that $t\mapsto \phi(u_t)$ is nonincreasing, in
  particular
  it is locally bounded in $(0,+\infty)$.
  It follows that 
  \begin{displaymath}
    \limsup_{t\downarrow t_0}\Dist{u_t}v\le
    \Dist{u_{t_0}}v\le \liminf_{t\uparrow t_0}\Dist{u_t}v\quad
    \forevery t_0>0 , \ v\in \Dom\phi ,
  \end{displaymath}
  so that $u$ is right continuous and the maps
  $t\mapsto \Dist{u_t}v$, $ t\mapsto\phi(u_t)$ are lower semicontinuous
  and satisfy \eqref{eq:EVI}; arguing as in step \eqref{eq:EVI} $\Rightarrow$ \eqref{eq:cap1:78} (only lower semicontinuity is needed) we can show that they satisfy \ICA\ and by the next step $u$ is also continuous. 

\medskip 

\noindent\emph{\eqref{eq:cap1:78} $\Rightarrow$ \eqref{eq:EVI}: solutions of \textnormal{\ICA}\ satisfy \eqref{eq:EVI} according to Definition \ref{def:GFlow} (and $t\mapsto \phi(u_t)$ is nonincreasing)}. To begin with, from \eqref{eq:cap1:78} we immediately get the right continuity of $ u $, since
  \begin{displaymath}
    \limsup_{t\down s}\DistSquare{u_t}v\le
    \DistSquare{u_s}v\quad
    \forevery\ v\in \Dom\phi
  \end{displaymath} 
  and $u_s\in \overline{\Dom\phi }$. Because $t\mapsto \phi(u_t) \in
  L^1_{\rm loc}((0,+\infty)) $ and $ u $ is right continuous, there
  exists a sequence $ v_n \subset \Dom \phi $ which is dense in the
  image of $ u $. In particular there holds
\[
\DistSquare{u_t}{u_s} = \sup_{n \in \N} \left| \Dist{u_t}{v_n} - \Dist{u_s}{v_n}  \right|^2 \quad \forall s,t \in (0,+\infty) ,
\]  
so that the real function $ (t,s) \mapsto \DistSquare{u_{t}}{u_s} $ is Lebesgue measurable in $ (0,+\infty) \times (0,+\infty) $. 
%  In particular, for every $r_0,s_0\in (0,+\infty)$ 
%  \begin{equation}
%    \label{eq:71}
%    \lim_{r\down r_0,s\down s_0}\DistSquare{u_r}{u_s}=\DistSquare{u_{r_0}}{u_{s_0}},
%  \end{equation}
%  so that a simple application of \cite[2.9.13]{Federer69} to the Vitali relation $V:=\{(x,x+[0,r]^2):x\in \R^2,r>0\}$ shows that 
%  the real function $(r,s)\mapsto \DistSquare{u_{r}}{u_s}$ is Lebesgue-measurable in $(0,+\infty)\times(0,+\infty)$.
%\Comment{qui come dicevo ho evitato di ricorrere alle ``malefiche'' relazioni di Vitali...}  
Given almost every $ s>0 $ and any $ h>0 $, we can choose $v=u_s$ and $ t=s+h$ in \eqref{eq:cap1:78} to get 
  \begin{equation}
    \label{eq:cap1:70}
    \frac 12\DistSquare{u_{s+h}}{u_s}+\frac \lambda 2\int_0^h
    \DistSquare{u_{s+r}}{u_s}\,\d r\le
    \int_0^h \Big(\phi(u_s)-\phi(u_{s+r})\Big)\,\d r;
  \end{equation}
  a further integration w.r.t.\ $s$ from $t_0>0$ to $t_1>t_0$ yields
  \begin{align*}
     & \int_{t_0}^{t_1} \DistSquare{u_{s+h}}{u_s}\,\d s
    + \lambda \int_0^h \int_{t_0}^{t_1}
    \DistSquare{u_{s+r}}{u_s}\,\d s \,\d r \le 
    2 \int_0^h \int_{t_0}^{t_1}\Big(\phi(u_s)-\phi(u_{s+r})\Big)\,\d
    s\,\d r
    \\ = &
    2\int_{0}^h \int_0^r
    \Big(\phi(u_{t_0+\xi})-\phi(u_{t_1+\xi})\Big)\,\d\xi\,\d r = 
    2\int_{0}^h \int_0^h
    \Big(\phi(u_{t_0+\xi})-\phi(u_{t_1+\xi})\Big)\nchi_{[0,r]}(\xi)\,\d\xi\,\d r 
    \\ = &
    2\int_{0}^h
    \Big(\phi(u_{t_0+\xi})-\phi(u_{t_1+\xi})\Big)(h-\xi)\,\d\xi
    = 
    2h^2\int_{0}^1
    \Big(\phi(u_{t_0+h\xi})-\phi(u_{t_1+h\xi})\Big)(1-\xi)\,\d\xi.
  \end{align*}
  Upon setting $x(h):=\int_{t_0}^{t_1}\DistSquare{u_{s+h}}{u_s}\,\d s$,
  we therefore get 
  \begin{equation*}
    \label{eq:cap1:84}
    x(h)+\lambda \int_0^h x(r)\,\d r\le
    h^2\,y(h),\qquad
    y(h):=2\int_{0}^1
    \Big(\phi(u_{t_0+h\xi})-\phi(u_{t_1+h\xi})\Big)(1-\xi) \,\d \xi,
  \end{equation*}
  so that Gronwall's Lemma yields
  \begin{equation}
    \label{eq:cap1:85}
    h^{-2}\,x(h)\le \rme^{\lambda^- h}\sup_{0\le \delta\le h}y(\delta).
  \end{equation}
  If $t_0,t_1$ are Lebesgue points of the map $s\mapsto \phi(u_s)$,
  then 
  \begin{equation*}
    \label{eq:cap1:86}
    \lim_{h\down0} y(h)=
    \lim_{h\down0} \sup_{0\le \delta\le h}y(\delta)
    = \phi(u_{t_0})-\phi (u_{t_1}),
    %\qquad     \sup_{0<\delta<1} y(\delta)<+\infty,
  \end{equation*}
 whence from \eqref{eq:cap1:85} 
  \begin{equation}
    \label{eq:cap1:87}
    %\sup_{0<h<1}\int_{t_0}^{t_1}\frac{\DistSquare{u_{s+h}}{u_s}}{h^2}\,\d s
    %<+\infty, \quad
    \limsup_{h\down
      0}\int_{t_0}^{t_1}\frac{\DistSquare{u_{s+h}}{u_s}}{h^2}\,\d s
    \le \phi(u_{t_0})-\phi(u_{t_1}).
  \end{equation}
  It follows that
  the restriction of the
  map $t\mapsto \phi(u_t)$ to its Lebesgue set is not increasing;
  $\phi$ being l.s.c.\ and $u$ right continuous, we get
  $\phi(u_t)\le \phi(u_{t_0})$ for every Lebesgue point $t_0<t$,
  and in particular $u_t\in \Dom\phi $ for every $t>0$.
  On the other hand,
  by choosing $ t=s+h $ and $v=u_s$ in \eqref{eq:cap1:78} (now we are allowed to do it for every $ s>0 $) and recalling the
  right continuity of $u$, we end up with
  \begin{equation*}
    \label{eq:cap2:1}
    \phi(u_t)\le \limsup_{h\down0}\frac 1{h}\int_s^{s+h}
    \phi(u_r)\,\d r\le \phi(u_s)\quad
    \forevery  t>s>0,
  \end{equation*}
 i.e.\ $ t \mapsto \phi(u_t) $ is nonincreasing in $ (0,+\infty) $. As $\phi$ is lower semicontinuous, we also deduce that 
  \begin{equation} 
    \label{eq:cap2:1bis}
    \phi(u_t)\le \phi(u_s)=
    \lim_{h\down0}\phi(u_{s+h})=
    \lim_{h\down0}\frac 1{h}\int_s^{s+h}
    \phi(u_r)\,\d r\quad
    \forevery  t>s > 0.
  \end{equation}
Moreover, \eqref{eq:cap1:87} and the triangle inequality show that for every interval $(t_0,t_1)\subset (0,+\infty)$
  there exists a constant $C_{t_0,t_1}>0$ such that 
  \begin{equation}
    \label{eq:29}
    \int_{t_0}^{t_1-h} \frac{\DistSquare{u_{s+h}}{u_s}}{h^2}\,\d s
    \le C_{t_0,t_1}\quad \forevery h\in (0,t_1-t_0).
  \end{equation} 
  Since $|\Dist {u_{s+h}}w-\Dist {u_s}w|\le \Dist{u_{s+h}}{u_s}$, by \eqref{eq:29} we can infer that the curve
  $t\mapsto \Dist{u_t}w$ belongs to the Sobolev space $W^{1,2}_{\rm loc}((0,+\infty))$ for every $w\in X$.
  Because it is also right continuous, it coincides with its continuous representative and therefore
 $u\in C^0((0,+\infty);\AmbientSpace)$. By differentiating \eqref{eq:cap1:78} and using \eqref{eq:cap2:1bis}, we finally obtain \eqref{eq:EVI}. 
  
  \medskip 

\noindent\emph{\eqref{eq:EVI} $\Rightarrow$ \eqref{eq:cap1:78bis}: solutions of \eqref{eq:EVI} according to Definition \ref{def:GFlow} satisfy \textnormal{\ICB}}. We multiply inequality \eqref{eq:EVI} by
  $\rme^{\lambda t}$ obtaining
  \begin{equation}
    \label{eq:cap2:13}
    \Urd\Big(\frac {\rme^{\lambda t}}2\DistSquare{u_t}v\Big)\le
    \rme^{\lambda t}\Big(\phi(v)-\phi(u_t)\Big).
  \end{equation}
By combining steps \eqref{eq:EVI} $\Rightarrow$ \eqref{eq:cap1:78} and \eqref{eq:cap1:78} $\Rightarrow$ \eqref{eq:EVI} we have established that $t\mapsto \phi(u_t)$ is nonincreasing, so that an integration of \eqref{eq:cap2:13} through Lemma \ref{le:monotonicity} yields \eqref{eq:cap1:78bis}.
\end{proof} 

\begin{remark}[$ \EVIname $ in $ [0,+\infty) $]\label{rem: to-zero} 
\upshape
We can actually extend the above results down to $ t=0 $. More precisely, a curve $u:[0,+\infty) \to \overline{\Dom\phi}$ is a solution of $\EVI X\DistName\phi\lambda$ according to Definition \ref{def:GFlow} (with the additional requirement that $ u \in C^0([0,+\infty);\AmbientSpace) $) if and only if it is a solution of \ICA\ down to $ s=0 $ (with the additional requirements that $ t \mapsto \DistSquare{u_t} v \in L^1_{\rm loc}([0,+\infty))$, $ t \mapsto \phi^-(u_t) \in L^1_{\rm loc}([0,+\infty)) $) and if and only if it is a solution of \ICB\ down to $ s=0 $. Accordingly, the map $ t \mapsto \phi(u_t) $ is nonincreasing in $ [0,+\infty) $. Furthermore, if $ u_0 \in \Dom \phi $ then \eqref{eq:EVI} also holds at $ t=0 $, whereas in case $ u_0 \not\in \Dom \phi $ both the l.h.s.\ and the r.h.s.\ are $ -\infty $. 

In order to prove such properties, one can recall the proof of Theorem \ref{thm:uniqueness} and exploit the already established local (in $[0,+\infty)$) boundedness from below of $ t \mapsto \phi(u_t) $. 
\end{remark}
% In order to get simpler formulae, we set
% \begin{equation}
%   \label{eq:cap2:6}
%   \EExp\lambda t:=\int_0^t \rme^{\lambda s}\,\d s=
%   \begin{cases}
%     t&\text{if }\lambda=0,\\
%     \frac{\rme^{\lambda t}-1}\lambda&\text{if }\lambda\neq0.
%   \end{cases}
% \end{equation}

\subsection[The fundamental properties of $ {\EVIname} $]{The fundamental properties of $ \boldsymbol{\EVIname} $}
The next theorem collects many useful results which illustrate some
important consequences of $\EVI\AmbientSpace\DistName\phi\lambda$,
which can be considered as the metric version of the analogous
properties
for gradient flows of $\lambda$-convex functionals in Hilbert spaces,
see e.g.~\cite{Brezis73};
partial results have also been obtained 
in the Wasserstein framework or under some joint convexity properties
of $\System$, see
\cite{Ambrosio-Gigli-Savare08,Ambrosio-Savare06}.
In particular, the energy identity \eqref{eq:Cap12:24} plays an important role in
\cite{Ambrosio-Gigli13,Ambrosio-Gigli-Savare14b}.

\begin{theorem}[Properties of $\EVIname_\lambda $] 
  \label{thm:main1}
  Let $u, u^1,u^2\in C^0([0,+\infty);\AmbientSpace)$ be solutions of
  $\EVI\AmbientSpace\DistName\phi\lambda$. Then the following claims hold.
  \begin{description}
  
  \item[$\boldsymbol \lambda$-contraction and uniqueness]
    \begin{equation}
    \label{eq:cap1:81}
    \Dist{u^1_t}{u^2_t}\le \rme^{-\lambda (t-s)}
    \Dist{u^1_s}{u^2_s}\quad
    \forevery 0\le s<t<+\infty.
  \end{equation}
  In particular, for every $u_0\in \overline{\Dom\phi }$
  there is at most one solution $u$ of $\EVI X\DistName\phi\lambda$ satisfying
  the initial condition $\lim_{t\down0}u_t=u_0$.
  
\item[Regularizing effects]
  \begin{gather}
 \text{$u$ is \emph{locally Lipschitz} in $(0,+\infty)$ and
   $u_t\in \DomainSlope\phi \subset \Dom\phi $
   for every $t>0$}; \nonumber \label{eq: c24} \\
     \text{the map $t\in[0,+\infty) \mapsto \phi(u_t)$ is
       nonincreasing and (locally semi- in $(0,+\infty)$, if $\lambda<0$) convex};
     \label{eq:Cap1:58}\\
     \text{the map $t\in[0,+\infty)\mapsto \rme^{\lambda t}
       \MetricSlope\phi{u_t}$ is nonincreasing and right continuous}. \label{eq:cap2:3}
   \end{gather}
   
   \item[\emph{A priori} estimates] for every $v\in \Dom\phi $, $ t > 0 $ 
    \begin{equation}
        \label{eq:2}
        \frac {\rme^{\lambda t}}2\DistSquare{u_t}v+
        \sfE\lambda t\Big(\phi(u_t)-\phi(v)\Big)+\frac {\big(\sfE \lambda t\big)^2}2\MetricSlopeSquare\phi{u_t}\le
        \frac 12\DistSquare{u_0}v ,
      \end{equation}
%    \begin{equation}
%      \label{eq:Cap12:111}
%      \phi(u_t) \le
%      \phi(v)+
%      \frac{1}{2 \sfE\lambda t}\DistSquare{u_0}{v},\\
%      %\quad\text{for every } v\in \Dom\phi ,
%    \end{equation}
    and for every $ v \in \Dom{\GSlopeE\lambda\phi} $
     \begin{equation}
      \label{eq:Cap1:38}
      \MetricSlopeSquare\phi{u_t} \le
      \frac{1}{2\rme^{\lambda t}-1} {\GSlopeSquare\lambda\phi v} +
      \frac1{(\sfE\lambda t)^2}\DistSquare{u_0}v\quad
      % \forevery\, v\in \DomainSlope\phi ,\
      \text{provided }-\lambda t<\log 2.
      % \rme^{-2\lambda^- t}\MetricSlopeSquare\phi{u_t}\le
%       \MetricSlopeSquare\phi v+
%       \frac1{t^2}\DistSquare{u_0}v-\frac{\lambda}{t}
%       \Big(\frac12\DistSquare{u_t}v+
%       \frac1{t}\int_0^t \DistSquare{u_s}v\,ds\Big)
    \end{equation}
    % \begin{gather}
%       \label{eq:cap2:12}
%       %\begin{aligned}
%         \DistSquare{u_t}{u_o}%&
%         \le \Big(\DistSquare{u_0}{u_o}+
%         2\sfE{\lambda+\lambda_o}t\big(\phi(u_o)-m_o\big)
%         \Big)\rme^{-(\lambda+\lambda_o) t}
%         %&%\text{if }&
%         % \lambda_o<0\text{ in }\eqref{eq:cap1:13};
%         % \\
% %         \label{eq:cap2:12bis}
% %         \DistSquare{u_t}{u_0}%&
% %         \le t\big(\phi(u_0)-m_o\big)\quad
% %         % &
% %         \text{if }
% %         % &
% %         \lambda_o=0\text{ in }\eqref{eq:cap1:13}.
% %       %\end{aligned}
%       \end{gather}

    \item[Asymptotic expansions as $ \boldsymbol{ t\downarrow0 }$] 
      if $u_0\in \DomainSlope\phi $ and $\lambda\le 0$ then for every $v\in \Dom\phi $, $t\ge0$ 
      \begin{equation}
        \label{eq:3}
        \frac {\rme^{2\lambda t}}2\DistSquare{u_t}v-
        \frac 12\DistSquare{u_0}v\le \sfE{2\lambda} t \Big(\phi(v)-\phi(u_0) \Big)+
        \frac {t^2}2\MetricSlopeSquare\phi{u_0}.
      \end{equation}
      The functional $ x \mapsto \phi(x) - \frac{\lambda}{2} \DistSquare x o $ is linearly bounded from below for all $ o \in \Dom \phi $, and for every $ u_0 , v  \in \Dom\phi $, $ t \ge 0 $ there holds 
        \begin{equation}
        \label{eq:3-lessreg}
        % \begin{gathered}
        \frac {\rme^{\lambda t}}2\DistSquare{u_t}v-
        \frac 12\DistSquare{u_0} v \le  \sfE{\lambda} t  \Big(\phi(v)-\phi(u_0)\Big) + 2 \rme^{-\lambda t} \int_0^t \rme^{2\lambda s} \Big( \phi(u_0) - \phi_{\sfE{\lambda}{s}}(u_0) \Big) \d s .   
        %\end{gathered}
      \end{equation}
%       \item[$\lambda$-regularity of $\phi$]
%     The functional $\phi$ satisfies the $\lambda$-regularity condition %(see Definition \ref{def:Reglambda})
%     \begin{equation}
%       \label{eq:4}
%       \MetricSlope\phi u=\GSlope\lambda\phi u\quad\text{for every }u\in  D\cap \Dom\phi .
%     \end{equation}

  \item[Right, left limits and energy identity]
  for every $t>0$ the \emph{right limits}
  \begin{equation*}
    \label{eq:Cap12:23}
    |\dot u_{t+}|:=\lim_{h\down0}\frac{\Dist{u_{t+h}}{u_t}}h,\quad
    \frac{\d}{\d t}\phi(u_{t+}):=
    \lim_{h\down0}\frac{\phi(u_{t+h})-\phi(u_t)}h
    % \frac 12\DistSquare{u_t}v:=
%     \lim_{h\down0}\frac 1{2h}
%     \Big(\DistSquare{u_{t+h}}v-
%     \DistSquare{u_{t}}v\Big)
  \end{equation*}
  exist finite, satisfy
  \begin{equation}
      \label{eq:Cap12:24}
      \frac{\d}{\d t}\phi(u_{t+})=-|\dot u_{t+}|^2=
      -\MetricSlopeSquare\phi{u_t}=
      -\lMetricSlopeSquare\phi{u_t}=
      -\GSlopeSquare\lambda\phi{u_t}\quad
      \forevery  t>0
    \end{equation}
    and define a right-continuous map. If in addition $ u_0 \in \DomainSlope\phi $ the same holds at $ t=0 $, otherwise all the terms in \eqref{eq:Cap12:24} are $ -\infty $ at $ t=0 $.
  
Moreover, there exists an at-most-countable set $\mathcal C\subset (0,+\infty)$ such that the analogous identities for the \emph{left limits} hold in $ (0,+\infty)\setminus\mathcal C$, while in $ \mathcal{C} $ the latter hold upon replacing $\MetricSlope\phi{u_t} $, $\lMetricSlope{\phi}{u_t}$ and $ \GSlope\lambda\phi{u_t} $ with $  \lim_{h\downarrow 0} \MetricSlope\phi{u_{t-h}} $, $  \lim_{h\downarrow 0} \lMetricSlope\phi{u_{t-h}} $ and $ \lim_{h\downarrow 0} \GSlope\lambda\phi{u_{t-h}} $, respectively. 
% \begin{equation}
%       \label{eq:Cap1:59}
%       \left\{
%       \begin{aligned}
%        & \vv_{t-}=\lim_{h\down0}\frac{\ShortMap{u_t}{u_{t-h}}-\Identity}h
%         =-\MinSelT\phi(u_t),\\
%        &  \frac{d}{dt_-}\phi(u_t)=
%         -\MetricSlopeSquare\phi{u_t}.
%       \end{aligned}
%        \right.`
%       \quad\forevery\, t\in (0,+\infty)\setminus\mathcal C.
%     \end{equation}
%   \item[Energy identity.]
%     The map $t\mapsto \phi(u_t)$ is locally absolutely
%     continuous and the energy identity holds:
%     \begin{equation}
%       \label{eq:Cap12:10bis}
%       \frac d{dt}\phi(u_t)=-\MetricSlope\phi{u_t}
%       |\dot u_t|=-\MetricSlopeSquare\phi{u_t}=-|\dot u_t|^2
%       \quad\text{for a.e.\ }t\in >0
%     \end{equation}

  \item[Asymptotic behaviour as $ \boldsymbol{t \to +\infty} $] 
    if $\lambda>0$ and $ \phi $ has complete sublevels, then
    it admits a unique minimum point $\bar u$
    and for every  $t > t_0\ge0$ we have:
    \begin{subequations} 
      \begin{gather}
        \label{eq:Cap1:42}
        \frac\lambda 2\DistSquare{u_t}{\bar u}\le
        \phi(u_t)-\phi(\bar u)\le \frac1{2\lambda}
        |\partial\phi|^2(u_t),
        %\quad\forevery\, t\ge0,
        \\
        \label{eq:Cap12:18}
        \Dist{u_t}{\bar u}\le
        \Dist{u_{t_0}}{\bar u} \, \rme^{-\lambda
          (t-t_0)},
        \\ 
        \label{eq:Cap1:40}
        \phi(u_t)-\phi(\bar u)\le
        \Big(\phi(u_{t_0})-\phi(\bar u)\Big)\rme^{-2\lambda(t-t_0)},
        \quad
        \phi(u_t)-\phi(\bar u)\le
        \frac{1}{2\EExp\lambda{t-t_0}}\DistSquare{u_{t_0}}{\bar u},
        \\
        \label{eq:Cap1:41}\MetricSlope\phi{u_t}\le
        \MetricSlope\phi{u_{t_0}}\,\rme^{-\lambda(t-t_0)},\quad
        \MetricSlope\phi{u_t}\le
        \frac1{\EExp\lambda {t-t_0}}\Dist{u_{t_0}}{\bar u}.
      \end{gather}
    \end{subequations} 
  If $\lambda=0$ and $\bar u$ is a minimum point of $\phi$ (if any), then 
  \begin{equation} 
    \label{eq:Cap12:19}
    \begin{gathered}
      \MetricSlope\phi{u_t}\le \frac{\Dist
        {u_0}{\bar u}}t,
      \quad
      \phi(u_t)-\phi(\bar u)\le \frac{
        \DistSquare{u_0}{\bar u}}{2t} \qquad \forevery t>0,
      \\
      \text{the map} \ t \mapsto \Dist{u_t}{\bar u} \ \text{is nonincreasing}.
  \end{gathered}
\end{equation}
If in addition $ \phi $ has \emph{compact} sublevels, then $ u_t $ converges to a minimum point of $ \phi $ as $ t\to+\infty $.

  \item[Stability]
    if $ \{u^n\} \subset C^0([0,+\infty);\AmbientSpace)$ is a sequence of solutions of $\EVI\AmbientSpace\DistName\phi\lambda$
    such that $\lim_{n \to \infty}u^n_0=u_0$, then
    \begin{align}\label{eq:200}
      \lim_{n \to \infty} u^n_t &= u_t \quad \forevery t \ge 0 , \\
      \label{eq:cap2:16}
      \lim_{n \to  \infty}\phi(u_t^n)&=\phi(u_t) \quad \forevery t>0,\\
      \label{eq:80}
      % \qquad
      \lim_{n\to\infty}\MetricSlope{\phi}{u_t^n}&=\MetricSlope{\phi}{u_t}
      \quad \forevery t\in (0,+\infty)\setminus\mathcal C.
    \end{align} 
    Moreover, \eqref{eq:200}, \eqref{eq:cap2:16}, \eqref{eq:80} occur \emph{locally uniformly} in $ [0,+\infty) $, $ (0,+\infty) $, $ (0,+\infty)\setminus\mathcal C $, respectively, and 
    \begin{equation}\label{eq:80-bis}
    \MetricSlope{\phi}{u_t} \le \liminf_{n\to\infty} \MetricSlope{\phi}{u_t^n} \le \limsup_{n\to\infty} \MetricSlope{\phi}{u_t^n} \le \lim_{s \up t} \MetricSlope{\phi}{u_s} \quad \forevery t \in \mathcal{C} .
    \end{equation}
  \end{description}
\end{theorem}
\begin{proof}
  Let us consider each statement in turn. 
  
  \medskip
  
  \paragraph{\emph{$\boldsymbol \lambda$-contraction}} 
  Since $u^1,u^2$ are solutions of \eqref{eq:EVI}, 
  by applying \eqref{eq:cap1:78bis} with $u_t = u^1_t$ and $v=u^2_t$,
  after a further multiplication by $\rme^{\lambda(t-s)}$ we get
  \begin{equation}
    \label{eq:cap2:14}
     \frac {\rme^{2\lambda(t-s)}}2\DistSquare{u^1_t}{u^2_t}-
    \frac {\rme^{\lambda(t-s)}}2\DistSquare{u^1_s}{u^2_t}\le
    \rme^{\lambda(t-s)}\sfE\lambda{t-s}\Big(\phi(u^2_t)-\phi(u^1_t)\Big);
  \end{equation}
  analogously, by applying \eqref{eq:cap1:78bis} 
  with $u_t =u^2_t$ and $v =u^1_s$, we obtain
  \begin{equation}
    \label{eq:cap2:15}
    \frac {\rme^{\lambda(t-s)}}2\DistSquare{u^2_t}{u^1_s}-
    \frac 12\DistSquare{u^2_s}{u^1_s}\le
    \sfE\lambda{t-s}\Big(\phi(u^1_s)-\phi(u^2_t)\Big) .
  \end{equation}
  The sum of \eqref{eq:cap2:14} and \eqref{eq:cap2:15} yields 
  \begin{displaymath}
    \frac {\rme^{2\lambda(t-s)}}2\DistSquare{u^1_t}{u^2_t}-
    \frac 12\DistSquare{u^1_s}{u^2_s}\le
    \Big(\rme^{\lambda(t-s)}-1\Big)\sfE\lambda{t-s}
    \Big(\phi(u^2_t)-\phi(u^1_t)\Big)+
    \sfE\lambda{t-s}\Big(\phi(u^1_s)-\phi(u^1_t)\Big)
  \end{displaymath}
  and therefore, upon inverting the roles of $u^1, u^2$, multiplying by $ \rme^{2\lambda s} $ and again summing up, 
  \begin{equation}
    \label{eq:cap2:17}
    \rme^{2\lambda t}\DistSquare{u^1_t}{u^2_t}-
    \rme^{2\lambda s}\DistSquare{u^1_s}{u^2_s}\le
    \rme^{2\lambda s}
    \sfE\lambda{t-s}\Big(\phi(u^1_s)-\phi(u^1_t)
    +
    \phi(u^2_s)-\phi(u^2_t)\Big) .
  \end{equation}
  Dividing \eqref{eq:cap2:17} by $t-s$ and passing to the limit as $t\downarrow s$ (using the lower semicontinuity of $ t \mapsto \phi(u_t) $) we end up with 
% \cite[Lemma 4.3.4]{Ambrosio-Gigli-Savare05} we obtain that
%   the map $t\mapsto \DistSquare{u^1(t)}{u^2(t)}$ is
%   locally absolutely continuous in $(0,+\infty)$ and satisfies
%   \begin{align*}
%     \frac \d{\d t}\DistSquare{u^1(t)}{u^2(t)}&\le
%     \limsup_{h\down0}\frac{\DistSquare{u^1(t)}{u^2(t)}-
%       \DistSquare{u^1(t-h)}{u^2(t)}}h\\&+
%     \limsup_{h\down0}\frac{\DistSquare{u^1(t)}{u^2(t+h)}-
%       \DistSquare{u^1(t)}{u^2(t)}}h
%   \end{align*}
%   for $\Leb 1$-a.e. $t>0$. Thanks to \eqref{eq:EVI} we get
%   \begin{align*}
%     \limsup_{h\down0}\frac{\DistSquare{u^1(t)}{u^2(t)}-
%       \DistSquare{u^1(t-h)}{u^2(t)}}{2h}&\le
%     \phi(u^2(t))-\phi(u^1(t))-\frac
%     \lambda2\DistSquare{u^1(t)}{u^2(t)}\\
%      \limsup_{h\down0}\frac{\DistSquare{u^1(t)}{u^2(t+h)}-
%        \DistSquare{u^1(t)}{u^2(t)}}{2h}
%      &\le
%     \phi(u^1(t))-\phi(u^2(t))-\frac
%     \lambda2\DistSquare{u^1(t)}{u^2(t)}
%   \end{align*}
%   Summing up these two contributions, we finally obtain
  \begin{displaymath}
    \Urd\Big(\rme^{2\lambda t}\DistSquare{u^1_t}{u^2_t}\Big)\le 0
      \quad
      \text{for every }t>0,
  \end{displaymath}
  which yields \eqref{eq:cap1:81} by Lemma \ref{le:monotonicity} (recall the continuity of $ u^1_t,u^2_t $ down to $ t=0 $). 
  
  \medskip
  \paragraph{\emph{Regularizing effects I: solutions are locally Lipschitz}}
  By choosing $u^1_t=u_t$ and $u^2_t=u_{t+h}$ in \eqref{eq:cap1:81} (note that for every $h>0$ the curve $t\mapsto u_{t+h}$ is still a solution of $\EVI\AmbientSpace\DistName\phi\lambda$), we find that 
  \begin{displaymath}
    \text{the map} \
    t\mapsto \rme^{2\lambda t}\frac{\DistSquare{u_{t+h}}{u_t}}{h^2}
    \ \text{is nonincreasing for every } h>0,
  \end{displaymath}
 which together with \eqref{eq:29} yields for every $t>3t_0$ and $0<h<t_0$ (it suffices to consider the case $\lambda\le 0$) 
  \begin{equation}
    \label{eq:72}
    \rme^{2\lambda t}\frac{\DistSquare{u_{t+h}}{u_t}}{h^2}\le
    \frac{1}{t_0}\int_{t_0}^{3t_0-h} \frac{\DistSquare{u_{s+h}}{u_s}}{h^2}\,\d s\le
    \frac{C_{t_0,3t_0}}{t_0}. 
  \end{equation} 
Hence \eqref{eq:72} ensures that $ u $ is locally Lipschitz in $ (0,+\infty) $. 
  % We integrate
%   in the interval $(0,t)$
%   the following form of  \eqref{eq:EVI}
%   \begin{equation}
%     \label{eq:Cap1:36}
%     \frac \d{\d s}\frac{\rme^{\lambda s}}2\DistSquare{u_s}v
%     +\rme^{\lambda s}\phi(u_s)\le \rme^{\lambda s}\phi(v);
%   \end{equation}
%   recalling that $t\mapsto \phi(u_t)$ is
%   nonincreasing we get
%   \begin{displaymath}
%     \EExp\lambda t\phi(u_t)\le
%     \int_0^t \rme^{\lambda s}\phi(v)\,ds-
%     \int_0^t\frac{d}{ds}\frac{\rme^{\lambda s}}2\DistSquare{u_s}v\,ds
%     \le
%     \EExp\lambda t\phi(v)+
%     \frac 12\DistSquare{u_0}v.
%   \end{displaymath}
%  \paragraph{\emph{$\lambda$-contractive semigroup.}}

\medskip
  \paragraph{\emph{Right limits, energy identity and regularizing effects II at $ \boldsymbol{t>0} $}}
  By reasoning as above, estimate \eqref{eq:cap1:81} yields 
  \begin{equation} 
    \label{eq:Cap1:24}
    \Dist{u_{t+h}}{u_t}\le \rme^{-\lambda (t-t_0)}
    \Dist{u_{t_0+h}}{u_{t_0}}\quad
    \forevery\, 0\le t_0<t<+\infty.
  \end{equation}
  If we set 
  \begin{equation*}
    \label{eq:Cap1:27}
    \delta_+(t):=\limsup_{h\down0}\frac{\Dist{u_{t+h}}{u_t}}h,\quad
    \delta_-(t):=\liminf_{h\down0}\frac{\Dist{u_{t+h}}{u_t}}h
    \qquad \forevery t\ge0, 
  \end{equation*} 
  then from \eqref{eq:Cap1:24} we deduce that
  \begin{equation} 
    \label{eq:Cap1:28}
    \text{the map} \
    t \in[0,+\infty) \mapsto \rme^{\lambda t}\delta_+(t) \ \text{is nonincreasing}.
  \end{equation} 
  We denote by $\mathcal I$
  the subset of $(0,+\infty)$ where
  the metric derivative \eqref{eq:cap1:65} of $u_t$ exists finite.
 As $ u $ is locally Lipschitz, by Theorem \ref{thm:metric_derivative} we know that
  $\Leb 1\big((0,+\infty)\setminus 
  \mathcal I\big)=0$ and
  \begin{equation}
    \label{eq:Cap1:25}
    \delta_-(t)=\delta_+(t)=\MetricDerivative u(t)<+\infty
    \quad\forevery t\in \mathcal I;
  \end{equation}
  in particular,
  \eqref{eq:Cap1:28} and \eqref{eq:Cap1:25} guarantee that 
  $\delta_+(t) \le M_{t_0} < +\infty$ for every $t>t_0>0 $. We aim at showing that in fact 
  \begin{equation}
    \label{eq:Cap1:35}
    \delta_-(t)=\delta_+(t)=|\partial\phi|(u_t)=\GSlope\lambda\phi {u_t}\quad\forevery t>0.
  \end{equation} 
  Dividing \eqref{eq:cap1:78} by $t-s$, for every $v\in \Dom\phi $
  and $ 0 < s <t $ we get 
  \begin{equation*}
    \label{eq:cap2:2}
    -\frac{\Dist{u_{t}}{u_s}}{2(t-s)}
    \Big(\Dist{u_t}v+
    \Dist{u_{s}}v\Big)+
    \frac 1{t-s}\int_s^{t}
    \Big(\phi(u_{r})+
    \frac\lambda{2}\DistSquare{u_{r}}v\Big)
    \,\d r
    \le
    \phi(v).
  \end{equation*}
  Passing to the limit as $t\down s$
  and recalling \eqref{eq:cap2:1bis}, we obtain 
  \begin{equation*}
    \label{eq:Cap1:30}
    \phi(v)\ge \phi(u_s)
    -\delta_-(s)\Dist{u_s}v+
    \frac\lambda{2}\DistSquare{u_{s}}v\quad
    \forevery v\in \Dom\phi, \ s > 0, 
  \end{equation*}
  which upon recalling \eqref{eq:cap1:89} yields  
  \begin{equation} 
    \label{eq:Cap1:31}
    |\partial\phi|(u_s)\le
    \GSlope\lambda\phi{u_s}\le \delta_-(s)\le \delta_+(s)
    \quad\forevery  s > 0.
  \end{equation} 
In particular, $u_t\in \DomainSlope\phi$ for all $ t>0 $. Now let us fix $ s>0 $. We know that $ \delta_+(s)<+\infty $ and from \eqref{eq:Cap1:31} we can assume with no loss of generality that $\delta_+(s)>0$; 
  so, dividing \eqref{eq:cap1:70} by $h^2$ and rescaling
  the integrand, we infer that for all $ \varepsilon>0 $ and $ h $ small enough there holds
  \begin{equation}\label{eq:chain-ineq} 
  \begin{aligned}
    \frac1{2h^2}\DistSquare{u_{s+h}}{u_s}
    &\le
    \frac 1h\int_0^{1}
    \Big(\phi(u_s)-\phi(u_{s+h\rho})
    -
    \frac\lambda{2}\DistSquare{u_{s+h\rho}}{u_s}\Big)
    \d\rho
    \\&\le
    \int_0^1
    \frac{(\phi(u_s)-\phi(u_{s+h\rho}))^+}{\Dist{u_{s+h\rho}}{u_s}}\frac {\Dist{u_{s+h\rho}}{u_s}} {h\rho}\rho\,\d\rho-
    \frac\lambda2\int_0^1 \frac{\DistSquare{u_{s+h\rho}}{u_s}}{h}\,\d\rho \\
    &  \le
    \int_0^1
    \big( \MetricSlope \phi {u_s} + \varepsilon \big) \big( \delta_+(s) + \varepsilon \big) \rho \,\d\rho  +
    \frac{\lambda^-}2 \, \varepsilon \int_0^1   \left( \delta_+(s) + \varepsilon \right) \rho \,\d\rho .
  \end{aligned}
  \end{equation}
  Letting $h\down 0$ first and $ \varepsilon \downarrow 0 $ then, we therefore obtain 
  \begin{equation}
    \label{eq:Cap1:34}
    \frac12\delta^2_+(s)\le
    \frac12\MetricSlope\phi{u_s}\,\delta_+(s)\quad
    \forevery s > 0,
  \end{equation} 
  which yields \eqref{eq:Cap1:35} in view of \eqref{eq:Cap1:31} and \eqref{eq:cap2:3} (in the open interval $ (0,+\infty) $) in view of \eqref{eq:Cap1:28}: the right continuity of
  $t\mapsto \rme^{\lambda t}\MetricSlope\phi{u_t}$
  just follows from the continuity of $u_t$ and
  the lower semicontinuity of the global $ \lambda $-slope of $\phi$. 

  Since the map $t\mapsto\MetricSlope\phi{u_t}$
  is locally bounded in $(0,+\infty)$, inequality \eqref{eq:cap1:89} together with the fact that $ u $ is also locally Lipschitz show
  that the map $t\mapsto \phi(u_t)$ is in turn locally Lipschitz continuous.
  Hence by combining \eqref{eq:cap1:87}, \eqref{eq:cap1:89}, \eqref{eq:Cap1:25} and \eqref{eq:Cap1:35} we get:
  \begin{displaymath}
% \int_{t_0}^{t_1}\rme^{2\lambda (t-t_0)}
%     |\dot u|^2(t_0+)\,\d t
%     \le
    \int_{t_0}^{t_1} |\dot u_{t+}|^2 \, \d t\le
    \phi(u_{t_0})-
    \phi(u_{t_1})\le
    \MetricSlope\phi {u_{t_0}}\Dist{u_{t_0}}{u_{t_1}}
    -\frac\lambda2\DistSquare{u_{t_0}}{u_{t_1}} \quad \forevery t_1>t_0>0 ; 
  \end{displaymath}
  dividing by $t_1-t_0$ and passing to the limit
  as $t_1\down t_0$, since $t\mapsto |\dot u_{t+}|=
  \MetricSlope\phi{u_t}$ is
  right continuous
  we obtain
  \begin{equation}
    \label{eq:cap2:5}
    |\dot u_{t_0+}|^2 = \lim_{t_1\down t_0} 
    \frac{1}{t_1-t_0} \int_{t_0}^{t_1} |\dot u_{t+}|^2 \, \d t \le
    -\frac {\d}{\d t}\phi(u_{t})\restr{t=t_0+}
    \le
    \MetricSlope\phi{u_{t_0}}|\dot u_{t_0+}|,
  \end{equation}
which yields the first two equalities of \eqref{eq:Cap12:24}. Hence if $ \lambda \ge 0 $ the convexity of $ t \mapsto \phi(u_t) $ (at least in $ (0,+\infty) $) is just a consequence of the latter and the fact that $t\mapsto \rme^{\lambda t} \MetricSlope\phi{u_t}$ is nonincreasing. More in general, the function 
$$ t \mapsto \rme^{2\lambda t} \phi(u_t) - 2\lambda \int_0^t \rme^{2\lambda s} \phi(u_s) \, \d s $$
is convex, which yields the claim about the local semi-convexity of $
t \mapsto \phi(u_t) $ in the case $ \lambda<0 $.

Finally, in order to prove that
$\MetricSlope\phi{u_t}=\lMetricSlope\phi{u_t}$,
it is sufficient to observe that
\begin{displaymath}
  \lDist{u_t}{u_{t+h}}
  \le \int_{t}^{t+h}|\dot u_s|\,\d s = \int_{t}^{t+h}\MetricSlope\phi{u_s}\,\d s ,
\end{displaymath}
where $\lDistName$ is either the length distance induced by $\Dom\phi$ or by $\AmbientSpace$, so that if $\MetricSlope\phi{u_t}>0$
(otherwise there is nothing to prove) then
\begin{equation}\label{eq:LLL}
\begin{aligned}
 \lMetricSlope\phi{u_t}\ge &
  \limsup_{h\down0}\frac{\phi(u_t)-\phi(u_{t+h})}{\lDist{u_{t+h}}{u_t}}
  \ge
  \limsup_{h\down0}\left(\frac{\phi(u_t)-\phi(u_{t+h})}{\int_{t}^{t+h}\MetricSlope\phi{u_s}\,\d s}\right) \\
  = & \limsup_{h\down0}\left(\frac{\int_{t}^{t+h}\MetricSlopeSquare\phi{u_s}\,\d s}{\int_{t}^{t+h}\MetricSlope\phi{u_s}\,\d s}\right) 
  \stackrel{\text{H\"older}}\ge \limsup_{h \down 0} \sqrt{\frac{\int_{t}^{t+h}\MetricSlopeSquare\phi{u_s}\,\d s}{h}}  = \MetricSlope\phi{u_t}.
  \end{aligned}
\end{equation}
Since the converse inequality $\lPlainSlope\phi\le \PlainSlope\phi$ is
always true, we conclude. 
\medskip
\paragraph{\emph{Right limits, energy identity and regularizing effects II at $ \boldsymbol{t=0} $}}  
Through Remark \ref{rem: to-zero} we have already seen that $ t \mapsto \phi(u_t) $ is in fact nonincreasing down to $ t=0 $, which in particular ensures that it is continuous at $ t=0 $ as well and therefore convex in the whole $ [0,+\infty) $ if $ \lambda \ge 0 $. 

As for the energy identity \eqref{eq:Cap12:24}, let $ u_0 \not \in \DomainSlope \phi  $. By combining \eqref{eq:Cap1:28}, \eqref{eq:Cap1:35} and the lower semicontinuity of the global $\lambda$-slope, it follows that $ \delta_-(0) = \delta_+(0) = + \infty $, whence \eqref{eq:Cap1:35} also holds at $ t=0 $. If $ u_0 \not\in \Dom \phi $, it is apparent that $ \frac{\d}{\d t}\phi(u_{0+}) = -\infty $, and all the slopes at $ u_0 $ are by definition $+\infty$. On the other hand, if $ u_0 \in \Dom \phi $ the left-hand inequality in \eqref{eq:cap2:5} does hold at $ t_0=0 $, and the l.h.s.\ is $ +\infty $ since $ \lim_{t \downarrow 0} \MetricSlope{\phi}{u_t} = +\infty $. Similarly, we have that $ \lMetricSlope{\phi}{u_0} = +\infty $ because in this case \eqref{eq:LLL} still holds at $t=0$. Suppose now that $ u_0 \in \DomainSlope \phi $. In order to prove the validity of \eqref{eq:Cap12:24} at $ t=0 $ it suffices to show that $ \delta_+(0) $ is finite: then by arguing as above the key inequalities \eqref{eq:Cap1:31}, \eqref{eq:Cap1:34}, \eqref{eq:cap2:5} and \eqref{eq:LLL} still hold at zero. To this end, let us consider \eqref{eq:chain-ineq} at $ s=0 $, which in particular yields
\begin{equation*}\label{eq:slope:zero}
\DistSquare{u_{h}}{u_0}
    \le
    2\big( \MetricSlope \phi {u_0} + \varepsilon \big)
    \int_0^h
    \Dist{u_{\tau}}{u_0} \, \d\tau -
    \lambda \int_0^h \DistSquare{u_{\tau}}{u_0}\,\d\tau
    \le 4 C \, \int_0^h
    \Dist{u_{\tau}}{u_0} \, \d\tau
\end{equation*}
for a suitable $ C>0 $ independent of $ h $ small enough. Upon letting $ x(h):=\int_0^h \Dist{u_{\tau}}{u_0} \, \d\tau $, an elementary ODE argument shows that $ x(h) \le {C} h^2 $, whence $ \Dist{u_{h}}{u_0} \le {2C} h $ and the finiteness of $ \delta_+(0) $ is proved. We can then conclude that \eqref{eq:cap2:3} is true at $ t=0 $ as well.
    
\medskip   
  \paragraph{\emph{Left limits}}
  We already know that
  $t\mapsto \rme^{\lambda t}\MetricSlope\phi{u_t}$
  is finite, nonincreasing
  and right continuous in $(0,+\infty)$;
  let us therefore denote by $\mathcal C$
  its (at most countable) jump set,
  i.e.
   \begin{equation*}
    \label{eq:Cap1:53}
    (0,+\infty)\setminus \mathcal C=
    \Big\{t\in (0,+\infty): \,
    \lim_{h\down0}|\partial\phi|(u_{t-h})=|\partial\phi|(u_t)\Big\}.
  \end{equation*}
  If $t_0\in (0,+\infty)\setminus \mathcal C$ it follows that
  \begin{equation}
    \label{eq:Cap1:47bis}
    \frac \d{\d t}\phi(u_t)\restr{t=t_0-}=
    \lim_{h\down0}\frac{\phi(u_{t_0})-\phi(u_{t_0-h})}h
    =-|\partial\phi|^2(u_{t_0})
    \topref{eq:cap1:89}\ge -\MetricSlope\phi{u_{t_0}}
    \liminf_{h\down0}\frac{\Dist{u_{t_0}}{u_{t_0-h}}}h,
  \end{equation}
  and therefore
  \begin{equation}
    \label{eq:73}
    \liminf_{h\down0}\frac{\Dist{u_{t_0}}{u_{t_0-h}}}h\ge \MetricSlope\phi{u_{t_0}};
 \end{equation}
  on the other hand,
  by the $\Leb{1}$-a.e.~equality between $|\dot u|(t)$
  and $|\partial\phi|(u_t)$ we get
  \begin{equation}
    \label{eq:Cap1:54}
    \frac{\Dist{u_{t_0-h}}{u_{t_0}}}h\le \frac{1}{h}\int_{t_0-h}^{t_0}|\dot u_s|\,\d s=
    \frac{1}{h}\int_{t_0-h}^{t_0}|\partial\phi|(u_s)\,\d s \quad \Rightarrow \quad
    \limsup_{h\down0}\frac{\Dist{u_{t_0-h}}{u_{t_0}}}h\le
    \MetricSlope\phi{u_{t_0}}.
  \end{equation}
If $ t_0 \in \mathcal{C} $, just note that inequalities \eqref{eq:Cap1:47bis}--\eqref{eq:Cap1:54} still hold provided one replaces $ \MetricSlope\phi{u_{t_0}} $ with $ \lim_{h\down0}|\partial\phi|(u_{t_0-h}) $. 
  
\medskip   
 \paragraph{\emph{\emph{A priori} estimates}} 
%  Estimate 
%  \eqref{eq:Cap12:111} follows directly from \eqref{eq:cap1:78bis}
%  by choosing $s=0$ and neglecting the first nonnegative term.  
  In order to show \eqref{eq:2} and 
  \eqref{eq:Cap1:38}, we can apply \eqref{eq:cap2:3}, 
  the fact that $-\frac \d{\dt}\phi(u_t)=|\partial\phi|^2(u_t)$ outside $ \mathcal{C} $, and finally
  \eqref{eq:EVI} to obtain
  \begin{align*}
    \frac 12\big(\EExp\lambda t\big)^2|\partial\phi|^2(u_t)&
    =\frac12\big(\EExp{-\lambda}t\big)^2\rme^{2\lambda t}|\partial\phi|^2(u_t)
    \le
    \int_0^t \EExp{-\lambda} s \rme^{-\lambda s}
    \rme^{2\lambda s}|\partial\phi|^2(u_s)\,\d s
    \\&
    =-\int_0^t  \EExp{-\lambda} s \rme^{\lambda s}
    \Big(\phi(u_s)-\phi(u_t)\Big)'\,\d s
    =
    \int_0^t \rme^{\lambda s}\Big(\phi(u_s)-\phi(u_t)\Big) \d s
    \\& \topref{eq:EVI}\le
    \int_0^t -\frac 12\Big(\rme^{\lambda s}\DistSquare{u_s}v\Big)'
    +\rme^{\lambda s}\Big(\phi(v)-\phi(u_t)\Big) \d s
    \\&=
    \frac 12\DistSquare {u_0}v+\EExp\lambda t\Big(\phi(v)-\phi(u_t)\Big)-
    \frac {\rme^{\lambda t}}2\DistSquare {u_t}v,
  \end{align*} 
  where $ \tfrac{\d}{\d s} $ has been replaced by $ ^\prime $ for notational convenience. This proves \eqref{eq:2}. If $v \in \Dom{\GSlopeE\lambda\phi} $, thanks to \eqref{eq:cap1:89}
  and the Cauchy-Schwarz inequality, 
  we can bound difference $\phi(v)-\phi(u_t)$ by
  \begin{align*}
    \EExp\lambda t\big(\phi(v)-\phi(u_t)\big)
    &\le \EExp\lambda t\Big(\GSlope\lambda \phi v \, \Dist{u_t}{v}-\frac\lambda 2
    \DistSquare{u_t}{v}\Big)
    \\&\le \frac{(\EExp\lambda t)^2}{2(2\rme^{\lambda t}-1)}
    \GSlopeSquare\lambda \phi v+
    \frac {2\rme^{\lambda t}-1}2\DistSquare{u_t}{v}-
    \frac{\rme^{\lambda t}-1}2\DistSquare {u_t}{v}
    \\&=
    \frac{(\EExp\lambda t)^2}{2(2\rme^{\lambda t}-1)}
    \GSlopeSquare\lambda \phi v +
    \frac{\rme^{\lambda t}}2\DistSquare{u_t}v,
  \end{align*}
  at least when $2\rme^{\lambda t}>1$.
  Substituting this bound in \eqref{eq:2}
  we obtain \eqref{eq:Cap1:38}. 
%   In order to prove \eqref{eq:cap2:12}
%   we combine \eqref{eq:cap1:13} with
%   \eqref{eq:EVI}, obtaining
%   \begin{align*}
%     \frac 12\frac{\d}{\d t}\DistSquare{u_t}{u_o}+
%     \frac{\lambda+\lambda_o}2 \DistSquare{u_t}{u_o}
%     \le \phi(u_o)-m_o,
%   \end{align*}
%   and we integrate in time after a multiplication by $e^{(\lambda+
%     \lambda_o)t}$.
%   \eqref{eq:cap2:12bis} in
%   the case $\lambda_o=0$ follows easily from the energy identity
%   \eqref{eq:Cap12:24}.

\medskip 
  \paragraph{\emph{Asymptotic expansions}}
  In order to show the validity of \eqref{eq:3} we multiply
  \eqref{eq:cap2:13} by $ \rme^{\lambda t} $, integrate and use the estimate
    \begin{displaymath}
      \begin{aligned}
        -\int_0^{t}\rme^{2\lambda r} \phi(u_r)\,\d r&=
        \int_0^{t}\big(\sfE{2\lambda}{t}-
        \sfE{2\lambda}{r}\big)\MetricSlopeSquare\phi{u_r}\,\d
        r - \sfE{2\lambda}{t}\phi(u_0)\\
      &\le
      \MetricSlopeSquare\phi{u_0}
      \int_0^{t}
      \big(\sfE{2\lambda}{t}-
      \sfE{2\lambda}{r}\big)\rme^{-2\lambda r}\,\d r-
      \sfE{2\lambda}{t}\phi(u_0)
   \end{aligned}
  \end{displaymath}
along with the elementary inequality (for $ \lambda \le 0 $)
  \begin{align*}
    \big(\sfE{2\lambda}{t}-
    \sfE{2\lambda}{r}\big)\rme^{-2\lambda r}\le (t-r)\quad \forevery r\in [0,t].
  \end{align*} 
  
The fact that $ x \mapsto \phi(x) - \frac{\lambda}{2} \DistSquare x o  $ is linearly bounded from below is a simple consequence of  \eqref{eq:cap1:89} and the energy identity \eqref{eq:Cap12:24}, which in particular ensures that $ \Dom{\GSlopeE\lambda\phi}$ is not empty.

In order to prove \eqref{eq:3-lessreg}, it is convenient to use an approach by \emph{curves of maximal slope} (see the next Section \ref{sec:max-slope}). Namely, first of all one observes that the energy identity \eqref{eq:Cap12:24} implies 
\begin{equation}\label{eq:max-slope-1}
\frac{\rme^{\lambda s}}{2} \left| \dot u_s \right|^2 + \frac{\rme^{\lambda s}}{2} \MetricSlopeSquare \phi {u_s} - \lambda \rme^{\lambda s} \phi({u_s}) \le - \frac{\d}{\d s} \big( \rme^{\lambda s} \phi(u_{s}) \big)  \quad \forevery s \in (0,+\infty) \setminus \mathcal{C} .
\end{equation}
An integration of \eqref{eq:max-slope-1} yields
\begin{equation}\label{eq:max-slope-2}
%\begin{aligned}
\frac{1}{2} \int_0^t \left| \dot u_s \right|^2 \rme^{\lambda s} \d s  + \frac{1}{2} \int_0^t \MetricSlopeSquare \phi {u_s} \, \rme^{\lambda s} \d s - \lambda \int_0^t \rme^{\lambda s} \phi({u_s}) \, \d s \\
\le \phi(u_0)- \rme^{\lambda t} \phi(u_t) \quad \forevery t > 0 .
%\end{aligned}
\end{equation}
Since H\"older's inequality and the definition of metric derivative entail
\begin{equation*}\label{eq:max-slope-3}
\DistSquare{u_t}{u_0} \le \left( \int_0^t \rme^{-\lambda s} \, \d s \right) \left( \int_0^t \left| \dot u_s \right|^2 \rme^{\lambda s} \d s \right) =  \sfE{-\lambda}{t} \, \int_0^t \left| \dot u_s \right|^2 \rme^{\lambda s} \d s  ,
\end{equation*} 
from \eqref{eq:max-slope-2} there follows, for every $ t>0 $,
\begin{equation}\label{eq:max-slope-4}
\frac{1}{2\sfE{-\lambda}{t}} \DistSquare{u_t}{u_0} + \rme^{\lambda t} \phi(u_t) + \frac{1}{2} \int_0^t \MetricSlopeSquare \phi {u_s} \, \rme^{\lambda s} \d s - \lambda \int_0^t \rme^{\lambda s} \phi({u_s}) \, \d s \le \phi(u_0) .
\end{equation}
Recalling the definition of Moreau-Yosida approximation \eqref{eq:1}, by virtue of \eqref{eq:max-slope-4} we deduce 
\begin{equation}\label{eq:max-slope-5}
\int_0^t \MetricSlopeSquare \phi {u_s} \, \rme^{\lambda s} \d s - 2 \lambda \int_0^t \rme^{\lambda s} \phi({u_s}) \, \d s \le 2 \Big( \phi(u_0) - \rme^{\lambda t} \, \phi_{\sfE{\lambda}{t}}(u_0) \Big) \quad  \forevery t > 0 .
\end{equation} 
Note that, $ \phi $ being quadratically bounded from below for all $ \kappa_o > -\lambda $ and $ \lambda \sfE{\lambda}{t} > -1 $ for all $ t>0 $, the r.h.s.\ of \eqref{eq:max-slope-5} is always finite. By exploiting again the energy identity, we can integrate by parts the first term in l.h.s.\ of \eqref{eq:max-slope-5} to get 
\begin{equation}\label{eq:max-slope-6} 
-\rme^{\lambda t} \phi(u_{t}) - \lambda \int_0^t \rme^{\lambda s} \phi(u_{s}) \, \d s \le  \phi(u_0) - 2 \rme^{\lambda t} \, \phi_{\sfE{\lambda}{t}}(u_0) \quad  \forevery t > 0 .
\end{equation} 
Hence by integrating the differential inequality \eqref{eq:max-slope-6} w.r.t.\ the unknown $ t \mapsto - \int_0^t \rme^{\lambda s} \phi(u_{s}) \, \d s $, we end up with 
\begin{equation}\label{eq:max-slope-7} 
\begin{aligned}
- \int_0^t \rme^{\lambda s} \phi(u_{s}) \, \d s \le & \rme^{-\lambda t} \int_0^t \rme^{\lambda s} \Big( \phi(u_0) - 2 \rme^{\lambda s} \, \phi_{\sfE{\lambda}{s}}(u_0) \Big) \d s \\
= & - \sfE{\lambda}{t} \phi(u_0) + 2 \rme^{-\lambda t} \int_0^t \rme^{2\lambda s} \Big( \phi(u_0) - \phi_{\sfE{\lambda}{s}}(u_0) \Big) \d s \qquad \forevery t>0 .
\end{aligned}
\end{equation} 
Estimate \eqref{eq:3-lessreg} is therefore a consequence of \eqref{eq:max-slope-7} combined with the integral version of \eqref{eq:cap2:13}.

 \medskip
  \paragraph{\emph{Stability}}
The convergence of $ \{ u^n_t \} $ to $ u_t $ for all $ t>0 $ along with the fact that the latter is locally uniform in $ [0,+\infty) $ are immediate consequences of \eqref{eq:cap1:81}. Moreover, \eqref{eq:cap2:3} and \eqref{eq:Cap1:38} ensure that the sequence $ \MetricSlope{\phi}{u^n_t} $ is uniformly (in $n$) bounded by a constant
  $M_t<+\infty$ for every $t>0$. Hence thanks to \eqref{eq:Cap12:24} and \eqref{eq:cap1:89} we get
  \begin{displaymath}
    \phi(u_t)\ge \phi(u^n_t) - M_t \, \Dist{u_t}{u^n_t} -
    \frac \lambda 2 \DistSquare{u_t}{u^n_t},
  \end{displaymath}
  so that $\limsup_{n\to\infty} \phi(u^n_t) \le \phi(u_t)$. On the other hand, $\phi$ being lower semicontinuous, we also have that $ \liminf_{n\to\infty} \phi(u^n_t) \ge \phi(u_t) $, which finally implies \eqref{eq:cap2:16}. Such convergence is locally uniform in $ (0,+\infty) $ since $ \{ \phi(u^n_t) \} $ is a sequence of nonincreasing functions converging pointwise to the continuous function $ \phi(u_t) $.

Let us turn to slopes. By virtue of \eqref{eq:cap2:3} and \eqref{eq:Cap12:24}, it is straightforward to deduce the following inequalities:
\begin{equation}\label{ineq:slope-1-pre}
 \frac{\phi(u^n_t)-\phi(u^n_{t+\tau})}{\sfE{-2\lambda}{\tau}} \le \MetricSlopeSquare{\phi}{u^n_t} \le \frac{\phi(u^n_{t-\tau})-\phi(u^n_t)}{\sfE{2\lambda}{\tau}} \quad \forevery t>0, \ \tau \in (0,t) ;
\end{equation}
\eqref{eq:80} and \eqref{eq:80-bis} can therefore be proved by letting first $ n \to \infty $ in \eqref{ineq:slope-1-pre} (using \eqref{eq:cap2:16}) and then $ \tau \down 0 $, exploiting the energy identities both for right and left limits. The reason why \eqref{eq:80} occurs locally uniformly in $ (0,+\infty) \setminus \mathcal{C} $ is again a consequence of the fact that $ \{ \rme^{\lambda t} \MetricSlope{\phi}{u^n_t} \} $ is a sequence of nonincreasing functions converging pointwise in $ (0,+\infty) \setminus \mathcal{C} $ to $ \rme^{\lambda t} \MetricSlope{\phi}{u_t}  $, which is continuous in $ (0,+\infty) \setminus \mathcal{C} $. 
%\Modified{
% As $ t \mapsto \rme^{\lambda t}|\partial\phi|(u_{t,n})$
% is nonincreasing for all $ n $ and uniformly bounded for positive times, applying Helly's Theorem it is not restrictive to assume
%  that the sequence $ |\partial\phi|(u_{t,n}) $ converges pointwise everywhere in $ (0,+\infty) $ to some limit $ \mathcal{H}(t) $. By passing to the limit in the energy identities (for $ u_n $) we get
%  \begin{displaymath}
%    \int_s^t \mathcal{H}^2(r) \,\d r=\phi(u_{s})-
%    \phi(u_{t})\qquad \forevery
%    0<s<t<+\infty .
%  \end{displaymath}
%In particular, $ \mathcal{H}(t) = |\partial\phi|(u_{t}) $ for $ \Leb 1 $-a.e.\ $ t>0 $. Since $ e^{\lambda t}\mathcal{H}(t) $ is nonincreasing and $ |\partial\phi|(u_{t}) $ is continuous outside $ \mathcal{C} $, equality holds in $ (0,+\infty) \setminus \mathcal{C} $.  
% }

\medskip 
  \paragraph{\emph{Asymptotic behaviour}} When $\lambda>0$, \eqref{eq:Cap1:58} and \eqref{eq:Cap1:24} ensure
  that the sequence $k\mapsto u_k$ belongs to a
  fixed sublevel of $\phi$ and
  satisfies the Cauchy condition
  in $\AmbientSpace$, since
  \begin{equation*}
    \label{eq:Cap1:39}
    \Dist{u_{k+1}}{u_k}\le \rme^{-\lambda}
    \Dist{u_{k}}{u_{k-1}};
  \end{equation*}
thus, it is convergent to some limit $ \bar{u} \in \Dom{\phi} $. If we multiply \eqref{eq:2} by $ \rme^{-\lambda t} $, let $ t=k $ and $ k \to \infty $, thanks to the lower semicontinuity of $ \phi $ we deduce that the constant curve $ t \mapsto \bar{u} $ solves \eqref{eq:EVI} (if we let $ v=u_t $ we get the left-hand inequality in \eqref{eq:Cap1:42}), which ensures that $ \bar{u} $ is the unique minimum point for $ \phi $, along with the validity of \eqref{eq:Cap12:18}.  
  Inequality \eqref{eq:Cap1:41} is just \eqref{eq:cap2:3}
  and \eqref{eq:Cap1:38} with
  $v=\bar u $ (note that $\MetricSlope\phi{\bar u}=\GSlope \lambda \phi {\bar{u}}=0$).
  In order to prove the right-hand inequality in \eqref{eq:Cap1:42}, observe that
  \eqref{eq:cap1:89} and Young's inequality yield
  \begin{equation}
    \label{eq:Cap1:43}
    \phi(u_t)-\phi(\bar u)\le
    \MetricSlope \phi {u_t} \, \Dist {u_t}{\bar u}-
    \frac\lambda2\DistSquare {u_t}{\bar u}\le
    \frac1{2\lambda} \MetricSlopeSquare \phi {u_t}.
  \end{equation}
%  For the opposite inequality,
%  since $\MetricSlope\phi{\bar u}=0$,
%  from \eqref{eq:cap1:89} we easily get
%  \begin{equation}
%    \label{eq:Cap1:44}
%    \phi(u)-\phi(\bar u)\ge
%    \frac\lambda2\DistSquare u{\bar u}.
%  \end{equation}
  The first estimate of \eqref{eq:Cap1:40} now 
  follows upon noticing that \eqref{eq:Cap1:43} yields
  \begin{equation*}
    \label{eq:Cap1:45}
    \frac{\d}{\d t}\Big(\phi(u_t)-\phi(\bar u)\Big)=-|\partial\phi|^2(u_t)
    \le -2\lambda \Big(\phi(u_t)-\phi(\bar u)\Big);
  \end{equation*}
  on the other hand, the second estimate of \eqref{eq:Cap1:40} easily follows from \eqref{eq:2} with $v=\bar u$.  
  
If $ \lambda=0 $ and $ \bar{u} $ is a minimum point of $ \phi $, the map $ t \mapsto \bar{u} $ trivially solves \eqref{eq:EVI}, so that $t\mapsto \Dist{u_t}{\bar u}$ is nonincreasing by virtue of \eqref{eq:cap1:81}. The two estimates in \eqref{eq:Cap12:19} follow similarly from \eqref{eq:2} and \eqref{eq:Cap1:38}, upon observing again that $ \GSlope  0 \phi {\bar{u}} = 0 $. If in addition $ \phi $ has compact sublevels, then by \eqref{eq:2} and the lower semicontinuity of $ \phi $ we easily deduce that there exists a sequence $ \{ u_{t_n} \} $ converging to some $ \bar u $ such that
$$
\phi(\bar u) \le \liminf_{n\to\infty} \phi(u_{t_n}) \le \phi(v) \quad \forevery v \in \AmbientSpace , 
$$
whence $ \bar{u} $ is a minimum point of $ \phi $. On the other hand, in this case we know that $ t \mapsto \Dist{u_t}{\bar u} $ is nonincreasing, so that the whole curve $ u_t $ is forced to converge to $ \bar{u} $ as $ t \to +\infty $. 
\end{proof} 
% \begin{remark}
%   \label{rem:Reglambda}
%   \upshape
%   Under the same assumption of Theorem \ref{thm:main1}, we easily get
%   by \eqref{eq:Cap12:24} that
%   \begin{equation}
%     \label{eq:cap2:18}
%     \MetricSlope\phi u=\GSlope\lambda\phi u\quad
%     \forevery\, u\in \overline D.
%   \end{equation}
%   In particular, if $D$ is dense in $\Dom\phi $ then
%   $\phi$ is a \Reg\lambda functional.
% \end{remark}

As a consequence of Theorem \ref{thm:main1}, we have a finer equivalence between the notions of slope on points that belong to the domain of the gradient flow, which is basically a consequence of the fact that solutions of the \EVIname\ are curves of maximal slope (see Section \ref{sec:max-slope} below). 
\begin{proposition}[Equivalence of slopes]\label{char-slope}
Let $\phi : X \rightarrow (-\infty,+\infty] $ be a proper l.s.c.~functional which admits an \EVIshort\lambda-gradient flow in $ D \not \equiv \emptyset $.
Then 
\begin{equation}\label{eq:slopes-semi}
\underset{ \substack{  y \to x  \\ y \in D \cap \DomainSlope{\phi} }
}{  \limsup  } \, \frac{\left( \phi(x) - \phi(y)
  \right)^+}{\Dist{x}{y}} =
\MetricSlope\phi x=
\lMetricSlope\phi x=\GSlope{\lambda}{\phi}{x} \quad \forevery x \in D .
\end{equation}
\end{proposition}
\begin{proof}
 The identity $\MetricSlope\phi x=
\lMetricSlope\phi x=\GSlope{\lambda}{\phi}{x}$ is an immediate
consequence
of Theorem \ref{thm:main1} (precisely, the statement immediately below
\eqref{eq:Cap12:24}).

In order to prove that these quantities also coincide with the
left-hand side of \eqref{eq:slopes-semi},
it is enough to notice that, in the proof of Theorem \ref{thm:main1}, when showing the energy identity \eqref{eq:Cap12:24} (see in particular inequalities \eqref{eq:chain-ineq}--\eqref{eq:Cap1:34}) the metric slope at $ x = u_s $ can in fact be replaced by the $ \limsup $ of $ (\phi(u_s) - \phi(u_t)) / \Dist{u_s}{u_t} $ as $ t \downarrow s $, which is clearly bounded from above by the l.h.s.~of \eqref{eq:slopes-semi}. 

Alternatively, one could just observe that $ u_s $ is also a solution of the \EVIname\ associated with $ {\phi}_D $, the latter functional being the same as $ \phi $ on $ D $ and $ +\infty $ elsewhere, so that \eqref{eq:slopes-semi} is a direct consequence of both the energy identities for $\EVI\AmbientSpace\DistName\phi\lambda$ and $\EVI\AmbientSpace\DistName{\phi_D}\lambda$. Here the fact that $ \phi_D $ may not be lower semicontinuous on $ \Dom{\phi} \setminus D $ is inessential.
\end{proof}

\begin{remark}[$ \EVIname $, slopes and Moreau-Yosida regularizations]\label{rem-apriori}\upshape 
Because the function $ \tau \mapsto \phi_\tau (u_0) $ is nonincreasing, one can simply bound the integral remainder in the r.h.s.\ of \eqref{eq:3-lessreg} by the \emph{pointwise} remainder 
\begin{equation}\label{eq:lessreg-comm}
2 \rme^{-\lambda t} \sfE{2\lambda}{t} \Big( \phi(u_0)-\phi_{\sfE{\lambda}{t}}(u_0) \Big) .
\end{equation} 
On the other hand, from \cite[Lemma 3.1.5]{Ambrosio-Gigli-Savare08} we have the identity 
\begin{equation}\label{eq:duality-slope}
\limsup_{\tau \down 0} \frac{\phi(u_0)-\phi_\tau(u_0)}{\tau} =  \frac{1}{2} \MetricSlopeSquare{\phi}{u_0} ,
\end{equation}
so that if $ u_0 \in \DomainSlope \phi $ then \eqref{eq:lessreg-comm} reproduces (as $ t \downarrow 0 $) the remainder in \eqref{eq:3} up to a factor $2$. However, in order to get asymptotically the same estimate (i.e.~with the same factor), it is essential to keep such remainder in the integral form \eqref{eq:3-lessreg}. Indeed, from \eqref{eq:3-lessreg} itself we easily deduce (use the fact that $ v $ is arbitrary)
\begin{equation}\label{eq:est-yosida}
\sfE{\lambda}{t} \Big( \phi(u_0)-\phi_{\sfE{\lambda}{t}}(u_0) \Big)  \le 2 \rme^{-\lambda t} \int_0^t \rme^{2\lambda s} \Big( \phi(u_0) - \phi_{\sfE{\lambda}{s}}(u_0) \Big) \d s ,
\end{equation} 
which implies in turn that
\begin{equation}\label{eq:est-yosida-bis}
\text{the map} \quad t\in(0,+\infty) \mapsto \frac{\int_0^t \rme^{2\lambda s} \left( \phi(u_0) - \phi_{\sfE{\lambda}{s}}(u_0) \right) \d s}{(\sfE{\lambda}{t})^2} \quad \text{is nonincreasing} .
\end{equation}
Hence by gathering \eqref{eq:est-yosida-bis} and \eqref{eq:duality-slope}, we end up with
\begin{equation}\label{eq:199}
\frac{\int_0^t \rme^{2\lambda s} \left( \phi(u_0) - \phi_{\sfE{\lambda}{s}}(u_0) \right) \! \d s}{(\sfE{\lambda}{t})^2} \le \limsup_{h \down 0} \frac{\int_0^h \rme^{2\lambda s} \left( \phi(u_0) - \phi_{\sfE{\lambda}{s}}(u_0) \right) \! \d s}{(\sfE{\lambda}{h})^2} \le \frac{1}{4} \MetricSlopeSquare{\phi}{u_0} ;
\end{equation}
in particular (compare with \cite[Theorem 3.1.6]{Ambrosio-Gigli-Savare08} in the $ \lambda $-convex case), \eqref{eq:duality-slope}, \eqref{eq:est-yosida}, \eqref{eq:199} yield
\begin{equation*}\label{sup-tau-yosida}
\sup_{\tau>0: \, \lambda \tau > -1} \, (1+\lambda \tau) \, \frac{\phi(u_0)-\phi_\tau(u_0)}{\tau} = \frac{1}{2} \MetricSlopeSquare{\phi}{u_0} . 
\end{equation*}
Nevertheless, we point out that in order to obtain the correct asymptotic expansion for $ \Dist{u_t}{u_0} $ (let $ \lambda=0 $ for simplicity) it is enough to use the energy identity as follows:
$$
\DistSquare{u_t}{u_0} \le t \int_0^t \left| \dot{u}_s \right|^2 \d s = t \left( \phi(u_0)-\phi(u_t) \right) ,
$$
from which it is easy to deduce that
$$
\frac{1}{2} \DistSquare{u_t}{u_0} \le t \left( \phi(u_0)-\phi_t(u_0) \right) .
$$
In any case, the integral remainder is necessary in \eqref{eq:3-lessreg}, which holds for all $ v \in \Dom{\phi} $.
\end{remark} 

\begin{corollary}[Construction of \EVIshort\lambda-gradient flows]\label{cor:gf} 
  Suppose that for every initial value 
  $u_0\in D_0\subset \overline{\Dom\phi}$ there exists a solution $u_t$ of $\EVI X\DistName\phi\lambda$ such that
  $\lim_{t\down 0}u_t=u_0$.
  Then, if we set
  \begin{equation*}
    \label{eq:30}
    D:=\bigcup_{t\ge 0, \, u_0\in D_0}\Big\{u_t: \, u \text{ is the solution of \eqref{eq:EVI} starting from $u_0$} \Big\},
  \end{equation*}
  there exists a unique \EVIshort\lambda-gradient flow  $\FlowName_t:D\to D$ of $\phi$
  according to Definition \ref{def:GFlow}, which satisfies  
  the $\lambda$-contraction property 
    \begin{equation}
      \label{eq:Cap12:1bis}
     \Dist{\Flow t{u_0}}{\Flow t{v_0}}\le
     \rme^{-\lambda t}\Dist{u_0}{v_0}\qquad
     \forevery u_0,v_0\in {D} , \ t \ge 0 .
    \end{equation}
    If in addition $\phi$ has complete sublevels (in particular, if $\AmbientSpace$ is complete), then
 $\FlowName_t$ can always be extended by density to an \EVIshort\lambda-gradient flow of $\phi$ in $\overline D$.
\end{corollary}
\begin{proof}
  We just check the last statement, since the other ones are straightforward consequences of the $ \lambda $-contraction and uniqueness properties entailed by Theorem \ref{thm:main1}.
  Thanks to the latter, it is easy to extend $\FlowName_t$ from $D$ to its closure:
  for every $u_0\in \overline D$ one can simply take an 
  arbitrary sequence $ \{ u_0^n \} \subset D$ converging to $u_0$ and
  set $\Flow t{u_0}:=\lim_{n\to\infty}\Flow t{u_0^n}$ at every $ t>0 $. 
  The limit does exist since $ \{ \Flow t{u_0^n} \} $ is a Cauchy sequence 
  by \eqref{eq:cap1:81},
  and in view of \eqref{eq:2} 
  \begin{displaymath} 
    %\frac {\rme^{\lambda(t-s)}}2\DistSquare{u_t}v-
    \phi(\Flow t{u_0^n})\le 
    \frac 1{2 \sfE\lambda{t}}\DistSquare{u_0^n}v+
    \phi(v)\quad \forevery v\in \Dom\phi,
  \end{displaymath}
  so that it also
  belongs to a fixed sublevel of $\phi$, which is complete by assumption.  
  Estimate \eqref{eq:cap1:81} itself shows that the limit is independent
  of the chosen sequence $ \{ u_0^n \} $. Thanks to the lower semicontinuity of $ \phi $, it is immediate to verify
  that each trajectory $u_t$ of the extended flow still satisfies e.g.~the integral
  formulation \eqref{eq:cap1:78bis} down to $ s=0 $. Finally, the semigroup property \eqref{eq:18} is inherited as well by the extended flow since, as mentioned above, the limit trajectory is independent of the particular sequence $ \{ u_0^n \} $.
\end{proof}

% \begin{remark}[A stronger metric in $ \Dom\phi $]
%   \upshape
%   The set $\Dom\phi $ can be endowed with the stronger energy-metric
%   \begin{equation}
%     \label{eq:cap2:19}
%     \Distphi xy:=\Dist xy+|\phi(x)-\phi(y)| \quad \forevery x,y \in \Dom{\phi} ,
%   \end{equation}
%   so that $\big(\Dom\phi ,\DistphiName\big)$ is also a
%   metric space. When the sublevels of $ \phi $ are complete, such space is complete if and only if the restrictions of $ \phi $ on all of its sublevels are $ \DistName $-continuous. By virtue of \eqref{eq:cap2:16},
%  it turns out that $\FlowName_t$ is a continuous map from $ ( D , \DistName )$ to
%   $\big(\Dom\phi ,\DistphiName\big)$ for every $t>0$.
% \end{remark} 

\subsection[$\EVIshort\lambda$-gradient flows,
  the length distance, and $  \lambda$-convexity]{$\boldsymbol{ \EVIshort\lambda}$-gradient flows,
  the length distance, and $ \boldsymbol \lambda$-convexity}
\label{subsec:lconvexity}

Let us first recall a result of \cite[Theorem 3.5]{Ambrosio-Erbar-Savare16}.
\begin{theorem}[Self-improvement of $\EVIshort\lambda$-gradient flows]
\label{thm:self-improvement}
Let $(\sfS_t)_{t\ge0}$ be an $\EVIshort\lambda$-gradient flow of $\phi$ in $ \AmbientSpace = \overline{\Dom\phi}^{\lDistName}$. Then $( \FlowName_t)_{t\ge 0} $ is also an $\EVIshort\lambda$-gradient flow in $ \AmbientSpace $ w.r.t.~the length distance $\lDistName$ on each equivalence class defined by $\sim_\ell$ in \eqref{eq:lequivalence}.
\end{theorem}
The restriction to equivalence classes in the previous Theorem is due
to the fact that our definition of \EVIshort\lambda-gradient flow
just refers to distances. If $\AmbientSpace$ is Lipschitz connected, then
$\AmbientSpace$ contains only one equivalence class of $\sim_\ell$
and therefore $(\sfS_t)_{t\ge0}$ is an $\EVIshort\lambda$-gradient flow
of $\phi$ in $X$ w.r.t.~$\lDistName$. Note that Theorem \ref{thm:self-improvement}, as a byproduct, provides an alternative way to prove the identity $ \MetricSlope{\phi}{u_t} = \lMetricSlope{\phi}{u_t} $ (at least under the corresponding assumptions).

A result of \cite{Daneri-Savare08} (see Theorem 3.2 there)
shows, in particular, that if $ D \subset \overline{\Dom{\phi}} $ is a
geodesic subset and $\phi$ admits an \EVIshort\lambda-gradient flow in $ D $, then $\phi$ is strongly $\lambda$-convex in $ D $. An analogous property is enjoyed by approximate length subsets. 

\begin{theorem}[$\EVIshort\lambda$-gradient flows entail $\lambda$-convexity] \label{thm:Daneri}
  Suppose that $\phi$ admits an \EVIshort\lambda-gradient flow
  in the subset $D\subset \overline{\Dom\phi}$. Then the following hold: 
  \begin{enumerate}[(1)]
  \item \label{e-a} if $ \Geodesic_0,\Geodesic_1 \in D \cap \Dom{\phi} $, $ \vartheta,\eps \in (0,1) $ and $\Geodesic_{\vartheta,\eps} \in D $ is a $(\vartheta,\eps)$-intermediate point between them,
    \begin{equation}
      \label{eq:94}
      \phi(\FlowName_t(\Geodesic_{\vartheta,\eps})) \le (1-\vartheta)\phi(\Geodesic_0)+\vartheta\phi(\Geodesic_1)
      -\frac 12 \left(\lambda-\frac{\eps^2}{\sfE{\lambda}t}\right)\vartheta(1-\vartheta)\DistSquare{\Geodesic_0}{\Geodesic_1} \quad \forevery t>0; 
    \end{equation}
  \item \label{e-b} if $D$ is an approximate length subset then $\phi$
    is approximately  $\lambda$-convex in $D$;
  \item \label{e-new}
    if $D=X=\overline{\Dom\phi}^{\lDistName}$, then $\phi$ is approximately $\lambda$-convex w.r.t.~the distance
    $\lDistName$ on each equivalence class of $\sim_\ell$;
  \item \label{e-c} if $\Geodesic_0,\Geodesic_1 \in D \cap \Dom\phi$ and $\Gdc \in \GeoCon D{\Geodesic_0}{\Geodesic_1}$
    then $\Gdc\in \GeoConPhi D\lambda{\Geodesic_0}{\Geodesic_1}$; in particular, if
    for every $\Geodesic_0,\Geodesic_1\in D\cap\Dom\phi$ the set $\GeoCon D{\Geodesic_0}{\Geodesic_1}$
    is not empty, then $\phi$ is strongly $\lambda$-convex in $D$.
  \end{enumerate}
\end{theorem}
\begin{proof}
Property \eqref{e-c} follows from \eqref{e-b}, and in fact had already
been established in \cite[Theorem 3.2]{Daneri-Savare08}, from which
one can borrow tools to prove \eqref{e-a}. The main idea consists in
evaluating \eqref{eq:2} with $ u_t =
\FlowName_t(\Geodesic_{\vartheta,\eps}) $ and $ v=\Geodesic_0 $,
multiply it by $ (1-\vartheta) $ and sum it to the analogous estimate
one obtains by plugging $ v=\Geodesic_1 $ instead, multiplied by $
\vartheta $. By exploiting \eqref{eq:89}, it is then not difficult to
deduce the validity of \eqref{eq:94}.

As for \eqref{e-b}, just note that $ \Dist{\FlowName_t(\Geodesic_{\vartheta,\delta})}{\Geodesic_i} \le
\rme^{-\lambda t} \Dist{\Geodesic_{\vartheta,\delta}}{\Geodesic_i} +
\Dist{\FlowName_t(\Geodesic_i)}{\Geodesic_i} $ ($ i=0,1 $). Hence for
any fixed $ \eps \in (0,1) $, one sees that $
\FlowName_t(\Geodesic_{\vartheta,\delta}) $ is a $
(\vartheta,\eps) $-intermediate point
up to picking $\delta\le \eps/2$ and $t>0$ sufficiently small independently of $ \delta $.
Choosing now $\delta\le (\eps \sfE{\lambda}t)^{1/2}$, thanks to \eqref{eq:94} we see that
$\FlowName_t(\Geodesic_{\vartheta,\delta})  $ also satisfies 
\eqref{eq:cap1:60-bis}.

Finally, \eqref{e-new} is an immediate consequence of \eqref{e-b} and Theorem \ref{thm:self-improvement}.
\end{proof} 

We conclude this subsection by an alternative \emph{local} formulation of gradient flows, which becomes a characterization in case
$\phi$ is geodesically $\lambda$-convex for some $\lambda\in \R$. This leads to a definition which is actually independent of $\lambda$. To this aim let us first introduce, for a given geodesic $\sfv$ starting from $\sfv_0=u_{t_0}$, the following quantity:
\begin{equation} 
  \label{eq:cap1:77}
  \begin{aligned}
    \DirDistDer u{t_0}\sfv:= \lim_{s\down0}\frac 1{2s}\Urd
    \DistSquare{u_t}{\sfv_s}\restr{t=t_0} = & \sup_{0<s\le 1}\frac
    1{2s}\Urd \DistSquare{u_t}{\sfv_s}\restr{t=t_0}
    \\=&\sup_{0<s\le
      1}\limsup_{t\down t_0}\frac {\DistSquare{u_t}{\sfv_s}
      -\DistSquare{u_{t_0}}{\sfv_s} }{2s(t-t_0)}.
  \end{aligned}
\end{equation}
The identity between $\sup$ and $\lim$ in \eqref{eq:cap1:77}
will be justified through Lemma \ref{le:UDer_monotone} below.

\begin{proposition}[Local characterization of $ \EVIname $]
  Let $u:(0,+\infty)\to \Dom\phi$ be a continuous 
  curve. If $u$ is a solution of $\EVI\AmbientSpace\DistName\phi\lambda$ 
  according to Definition \ref{def:GFlow}, then 
  for every $t>0$ and every geodesic $\sfv$ emanating
  from $\sfv_0=u_t\in \Dom\phi $ there holds
  \begin{equation} 
    \label{eq:cap1:80}
    \DirDistDer ut\sfv\le \DirSuperDer\phi{u_t}{\sfv}.
  \end{equation}
  Conversely, if $u$ satisfies \eqref{eq:cap1:80} and $\phi$ is $\lambda$-convex, then
  $u$ is a solution of $\EVI\AmbientSpace\DistName\phi\lambda$.
\end{proposition}
\begin{proof}
  Let the continuous curve
  $u:(0,+\infty)\to \Dom\phi $ satisfy $\EVI\AmbientSpace\DistName\phi\lambda$
  and $\sfv$ be a geodesic emanating from $u_{t_0}$, for any $ t_0>0 $. By plugging $ v=\sfv_s $ in \eqref{eq:EVI} (one can assume with no loss of generality that $ \sfv_s \in \Dom \phi $), for every $ s \in (0,1] $ we get 
  \begin{equation*}
    \label{eq:cap2:7}
    \frac 12\Urd \DistSquare{u_t}{\sfv_s} \restr{t=t_0} \le \phi(\sfv_s)-\phi(u_{t_0})-
    \frac {\lambda  s^2}2\DistSquare{u_{t_0}}{\sfv_1} .
  \end{equation*}
  Dividing by $s$ and passing to the limit as $s\down0$, we end up with 
  \begin{displaymath}
    \DirDistDer u {t_0} \sfv
    =\lim_{s\down0}\frac 1{2s} \Urd \DistSquare{u_t}{\sfv_s} \restr{t=t_0} \le
    \liminf_{s\down0}
    \left(\frac{\phi(\sfv_s)-\phi(u_{t_0})}s-
    \frac {\lambda  s}2\DistSquare{u_{t_0}}{\sfv_1}\right)=
    \DirSuperDer\phi{u_{t_0}}\sfv.
  \end{displaymath}  
  Suppose now that $u$ satisfies \eqref{eq:cap1:80} and $\phi$ is $\lambda$-convex;
  for every $t_0>0$ and $v\in \Dom\phi $, let $\sfv$ be an admissible geodesic in
  $\GeoConPhi\AmbientSpace\lambda{u_{t_0}}v$.
  By the definition of $\DirDistDer ut\sfv$ and
  \eqref{eq:cap1:61}, we obtain:
  \begin{displaymath}
    \frac 12\Urd \DistSquare{u_t}v\restr{t=t_0} \topref{eq:cap1:77} \le
    \DirDistDer u{t_0}\sfv\le \DirSuperDer\phi{u_{t_0}}\sfv
    \le \phi(v)-\phi(u_{t_0})-\frac{\lambda} 2\DistSquare{u_{t_0}}v.\qedhere
  \end{displaymath}
\end{proof} 

\begin{corollary}[$ \EVIname_\lambda $ with different $ \lambda $]
  \label{cor:trivial}
  Let $\lambda_1<\lambda_2$ and 
  let $u$ be a solution of $\EVI\AmbientSpace\DistName\phi{\lambda_1}$ 
  according to Definition \ref{def:GFlow}.
  If $\phi$ is $\lambda_2$-convex, then
  $u$ is a also solution of $\EVI\AmbientSpace\DistName\phi{\lambda_2}.$ 
\end{corollary}

\begin{lemma}
  \label{le:UDer_monotone}
  Let $u:(0,+\infty)\to \AmbientSpace$ be a continuous curve
  and $\sfv$ be a given geodesic emanating from $\sfv_0:=u_t$
  for some $t>0$. Then the map 
  \begin{equation*}
    \label{eq:cap2:8}
    s \mapsto s^{-1} \, \Urd \DistSquare{u_t}{\sfv_s} \
    \text{is nonincreasing in $(0,1]$}.
  \end{equation*}
\end{lemma}
\begin{proof}
  Thanks to the continuity of $u$,
  we have:
  \begin{equation*}
    \label{eq:cap2:9}
    s^{-1} \, \Urd \DistSquare{u_t}{\sfv_s}=
    2s^{-1}\Dist{u_t}{\sfv_s}\limsup_{h\down 0}\frac{\Dist{u_{t+h}}{\sfv_s}
      -\Dist{u_t}{\sfv_s}}h= 
    2 |\dot \sfv| \, \Urd \Dist{u_t}{\sfv_s}.
  \end{equation*} 
  We are therefore left with proving that 
  $s\mapsto \Urd \Dist{u_t}{\sfv_s}$ is not increasing with respect to $s$. To this aim, note that if $s_1<s_2$ the triangular inequality
  and the geodesic property of $\sfv$ yield
  \begin{equation*} 
    \label{eq:cap2:4}
    \Dist{u_{t+h}}{\sfv_{s_2}}\le 
    \Dist{u_{t+h}}{\sfv_{s_1}}+
    \Dist{\sfv_{s_1}}{\sfv_{s_2}},\qquad
    \Dist{u_{t}}{\sfv_{s_2}}=
    \Dist{u_{t}}{\sfv_{s_1}}+
    \Dist{\sfv_{s_1}}{\sfv_{s_2}}, 
  \end{equation*}
  whence
  \begin{equation}
    \label{eq:cap2:10}
    \frac{\Dist{u_{t+h}}{\sfv_{s_2}}-       \Dist{u_{t}}{\sfv_{s_2}}}h\le
    \frac{\Dist{u_{t+h}}{\sfv_{s_1}}-       \Dist{u_{t}}{\sfv_{s_1}}}h.
  \end{equation}
  The thesis follows by passing to the limit in \eqref{eq:cap2:10} as $h \down 0$.
\end{proof}

\subsection{Examples}
\label{subsec:examples}
We already considered the case of $\lambda$-convex
functionals in Hilbert spaces in the Introduction;
many examples of applications can be found in \cite{Brezis71,Brezis71b,Brezis72,Barbu76,Showalter97}.
A discussion concerning the smooth Riemannian case can be found in
\cite[Proposition 23.1]{Villani09}.
\subsubsection[Hadamard NPC spaces, $\mathrm{CAT}(k)$ spaces and
  convexity along generalized geodesics]{Hadamard NPC spaces, $\boldsymbol{\mathrm{CAT}(k)}$ spaces and
  convexity along generalized geodesics}
\label{subsubsec:CAT}
One of the nicest metric setting where general existence of
\EVIshort\lambda-gradient flows can be proved is provided by the class
of Hadamard \emph{non
  positively curved} (NPC) metric spaces, which
we introduce here by one of their
equivalent characterizations \cite[Section 1.2]{Bacak14}.
A Hadamard (or $\mathrm{CAT}(0)$) space is a geodesic space (or set, recall Definition \ref{geo-geo})
such that the map $x\mapsto \frac 12 \DistSquare{y}{x}$ is strongly
$1$-convex for every $y\in \AmbientSpace$; that is, 
for every $\sfx\in \GeoCon{}{\sfx_0}{\sfx_1}$ and every $y\in
\AmbientSpace$ we have
\begin{equation}
  \label{eq:60}
  \DistSquare{y}{\sfx_t}\le (1-t)\DistSquare y{\sfx_0}+t\DistSquare
  y{\sfx_1}-
  t(1-t)\DistSquare{\sfx_0}{\sfx_1}.
\end{equation}
When $X$ is a smooth Riemannian manifold, the above condition
is equivalent to requiring the non-positivity of its sectional curvature.

If $\phi$ is geodesically convex in a Hadamard space, \textsc{Mayer} \cite{Mayer98} and \textsc{Jost}
\cite{Jost98} proved the convergence of the Minimizing
Movement scheme (recall (\ref{eq:50},b) or \eqref{eq:56}; see also Section \ref{sec:uniform_estimate}) to a
contraction semigroup $(\FlowName_t)_{t\ge0}$ in $\overline {\Dom\phi}$, using in
particular that the map $x\mapsto \Resolvent\tau x$ is a contraction;
they also provide nice applications to the Harmonic map flow in
Hadamard spaces.
The generation result has been extended to arbitrary $\mathrm{CAT}(k)$ spaces,
$k\in \R$,
by \cite{Ohta-Palfia17}.
In this framework, the link between the limit contraction semigroup
(obtained by the convergence of the variational scheme) with the \EVIshort\lambda-formulation 
has been clarified by \cite{Ambrosio-Gigli-Savare08}, with two further
important developments:
the optimal error estimate of order 1 (as in Hilbert spaces, \cite{Baiocchi89,Nochetto-Savare-Verdi00})
\begin{equation}
  \label{eq:62}
  \Dist{\Flow t{u_0}}{\mathrm J^n_{t/n}(u_0)}\le \frac{t}{n\sqrt 2}\MetricSlope\phi{u_0} 
\end{equation}
and a considerably weaker assumption on the system $\System$.
It is in fact sufficient that for every triple of points $\sfx_0,\sfx_1,y\in
\Dom\phi$ there exists a curve $\sfx:[0,1]\to \AmbientSpace$
connecting $\sfx_0$ to $\sfx_1$ (and possibly depending on $y$) along which 
\eqref{eq:60} holds and $\phi$ satisfies the $\lambda$-convexity
inequality \eqref{eq:cap1:60}. This condition clearly covers the case of Hadamard NPC spaces but also
allows for important applications to the
Kantorovich-Rubinstein-Wasserstein space $(\ProbabilitiesTwo {\R^d},\WDistName)$,
which is not an NPC space if $d\ge 2$: a large class of interesting
functionals in $ \ProbabilitiesTwo {\R^d} $ are in fact \emph{convex along generalized geodesics}, we refer to
\cite[Section 11.2]{Ambrosio-Gigli-Savare08} for more details.

We can recap the above discussion in the following result:
\begin{theorem}[Existence of $\EVIshort\lambda$-flows in
  $\mathrm{CAT}$ spaces]
  \label{thm:CAT}
  Let $\System$ be a metric-functional system, for which
  $(\AmbientSpace,\DistName)$ is a complete geodesic metric space.
  Let us assume that at least one of the
  following two conditions holds:
  \begin{enumerate}
  \item $(X,\DistName)$ is a $\mathrm{CAT}(k)$ space
    for some $k\in \R$ and $\phi$ is $\lambda$-convex;
  \item for every $\sfx_0,\sfx_1,y\in \Dom\phi$
    there exists a curve $\sfx:[0,1]\to\Dom\phi$ connecting $\sfx_0$ to
    $\sfx_1$ such that
    \eqref{eq:60} and \eqref{eq:cap1:60} hold.
  \end{enumerate}
  Then the functional $\phi$ generates an $\EVIshort\lambda$-flow in
  $\overline{\Dom\phi}$,
  according to Definition \ref{def:GFlow}.  
\end{theorem}
\subsubsection{Alexandrov spaces}
\label{subsubsec:Alexandrov}
A second important class of metric spaces where $\lambda$-convex
functionals
generate an $\EVIshort\lambda$-flow is provided by Alexandrov spaces
with curvature bounded from below
\cite{Burago-Gromov-Perelman92,Gromov99,Sturm99,Burago-Burago-Ivanov01,Plaut02}: they
can be considered as the natural non-smooth metric version of 
Riemannian manifolds with sectional curvature bounded from below.
Referring to the above quoted papers for the general definition,
which is based on triangle comparison with the reference
$2$-dimensional model of constant curvature \cite{Alexandrov51},
here we limit ourselves to recalling one of the equivalent characterizations
of positively curved (PC) geodesic spaces:
the squared distance function $ x \mapsto \frac 12\DistSquare y x $ is
\mbox{$-1$-concave} along geodesics; more explicitly, 
for every $y\in \AmbientSpace$ and $\sfx\in \GeoCon{}{\sfx_0}{\sfx_1}$
we have
\begin{equation*}
  \label{eq:60bis}
  \DistSquare{y}{\sfx_t}\ge (1-t)\DistSquare y{\sfx_0}+t\DistSquare
  y{\sfx_1}-
  t(1-t)\DistSquare{\sfx_0}{\sfx_1}.
\end{equation*}
Existence of gradient flows for $\lambda$-convex functionals in complete
$m$-dimensional (in particular locally compact) Alexandrov spaces has
been proved by \cite{Perelman-Petrunin}. In \cite{Muratori-SavareIII} we will generalize this result to cover
arbitrary Alexandrov spaces and even more general cases,
characterized by a new geometric property,  weaker than the Alexandrov one,
based on a local angle condition
between geodesics
and on the semi-concavity of the squared distance
\cite{Savare07,Daneri-Savare14}
(see also \cite{Lytchak05} and \cite{Ohta09} for a different viewpoint).
Such an approach has also the advantage to be
stable with respect to the Wasserstein construction, in the sense that the above property is
shared by the Wasserstein space $(\ProbabilitiesTwo X,\WDistName)$. Note that the space $(\ProbabilitiesTwo X,\WDistName)$ constructed on an Alexandrov space $(X,\DistName)$
is not finite dimensional and fails to be an Alexandrov space if
$(X,\DistName)$ is not positively curved
(see \cite[Chapter 12.3]{Ambrosio-Gigli-Savare08}, \cite[Proposition 2.10]{Sturm06I}).

Referring to \cite{Muratori-SavareIII} for a more detailed discussion,
here we limit ourselves to stating the Alexandrov case:
\begin{theorem}[Existence of $\EVIshort\lambda$-flows in
  Alexandrov spaces with lower curvature bounds]
  \label{thm:Alexandrov}
  \hfill Let \\ $\System$ be a metric-functional system, for which
  $(\AmbientSpace,\DistName)$ is a complete geodesic metric space
  satisfying a uniform lower curvature bound in the sense of
  Alexandrov
  and $\phi$ is $\lambda$-convex.
  Then the functional $\phi$ generates an $\EVIshort\lambda$-flow in
  $\overline{\Dom\phi}$,
  according to Definition \ref{def:GFlow}.  
\end{theorem}  
\subsubsection[The Kantorovich-Rubinstein-Wasserstein space
  $(\ProbabilitiesTwo X,\WDistName)$]{The Kantorovich-Rubinstein-Wasserstein space
  $\boldsymbol{(\ProbabilitiesTwo X,\WDistName)}$}
\label{subsubsec:Wasserstein}
One of the main recent motivations to study gradient flows in metric
spaces
comes from the remarkable \textsc{Otto}'s
interpretation of the heat and Fokker-Planck
flows \cite{Jordan-Kinderlehrer-Otto98}
and of a large class of nonlinear diffusion equations \cite{Otto01}
as gradient flows of suitable entropy functionals in the
Kantorovich-Rubinstein-Wasserstein (KRW) space of probability
measures $(\ProbabilitiesTwo X,\WDistName)$. 

We assume for simplicity that $(X,\DistName)$ is complete and
separable. We denote by
$\Probabilities X$ the space of Borel probability measures on $X$ and by
$\ProbabilitiesTwo X$ the subset
of those $\mu\in \Probabilities X$ with finite quadratic moment:
\begin{equation*}
  \label{eq:64}
  \int_X \DistSquare{x}{o}\,\d\mu(x)<\infty\quad
  \text{for some (and thus any) $o\in\AmbientSpace$}.
\end{equation*}
The KRW-distance between two measures $\mu_1,\mu_2\in
\ProbabilitiesTwo X$ can be defined as
\begin{equation*}
  \label{eq:63}
  \WDistSquare{\mu_1}{\mu_2}:=\min\left\{
  \int_{X\times X}\DistSquare{x_1}{x_2}\,\d\mmu(x_1,x_2):
  \mmu\in \Probabilities{X\times X},\quad
  \pi^i_\sharp\mmu=\mu_i \right\},
\end{equation*}
where $\pi^i:X\times X\to X$, 
$\pi^i(x_1,x_2):=x_i$, are the Cartesian projections and for 
a Borel map $\tt:X\to Y$ between metric spaces,
$\tt_\sharp:\Probabilities X\to\Probabilities Y$ denotes the
push-forward operation defined by
\begin{equation*}
  \label{eq:65}
  \tt_\sharp\mu(B):=\mu(\tt^{-1}(B))\quad
  \text{for every Borel subset $B\subset Y$ and every $\mu\in
    \Probabilities X$.}
\end{equation*}
It turns out (see e.g.~\cite{Ambrosio-Gigli-Savare08,Villani09})
that $(\ProbabilitiesTwo X,\WDistName)$ is complete and separable, and
it is also a geodesic (resp.~length) space provided
$(X,\DistName)$ is geodesic (resp.~length).

If $X=\R^d$ with the usual Euclidean distance,
by Otto's \cite{Otto01} and subdifferential
\cite{Ambrosio-Gigli-Savare08} calculus, one can see that
the metric gradient flow of the entropy functional
(also called \emph{Csisz\`ar $f$-divergence)}
\begin{equation}
  \label{eq:66}
  \PhiF(\mu):=\int_X
  F\Big(\frac{\d\mu}{\d\mm}\Big)\,\d\mm,%\quad
  %\mm:=\rme^{-V}\Leb d,
\end{equation}
where $F:[0,+\infty)\to[0,\infty)$ is a smooth, convex and superlinear
function and $\mm:=\rme^{-V}\Leb d $ is a Borel measure
induced by a smooth potential $V:X\to \R$ (for simplicity), provides a solution to the diffusion
equation
\begin{equation}
  \label{eq:67}
  \partial_t \mu=\nabla\cdot\left(\mu \nabla F'(\varrho)\right) = \nabla\cdot\left( \rme^{-V} \nabla\! \left( \varrho F'(\varrho) - F(\varrho) \right) \right) \quad
  \text{in }(0,\infty)\times \R^d,\quad
  \varrho=\frac{\d\mu}{\d\mm},
\end{equation}
e.g.~in the weak sense
\begin{displaymath}
  \int_0^\infty\kern-5pt\int_{\R^d} \partial_t\zeta(t,x)\,\d\mu_t(x)\,\d t=
  \int_0^\infty\kern-5pt\int_{\R^d} \nabla\zeta(t,x)\cdot \nabla
  F'(\varrho_t(x))\,\d\mu_t(x)\,\d t
\end{displaymath}
for every $\zeta\in \rmC^\infty_c((0,\infty)\times \R^d)$. Among the class of entropy functionals \eqref{eq:66}, the logarithmic
entropy one
\begin{equation}
  \label{eq:79}
  \PhiE(\mu):=\int_X
  E\Big(\frac{\d\mu}{\d\mm}\Big)\,\d\mm, \quad
  E(\varrho):=\varrho\log\varrho,
\end{equation}
plays a distinguished role, since one
easily sees that in the case of the functional $\PhiE$ the density $ \sigma = \rme^{-V} \varrho $ of $\mu$ w.r.t.~$ \Leb d $ solves the Fokker-Planck equation
\begin{equation}
  \label{eq:69}
  \partial_t \sigma =\Delta \sigma + \nabla\cdot (\sigma \nabla V)
  \quad \text{in }(0,\infty)\times \R^d ,
\end{equation}
which reduces to the heat equation when $V$ is constant (so that $\mm=\Leb d$).

It has been a striking achievement of \textsc{McCann} \cite{McCann97}
to show that if $F$ satisfies the condition (clearly fulfilled by $E$)
\begin{equation}
  \label{eq:75}
  s\mapsto \rme ^sF(\rme^{-s})\quad\text{is convex and nonincreasing in }(0,\infty),
\end{equation}
and $\mm$ is log-concave (equivalently, $V$ is convex in $\R^d$), then the functional $\PhiF$
of \eqref{eq:66} is convex in $(\ProbabilitiesTwo
X,\WDistName)$; it turns out that it is also convex along generalized geodesics
\cite[Chapter 9]{Ambrosio-Gigli-Savare08}, according to the condition stated in Theorem \ref{thm:CAT} (a more refined characterization of the $\lambda$-convexity of $\PhiF$, involving the dimension $d$ and the first and second differential of $V$, is also available, see \cite{Villani09}).

This nice property has been used in \cite{Ambrosio-Gigli-Savare08}
to show that the gradient flow of $\PhiF$ in $(\ProbabilitiesTwo
X,\WDistName)$ (initially obtained as a limit of the Minimizing
Movement scheme) is in fact an $\EVIshort0$-flow in the sense of
Definition \ref{def:GFlow}. Such a characterization, further extended to to $\lambda$-convex
potentials $V$
and to functionals including interaction
energies, has been extremely useful in the study of
the asymptotic behaviour
\cite{CarrilloMcCannVillani06} and in the analysis of more complex
flows, see e.g.~\cite{Matthes-McCann-Savare09,Blanchet-Carlen-Carrillo12,Kinderlehrer-Monsaingeon-Xu17}.

Starting from further geometric investigations of the link between
lower Ricci curvature bounds and the geodesic convexity of the
entropy functionals \eqref{eq:66} when $X$ is a complete Riemannian manifold
endowed with the Riemannian distance $\sfd$
\cite{Otto-Villani00,Cordero-McCann-Schmuckenschlager01,Cordero-McCann-Schmuckensclager06,Sturm05},
an intensive effort has been devoted to
identify the Wasserstein gradient flow of the logarithmic entropy with
the heat flow in $X$ or, more generally, with the
$L^2(X,\mm)$-gradient flow of the Dirichlet form
\begin{displaymath}
  \DD(u):=\frac 12\int_X |\nabla u(x)|^2\,\d\mm(x)
\end{displaymath}
associated with the Sobolev space $W^{1,2}(X,\DistName,\mm)$.
The case of smooth Riemannian manifolds has been
fully analyzed by \textsc{Erbar} \cite{Erbar10} and \textsc{Villani}
\cite[Chapter 23]{Villani09}; the case where $X$ is a Hilbert space and
$\mm$
is a log-concave measure has been studied in
\cite{Ambrosio-Savare-Zambotti09},
whereas Alexandrov spaces have been considered in
\cite{Ohta09,GigliKuwadaOhta10}.
We can summarize part of the above discussion in the following result
(which, however, is far from being exhaustive):
\begin{theorem}[Heat flow as \EVIshort\lambda-flow in KRW spaces]
  \label{thm:Heat}
  Let $(X,\DistName)$ be a complete, separable
  geodesic metric space endowed with a Borel measure $\mm$,
  $\lambda\in \R$
  and let us consider the metric-functional system
  $\X=(\ProbabilitiesTwo X,\WDistName,\PhiF)$,
  where $\PhiF$ is the entropy functional \eqref{eq:66} satisfying
  McCann's condition \eqref{eq:75}.
  If one of the following assumptions holds
  \begin{enumerate}
  \item $(X,\sfd)$ is a $d$-dimensional smooth Riemannian manifold
    with Ricci curvature $\ge \kappa_1$, $\mm:=\rme^{-V}\,\mathrm{Vol}$
    where $V:X\to \R$ is a
    geodesically $\kappa_2$-convex potential, $\lambda:=\kappa_1+\kappa_2$;
  \item $(X,\sfd)$ is a Hilbert space and $\mm$ is a log-concave
    measure, $\lambda:=0$;
  \item $(X,\sfd)$ is a compact $d$-dimensional Alexandrov space
    with curvature $\ge\kappa$ and
    $\mm=\mathscr H^d$
    is the Hausdorff $d$-dimensional measure, $\lambda:=\kappa$,
  \end{enumerate}
  then $\PhiF$ generates an $\EVIshort\lambda$-flow in
  $(\ProbabilitiesTwo X,\WDistName)$.
\end{theorem}
As we already remarked, when $X$ is not positively curved,
$(\ProbabilitiesTwo X,\WDistName)$ is not an Alexandrov space
so that Theorem \ref{thm:Alexandrov} cannot be directly applied.

\subsubsection[$\mathrm{RCD}(K,\infty)$ metric-measure spaces]{$\boldsymbol{\mathrm{RCD}(K,\infty)}$ metric-measure spaces}
\label{subsubsec:RCD}
The equivalence between a lower Ricci curvature bound $\mathrm{Ric}\ge
K$
and the $K$-convexity of the logarithmic entropy functional \eqref{eq:79}
in Riemannian manifolds endowed with the Riemannian volume measure
$\mm$ motivated the deep investigations of \textsc{Lott-Villani}
\cite{Lott-Villani09} and \textsc{Sturm} \cite{Sturm06I, Sturm06II}
to solve the problem of finding synthetic notions of lower Ricci
curvature bounds for general metric-measure spaces
$(X,\DistName,\mm)$, a structure formed by a complete, separable, and
length metric space $(X,\DistName)$ and a Borel measure $\mm$
satisfying the growth condition
\begin{equation*}
\text{$\mm(\mathrm B(o,r))\le A\rme^{Br^2}$ for some
$o\in \AmbientSpace$ and constants $A,B\ge0$,}
\label{eq:76}
\end{equation*}
where $\mathrm B(o,r):=
\{x\in \AmbientSpace:\Dist xo<r\}$.
According to their definition, such a structure satisfies the \emph{Curvature-Dimension}
$\mathrm{CD}(K,\infty)$ condition if
the logarithmic entropy functional $\PhiE$ \eqref{eq:79} is $K$-convex in $(\ProbabilitiesTwo
X,\WDistName)$.

In order to capture a Riemannian-like structure, related to the
linearity of the heat flow in $(X,\DistName,\mm)$, the
Lott-Sturm-Villani condition has been reinforced in 
\cite{Ambrosio-Gigli-Savare14b}
(see also \cite{Ambrosio-Gigli-Savare14})
by asking that the logarithmic Entropy function $\PhiE$ generates
an $\EVIshort K$-flow in $(\ProbabilitiesTwo
X,\WDistName)$. This condition precisely characterizes the class of
$\mathrm{RCD}(K,\infty)$ metric-measure spaces:
\begin{definition}[Metric-measure spaces with Ricci
  curvature bounded from below]
  \label{def:RCD}
  \hfill A complete, \\ separable metric-measure space $(X,\DistName,\mm)$
  satisfies the $\mathrm{RCD}(K,\infty)$ condition if
  the logarithmic entropy function $\PhiE$ of \eqref{eq:79} generates
  an $\EVIshort K$-flow in $(\ProbabilitiesTwo
  X,\WDistName)$.
\end{definition}
It is a remarkable fact that in this case the gradient flow of $\PhiE$
is a semigroup of operators in $\ProbabilitiesTwo
X$ satisfying $ \FlowName_t(\alpha \mu + (1-\alpha) \sigma) = \alpha \FlowName_t(\mu ) + (1-\alpha) \FlowName_t(\sigma)$ for every $ \alpha \in (0,1) $, a property encoded by the logarithmic structure of $\PhiE$ and
the variational formulation of the $\EVIshort K$-flow \cite{Savare07}.

The $\mathrm{RCD}(K,\infty)$-condition shows the relevance of the
synthetic notion
of \EVIshort K-flows; its equivalence with the celebrated Bakry-\'Emery
curvature-dimension condition has been proved in
\cite{Ambrosio-Gigli-Savare15} and its $N$-dimensional
version has been introduced and deeply studied by
\textsc{Erbar-Kuwada-Sturm} in \cite{Erbar-Kuwada-Sturm15}
(by using a refined notion of the \EVIshort\lambda-flow)
and in \cite{Ambrosio-Mondino-Savare15} (by studying the
gradient flows generated by other entropy functionals).
We refer to the survey \cite{Ambrosio-Gigli-Savare17} for a brief
introduction to this theory and to its further developments.

We conclude this discussion by observing that the existence of the
\EVIshort K flow of the entropy functional $\PhiE$
at the level of measures $(\ProbabilitiesTwo X,\WDistName)$
implies relevant generation properties also for the original space $(X,\DistName)$.
In fact, \textsc{Sturm} \cite{Sturm18} proved
that for locally compact metric-measure spaces satisfying the
$\mathrm{RCD}(K,\infty)$-condition
every continuous $\lambda$-convex functional $\phi:X\to \R$ generates an
$\EVIshort\lambda$-flow.
As a byproduct of our analysis of the variational convergence of
\EVIshort\lambda-flows
in \cite{Muratori-SavareII} we will show that this property actually holds in arbitrary $\mathrm{RCD}(K,\infty)$ spaces, thus obtaining the following result.
\begin{theorem}[Existence of $\EVIshort\lambda$-flows in
  $\mathrm{RCD}$ spaces]
  \label{thm:EVIinRCD}
  Let $(X,\DistName,\mm)$ be an $\mathrm{RCD}(K,\infty)$ metric measure
  space and let $\phi:X\to (-\infty,+\infty]$ be
  a continuous $\lambda$-convex functional with proper domain.
  Then $\phi$ generates an $\EVIshort\lambda$-flow in
  $\overline{\Dom\phi}$,
  according to Definition \ref{def:GFlow}.  
\end{theorem}

\section[Energy-dissipation inequality, curves of maximal slope and ${\EVIname}$]{Energy-dissipation inequality, curves of maximal slope and $\boldsymbol{\EVIname}$}\label{sec:max-slope}
There exists a weaker notion of
gradient flow for the system $\System$, which is 
strictly related to the energy identity
\eqref{eq:Cap12:24}. Referring to \cite[Chapters 2, 3]{Ambrosio-Gigli-Savare08} for more
details (see also \cite{Ambrosio-Gigli13}),
we provide the following definition, which is again independent of $ \lambda $.
\begin{definition}[EDI and curves of maximal slope]\label{def: max-slope}
Let $\System$ be a metric-functional system as in \eqref{eq:system} and $ v \in  \AC^2_{\rm loc}([0,+\infty);\AmbientSpace) $. We say that $v$ satisfies the \emph{Energy-Dissipation Inequality (EDI$_0$)} starting from $t=0$ if
\begin{equation}
  \label{eq:17}
  v_0\in \Dom\phi,
  \qquad
  \int_0^t \Big(\frac12\MDS vr
  +\frac 12\MetricSlopeSquare\phi{v_r}\Big)\,\d r+\phi(v_t)\le
  \phi(v_0)
  \quad
  \text{for every }t>0.
  \tag{EDI$_0$}
\end{equation}
Moreover, we say that $v$ is a \emph{curve of maximal slope} if the map $t\mapsto \phi(v_t)$ is locally absolutely continuous in $ [0,+\infty) $ and 
\begin{equation} 
  \label{eq:cap2:20}
  \frac {\d}{\d t}\phi(v_t)\le -\frac12\MDS vt
  -\frac 12\MetricSlopeSquare\phi{v_t}\quad
  \text{for $\Leb 1$-a.e.\ $t>0$}.
\end{equation}
\end{definition}
Note that every curve of maximal slope for the system $\System$
satisfies \eqref{eq:17}: it is sufficient to integrate
\eqref{eq:cap2:20} in the interval $[0,t]$.
In fact \eqref{eq:cap2:20} is much stronger, since it also yields the identity
\begin{equation} 
  \label{eq:cap2:20-equality}
 \frac {\d}{\d t}\phi(v_t) = - \MetricSlopeSquare\phi{v_t} = - \MDS vt \quad
  \text{for $\Leb 1$-a.e.\ $t>0$}.
\end{equation}
Indeed, \eqref{eq:cap1:83} and the differentiability of $ t \mapsto \phi(v_t) $ and $ t \mapsto v_t $ (in the metric sense \eqref{eq:cap1:65}) for $\Leb 1$-a.e.\ $t>0$ entail
$$ 
\frac {\d}{\d t}\phi(v_t) \ge - \MetricSlope\phi{v_t} \, \MD vt \quad
  \text{for $\Leb 1$-a.e.\ $t>0$},
  $$
which combined with \eqref{eq:cap2:20} entails $ (\MetricSlope\phi{v_t}-\MD vt)^2=0 $. 
In turn, \eqref{eq:cap2:20-equality} is equivalent to the Energy-Dissipation
Equality
\begin{equation}
  \label{eq:31}
  \int_s^t \Big(\frac12\MDS vr
  +\frac 12\MetricSlopeSquare\phi{v_r}\Big)\,\d r+\phi(v_t)=
  \phi(v_s)
  \quad
  \text{for every }0\le s\le t.
  \tag{EDE}
\end{equation}
When $\phi$ admits an \EVIshort\lambda-gradient flow according to Definition
\ref{def:GFlow}, the two notions in fact coincide.
We have already seen in Theorem \ref{thm:main1} (identity \eqref{eq:Cap12:24})
that every solution of
$\EVI\AmbientSpace\DistName\phi\lambda$ (starting from $ u_0 \in \Dom{\phi} $) is a curve of maximal slope
and therefore satisfies
the Energy-Dissipation Equality \eqref{eq:31}.
The converse implication, only upon requiring the Energy-Dissipation
Inequality \eqref{eq:17}, is stated in the next theorem. 
\begin{theorem}[Curves of maximal slope and \EVIshort\lambda-gradient flow]\label{thm:MaxSlopeAreGFlows}
  Let $v$ be a curve that satisfies the {Energy-Dissipation Inequality} \eqref{eq:17}. 
  Assume in addition that the \EVIshort\lambda-gradient flow $\FlowName_t $ of $\phi$ exists (for some $ \lambda \in \R $) in a set
$D\subset \overline{\Dom\phi}$ that contains the image of $ v $. Then $v_t=\Flow t{v_0}$ for every $t\ge0$.
\end{theorem}
\begin{proof}
  %First of all note
  % that, since $ t \mapsto \phi(v_t) \in \AC^1_{\rm loc}((0,+\infty);\R) $ and $ v \in C^0([0,+\infty);\AmbientSpace) $, to our purposes it is not restrictive to assume that $u_0\in \Dom\phi $ and $ t \mapsto \phi(v_t) $ is locally absolutely continuous in $ [0,+\infty) $.
  We can suppose with no loss of generality that $\lambda\le 0$.
  The map $t\mapsto \MetricSlopeSquare\phi{v_t}$ is Lebesgue measurable (recall \eqref{eq: loc-slope}) and belongs to 
  $L^1((0,T))$ for every $T>0$ in view of
\eqref{eq:17}.
  By applying Lemma \ref{le:RiemannSums} below, we can therefore find $ \bar t \in (0,1) $, a sequence $\tau_{\h}\downarrow 0$
  and points $t_{\h}:=\tau_{\h} \bar t$ such that, upon letting 
  $t^{\n}_{\h}:=t_{\h}+\n\tau_{\h}$, $v^{\n}_{\h}:=
  v_{t^{\n}_{\h}}$ and choosing $N_{\h}\in \N$ so that
  $\tau_{\h} (N_{\h}+2)\le T< \tau_{\h} (N_{\h}+3)$, we have
  \begin{equation}
    \label{eq:cap2:21}
    C_{\h}:=
    \sum_{\n=0}^{N_{\h}}\left|\tau_{\h}\MetricSlopeSquare\phi{v^{\n}_{\h}}
    -\int_{t^{\n}_{\h}}^{t_{\h}^{\n+1}}
    \MetricSlopeSquare\phi{v_r}\,\d r\right|,\qquad
    \lim_{k\to\infty} C_{\h}=0.
  \end{equation}
  Note that eventually $ v^{\n}_{\h} \in \DomainSlope \phi $ for all $ n $ as above. Now let us set $u^{\n}_{\h}:=\FlowName_{\n\tau_{\h}}[v^{0}_k]$ and $w^{\n+1}_{k}:=\FlowName_{\tau_{\h}}[v^{\n}_{\h}]$; 
  in order to estimate $\Dist{v^{\n}_{\h}}{u^{\n}_{\h}}$
  we start from the basic recurrence relations
  \begin{displaymath}
    \rme^{\lambda t^{\n}_{\h}}\Dist{v^{\n}_{\h}}{u^{\n}_{\h}}\le
    \rme^{\lambda t^{\n}_{\h}}
    \Big(\Dist{v^{\n}_{\h}}{w^{\n}_{\h}}+
    \Dist{w^{\n}_{\h}}{u^{\n}_{\h}}\Big)
    \topref{eq:Cap12:1bis}\le
    \rme^{\lambda t^n_{\h}}\Dist{v^{\n}_{\h}}{w^{\n}_{\h}}+
    \rme^{\lambda t^{\n-1}_{\h}}\Dist{v^{\n-1}_{\h}}{u^{\n-1}_{\h}},
  \end{displaymath}
  so that a telescopic summation plus the inequality $\lambda t^n_k\le \lambda\tau_k$ (for $ n \ge 1 $) yield
  \begin{equation}
    \label{eq:cap2:26}
    \sup_{1\le \n\le N_{\h}} \rme^{\lambda t^{\n}_{\h}}\Dist{v^{\n}_{\h}}{u^{\n}_{\h}}
    \le \sum_{\n=1}^{N_{\h}} \rme^{\lambda\tau_{\h}}\Dist{v^{\n}_{\h}}{w^{\n}_{\h}}.
  \end{equation}
  It is therefore crucial to estimate the terms $\Dist{v^{\n}_{\h}}{w^{\n}_{\h}}$. If we apply \eqref{eq:3} with $ u_t=w_\h^\n(t)=\FlowName_t[v^{\n-1}_{\h}] $, $v=v^{\n}_{\h}$ and $ t=\tau_\h $, we obtain: 
%  we multiply \eqref{eq:EVI} by $e^{2\lambda t}$ and
%  integrate in time in the interval $(0,\tau_{\h})$ choosing
%  $v^{\n-1}_{\h}$ as initial datum and $v:=v^{\n}_{\h}$;
%  setting $w^{\n}_{\h}(t):=\FlowName_t[v^{\n-1}_{\h}]$ and applying \eqref{eq:3} we get
  \begin{equation}
    \label{eq:cap2:22}
    \frac {\rme^{2\lambda \tau_{\h}}}2\DistSquare{w^{\n}_{\h}}{v^{\n}_{\h}}-
    \frac 12\DistSquare{v^{\n-1}_{\h}}{v^{\n}_{\h}}\le
    \sfE{2\lambda}{\tau_{\h}}\big(\phi(v^{\n}_{\h})-\phi(v^{n-1}_\h)\big)
    +\frac{\tau_k^2}2\MetricSlopeSquare\phi{v^{n-1}_k};%\int_0^{\tau_{\h}}\rme^{2\lambda r} \phi(w^{\n}_{\h}(r))\,\d r.
  \end{equation}
on the other hand, by the definition of metric slope and H\"older's inequality we infer that
  \begin{equation}
    \label{eq:cap2:23}
    \frac 12\DistSquare{v^{\n-1}_{\h}}{v^{\n}_{\h}}\le
    \frac{\tau_{\h}}2\int_{t^{\n-1}_{\h}}^{t^{\n}_{\h}}\MDS vr\,\d r=
    \tau_\h\,I^n_\h+\tau_{\h}\phi(v^{\n-1}_{\h})-\tau_{\h}\phi(v^{\n}_{\h})-\frac{\tau_{\h}}2
    \int_{t^{\n-1}_{\h}}^{t^{\n}_{\h}}\MetricSlopeSquare\phi{v_r}\,\d r,
  \end{equation}
  where
  \begin{equation*}
    \label{eq:36}
    I^n_\h:=\phi(v^{\n}_{\h})-\phi(v^{\n-1}_{\h})+
    \frac12 \int_{t^{\n-1}_{\h}}^{t^{\n}_{\h}}\Big(\MDS vr+\MetricSlopeSquare\phi{v_r}\Big)\,\d r .
  \end{equation*}
  % An integration by parts and the monotonicity of
%   $r\mapsto \rme^{2\lambda r} \MetricSlopeSquare\phi{w^{\n}_{\h}(r)}$
%   yield
%   \begin{displaymath}
%     \begin{aligned}
%       -\int_0^{\tau_{\h}}\rme^{2\lambda r} \phi(w^{\n}_{\h}(r))\,\d r&=
%       \int_0^{\tau_{\h}}\big(\sfE{2\lambda}{\tau_{\h}}-
%       \sfE{2\lambda}{r}\big)\MetricSlopeSquare\phi{w^{\n}_{\h}(r)}\,\d
%       r - \sfE{2\lambda}{\tau_{\h}}\phi(v^{\n-1}_{\h})\\
%       &\le
%       \MetricSlopeSquare\phi{v^{\n-1}_{\h}}
%       \int_0^{\tau_{\h}}
%       \big(\sfE{2\lambda}{\tau_{\h}}-
%       \sfE{2\lambda}{r}\big)\rme^{-2\lambda r}\,\d r-
%       \sfE{2\lambda}{\tau_{\h}}\phi(v^{\n-1}_{\h})
%       \\&=
%       \MetricSlopeSquare\phi{v^{\n-1}_{\h}}
%       \tsfE{2\lambda}{\tau_{\h}}
%       -
%       \sfE{2\lambda}{\tau_{\h}}\phi(v^{\n-1}_{\h}),\qquad
%       \tsfE\lambda t:=\int_0^t \sfE\lambda r\,\d r.
%     \end{aligned}
%   \end{displaymath}
  Summing up \eqref{eq:cap2:22} and \eqref{eq:cap2:23},
  after a multiplication by a factor $2$ we end up with
  \begin{gather*}
    \label{eq:cap2:25}
    \rme^{2\lambda \tau_{\h}}\DistSquare{w^{\n}_{\h}}{v^{\n}_{\h}}
    \le
    \tau_{\h} \big(a^{\n}_{\h} %+b^{\n}_{\h}
    +c^{\n}_{\h}+I^n_\h),\\
    a^{\n}_{\h}:= 2\big(1-\tau_{\h}^{-1}\sfE{2\lambda}{\tau_{\h}}\big)
    \big(\phi(v^{\n-1}_{\h})-\phi(v^{\n}_{\h})\big),
    %b^{\n}_{\h}&:= 
    %\MetricSlopeSquare\phi{v^{\n-1}_{\h}}
    %\Big(\frac 2{\tau_{\h}}\tsfE{2\lambda}{\tau_{\h}}-\tau_{\h}\Big)
    %\\
     \qquad c^{\n}_{\h}:=
    \tau_{\h}\MetricSlopeSquare\phi{v^{\n-1}_{\h}}-
    \int_{t^{\n-1}_{\h}}^{t^{\n}_{\h}}\MetricSlopeSquare\phi{v_r}\,\d r,
  \end{gather*}
  so that by \eqref{eq:cap2:26} there follows
  \begin{equation}
    \label{eq:cap2:27}
    \sup_{1\le \n\le N_{\h}} \rme^{\lambda t^{\n}_{\h}}\Dist{v^{\n}_{\h}}{u^{\n}_{\h}}
    \le \sqrt T \left(\sum_{\n=1}^{N_{\h}} a^{\n}_{\h}+c^{\n}_{\h} +I^n_\h \right)^{ \frac12 }.
  \end{equation} 
  An elementary numerical inequality yields 
  \begin{displaymath}
    2\big(1-\tau_{\h}^{-1}\sfE{2\lambda}{\tau_{\h}}\big)\le {2} |\lambda|\tau_{\h};
    %\tsfE{2\lambda}{\tau_{\h}}\le \frac 12\tau_{\h}^2
  \end{displaymath}
 thus, letting $T_\h:=t^{N_\h}_\h$, a telescopic summation and $ \phi(v_{t_k}) \le \phi(v_0) $ ensure that
  \begin{equation*}
    \label{eq:cap2:28}
    \sum_{\n=1}^{N_{\h}}a^{\n}_{\h}\le
    {2} |\lambda|\tau_{\h} \left| \sum_{\n=1}^{N_{\h}}\big(\phi(v^{\n-1}_{\h})-\phi(v^{\n}_{\h})\big) \right| \le
    {2} |\lambda|\tau_{\h}\big( \phi(v_0)-\phi(v_{T_k})\big),\qquad
    %b^{\n}_{\h}\le 0,\quad
     \sum_{\n=1}^{N_{\h}}c^{\n}_{\h}\le C_{\h}.
   \end{equation*}
   Furthermore, another telescopic summation entails
   \begin{align*}
     \sum_{\n=1}^{N_{\h}}I^n_\h={}
                                 \phi(v_{T_\h})-\phi(v_{t_\h})+\frac 12\int_{t_\h}^{T_\h}\Big(\MDS
                                 vr+\MetricSlopeSquare\phi{v_r}\Big)\,\d
                                 r
                                 \topref{eq:17}\le{} \phi(v_0)-\phi(v_{t_\h}).
   \end{align*}
   Recalling \eqref{eq:cap2:21}, \eqref{eq:cap2:27}, the lower
   semicontinuity of $\phi$ and the fact that $ \phi(v_T) > -\infty $, 
   we finally obtain
   \begin{displaymath}
     \lim_{k \to \infty}
     \sup_{1\le \n\le N_{\h}} \rme^{\lambda t^{\n}_{\h}}\Dist{v^{\n}_{\h}}{u^{\n}_{\h}}=0 ,
   \end{displaymath}
which shows that $ v_t = \Flow t{v_0} $ for every $t \ge 0$ upon noticing the continuity of the map $(t,u) \mapsto \FlowName_t(u) $, consequence of \eqref{eq:Cap12:1bis} and continuity w.r.t.\ $t$.
 \end{proof}

\begin{remark}[When $ v_0 \not\in \Dom \phi $]\upshape
The case of a curve $ v_t $ such that $ v_0 \not \in \Dom \phi $ can be dealt with similarly. Indeed, more in general, we can say that a curve $ v \in \AC^2_{\rm loc}((0,+\infty);\AmbientSpace) \cap C^0([0,+\infty);\AmbientSpace) $ satisfies the Energy-Dissipation Inequality if 
\begin{equation*}
  \label{eq:17-nodomain}\tag{EDI$_s$}
   v_t\in \Dom\phi, \qquad \liminf_{s \downarrow 0} \left[ \int_s^t \Big(\frac12\MDS vr
  +\frac 12\MetricSlopeSquare\phi{v_r}\Big)\,\d r + \phi(v_t) - \phi(v_s) \right] \le 0
  \quad
  \text{for every }t>0.
\end{equation*}
Note that \eqref{eq:17-nodomain} is consistent with \eqref{eq:17}, in the sense that if $ v_0 \in \Dom\phi $ then it is implied by \eqref{eq:17} (and it is equivalent to the latter provided $ t \mapsto \phi(v_t) $ is continuous at $t=0$). It is then straightforward to check that one can repeat the proof of Theorem \ref{thm:MaxSlopeAreGFlows} up to replacing the initial time $ 0 $ with $s$, eventually letting $ s \downarrow 0 $.
\end{remark} 
 
 \begin{lemma}
   \label{le:RiemannSums}
   Let $g \in L^1((0,T)) $ for some $ T>0 $. There exists a sequence
   $\tau_{\h}\down0$ and a set of points $\bar t\in (0,1)$ with full $ \Leb 1 $-measure such that, by
   choosing $t^n_k:=\tau_k(\bar t+n)$ and $N_k\in \N$ so that $T\in [\tau_k (N_k+2),\tau_k(N_k+3))$, there holds
   \begin{equation*}
     \label{eq:cap2:29}
     \lim_{\h \to \infty}\sum_{\n=0}^{N_{\h}} \left|\tau_{\h} g(t^{\n}_{\h})
     -\int_{t^{\n}_{\h}}^{t_{\h}^{\n+1}} g(r)\,\d r \right|=0.
   \end{equation*}
 \end{lemma}
 \begin{proof}
   First of all, let us trivially extend $g$ by $0$ outside the interval $(0,T)$ (for notational convenience we do not relabel such an extension). It is well known that the bounded map  
   \begin{displaymath}
     y \in [0,+\infty)\mapsto I(y):=\int_0^{T}|g(x)-g(x+y)|\,\d x
   \end{displaymath}
   is continuous; hence a further integration
   with respect to $y$ and Fubini's Theorem yield
   \begin{align*}
     \limsup_{\tau\down0}\int_0^{T}
     \Big|g(x)-\int_0^1 g(x+\tau y)\,\d y\Big|\,\d x
     &\le \limsup_{\tau\down0}
     \int_0^{T}\int_0^1 
     \Big|g(x)-g(x+\tau y)\Big|\,\d y\,\d x
     \\& = \limsup_{\tau\down0}
     \int_0^1I(\tau y)\,\d y 
     =0 .
   \end{align*}
   If we denote by $N(\tau)$ the unique integer such that $T\in [\tau (N(\tau)+2),\tau(N(\tau)+3))$, it follows that
   \begin{displaymath}
     \lim_{\tau\down0}
     \tau \sum_{\n=0}^{N(\tau)} \int_0^1 \Big|g(\tau (x+\n))-
     \int_0^1 g(\tau (x+y+\n))\,\d y\Big|\,\d x=0;
   \end{displaymath}
   in particular, there exists a vanishing subsequence $\tau_{\h}$
   such that the integrand
   \begin{displaymath}
     \sum_{\n=0}^{N(\tau_{\h})} \tau_k\Big| g(\tau_{\h} (x+\n))-
     \int_0^1 g(\tau_{\h} (x+y+\n))\,\d y\Big|
   \end{displaymath}
   converges to $0$ for $\Leb 1$-a.e.\ $x\in (0,1)$.
 \end{proof}

 \section{``Ekeland relaxation'' of the
   Minimizing Movement method and uniform error estimates}
\label{sec:uniform_estimate}

A general variational method to approximate gradient flows (and often prove their existence) for the system $\System$ is 
provided by the so called 
\emph{Minimizing Movement} variational scheme.
In his original formulation (see e.g.\ \cite{DeGiorgi93}), the method consists in finding
a discrete approximation $\Pc U\tau$ of the continuous
gradient flow $u$ by solving a recursive
variational scheme.
In fact $\Pc U\tau$ is a piecewise constant
function on 
the partition
${\mathcal P}_\tau:=\big\{0,\tau,2\tau,\cdots,n\tau,\cdots\big\}$
induced by the time step $\tau>0$; in each interval $((n-1)\tau,n\tau]$ of the partition,
$\Pc U\tau$ takes the value $U^n_\tau$ which
minimizes the functional
\begin{equation}
  \label{eq:cap4:30}
  U\mapsto \frac 1{2\tau}\DistSquare{U}{U^{n-1}_\tau}+\phi(U).
\end{equation}
In order to carry out the iteration, one has to assign the starting value  $ U^0_\tau = \Pc U\tau(0) $, which is supposed to be a suitable approximation of $u_0$. 

The existence of a minimizing sequence $\{U^n_\tau\}_{n\in\N}$
is usually obtained by invoking the direct method of the
Calculus of Variations,
thus requiring that the functional
\eqref{eq:cap4:30}
has compact
sublevels with respect to some
Hausdorff topology $\sigma$ on $X$
(see e.g.\ the setting of \cite[Section 2.1]{Ambrosio-Gigli-Savare08}).
Another possibility, still considered in
\cite{Ambrosio-Gigli-Savare08},
is to suppose that the functional
\eqref{eq:cap4:30} satisfies a strong convexity assumption.

Here we try to avoid these restrictions by applying Ekeland's Variational Principle to
the functional \eqref{eq:cap4:30}, as we did in the Definition 
\ref{def:MYE-resolvent} of the Moreau-Yosida-Ekeland resolvent.
This approach only requires the completeness of the sublevels
of $\phi$; we will show that the sole existence of an \EVIshort\lambda-gradient flow for
$\phi$ provides an explicit error estimate between the
solution to \eqref{eq:EVI} and any discrete approximation obtained by
the variational method, so that we can drop any compactness assumptions.

% We are assuming that the functional $\phi$ has \emph{complete}
% sublevels and it is quadratically bounded from below, according to \eqref{eq:cap1:13}.
\begin{definition}[The Ekeland relaxation of the Minimizing Movement scheme]
  \label{def:MMS}
  Let us consider a ti\-me step $\tau>0$, a relaxation parameter $\eta\ge0$,
  and a discrete initial datum $U^0_\taueta\in \Dom\phi $.
  A $(\tau,\eta)$-discrete Minimizing Movement starting from $U^0_\taueta$
  is any sequence $(U^n_\taueta)_{n\in \N}$ in $\Dom\phi$ s.t.
  \begin{equation*}
    \label{eq:74}
    U^n_\taueta\in \Resolvent\TauEta {U^{n-1}_\taueta}\quad
    \text{for every }n\in \N,
  \end{equation*}
  i.e.\ $U^n_\taueta$ is a solution of
  the family of variational problems
  \begin{subequations}
    \label{problemb}
    \begin{equation}
      \label{eq:cap1:67a}
      \begin{aligned}
        \frac 1{2\tau}\DistSquare{U^{n}_\taueta}{U^{n-1}_\taueta}+
        \phi(U^n_\taueta)&\le \frac 1{2\tau}\DistSquare
        V{U^{n-1}_\taueta}+\phi(V)\\&\qquad+
        \frac \eta2\,\Dist{U^n_\taueta}{U^{n-1}_\taueta}\,\Dist
        V{U^n_\taueta}\quad \text{for every }V\in \Dom\phi ,
      \end{aligned}
    \end{equation}
    satisfying the further condition
    \begin{equation}
      \label{eq:cap4:22a}
      \frac1{2\tau}\DistSquare{U^{n}_\taueta}{U^{n-1}_\taueta}+
      \phi(U^n_\taueta)\le
      \phi(U^{n-1}_\taueta),
    \end{equation}
  \end{subequations}
  for every $n\in \N$. We denote by $\discrMM\TauEta{U^0_\taueta}\subset \AmbientSpace^\N$ the corresponding
  collection of all the discrete minimizing movements that start from $U^0_\taueta$.
\end{definition}
Note that if we pick $\eta=0$ in the previous definition
then we have the usual minimizing movements: in this case,
\eqref{eq:cap4:22a} is a direct consequence of
\eqref{eq:cap1:67a}, since it can easily be obtained by taking
$V:=U^{n-1}_\tau$.
The particular scaling choice of the parameter $\eta$ in
\eqref{eq:cap1:67a}
is motivated by the simpler form of the next estimates,
where $\eta$ can be considered as a mild perturbation of the parameter $\lambda$
and therefore it does not affect the stability on finite intervals and the order of
convergence of the method. In particular, $\eta$ can be taken fixed
and positive as $\tau$ vanishes.

Theorem \ref{le:slope_by_Ekeland} 
ensures that if $\phi$ has complete sublevels and it is quadratically
bounded from below
according to \eqref{eq:cap1:13}, then a $(\tau,\eta)$-discrete Minimizing Movement
$(U^n_\taueta)_{n\in \N}\in \discrMM\TauEta{U^0_\taueta}$
always exists for every $\tau\in (0,\tau_o)$,
$\eta>0$, and $U^0_\taueta\in \Dom\phi$, where $\tau_o$ is given by \eqref{eq:1-bis}. The piecewise constant interpolant $\Pc U\taueta$ is defined as (for any sequence actually, not necessarily a minimizing movement) 
\begin{equation}
  \label{eq:cap1:66}
  \Pc U\taueta(t):= U^n_\taueta
  %\quad
  %\Pcb U\tau(t)\equiv U^{n-1}_\taueta
  \quad
  \text{if }t\in \big((n-1)\tau,n\tau\big], \qquad \Pc U\taueta(0) := U^0_\taueta .
\end{equation}

We can now provide the definition of (continuous) Minimizing Movement.

\begin{definition}[Minimizing Movements]
  \label{def:MM}
  We say that a curve $u:[0,+\infty)\to X$
  is a Minimizing Movement and 
  belongs to $\contMM{\Syst;u_0
  }$ 
  if $ u(0)=u_0 $ and there exist
  $\eta\ge0$ and piecewise constant curves $\Pc U\taueta$
  associated with sequences $(U^n_\taueta)_{n\in \N}\in
  \discrMM\TauEta{u_0}$ for sufficiently small $\tau>0$
  such that
  \begin{equation*}
    \label{eq:37}
    \lim_{\tau\down0}\Pc U\taueta(t)=u(t)\quad\text{for every }t\ge0.
  \end{equation*}
  We say that $u$ is a Generalized Minimizing Movement
  in $\GMM{\Syst;u_0}$
  if $ u(0)=u_0 $ and there exist
  $\eta\ge 0$, a vanishing sequence $k\mapsto \tau(k) \downarrow0$
  and piecewise constant curves $\Pc U k = \Pc U{\tau(k),\eta} $
  associated with sequences $(U^n_{\tau(k),\eta})_{n\in \N}\in
  \discrMM{\tau(k),\eta}{u_0}$ 
  such that
  \begin{equation*}
    \label{eq:37bis}
    \lim_{k\to\infty}\Pc U k(t)=u(t)\quad\text{for every }t\ge0.
  \end{equation*}
\end{definition}

As is typical in numerical analysis of differential equations, uniform
error estimates will result from the combination of the discrete
uniform  stability (as $ \tau \downarrow 0 $) of the method and a
local error estimate, an approach known in the literature as
\emph{Lady Windermere's Fan} (see e.g.~\cite[Page 39]{Hairer}). The purpose of the next two
lemmas is to somehow reproduce such a strategy in our setting.
First of all, we will show that the discrete Minimizing Movements satisfy
an approximate version of the Energy-Dissipation Inequality; from the latter, stability follows (see Proposition \ref{lem:cont-stab}).
% \begin{remark}[Discrete solvability]
%   \label{ass:discrete_solvability}
%   \upshape
%   The ``direct method'' of the Calculus of Variations
%   shows that the variational scheme is always solvable 
%   even for $\eta=0$ if, e.g., the sublevels
%   of $\phi$ are locally relatively compact.
%   More generally, it is sufficient that the functional
%   $V\mapsto (2\tau)^{-1}\DistSquare UV+\phi(V)$ has compact
%   sublevels with respect to some
%   Hausdorff topology $\sigma$ on $X$
%   (see e.g.\ the setting of \cite[\S\, 2.1]{Ambrosio-Gigli-Savare05}).
% \end{remark}
\begin{lemma}%[Discrete \GR{energy dissipation}]
\label{lem:disc-stab} 
  Let the sequence $(U^n_\taueta)_{n\in \N}\in \discrMM\TauEta{U^0_\taueta}$ be a solution of
  the Minimizing Movement scheme of Definition \ref{def:MMS},
  for some $\tau>0$, $\eta\ge0$.
  Then for every $n\in \N$
  \begin{subequations}
    \begin{align}
      \label{eq:5}
      (1-\tfrac12\eta\tau)\MetricSlope\phi{U^n_\taueta}&\le
      \frac{\Dist{U^n_\taueta}{U^{n-1}_\taueta}}{\tau}.
    %\end{align}
      \intertext{Moreover, if $ \phi $ is approximately $ \lambda
                                                         $-convex,
                                                         then } 
  %\begin{align}
      \label{eq:6}
      (1+\tfrac12 (\lambda-\eta) \tau)\frac{\DistSquare{U^n_\taueta}{U^{n-1}_\taueta}}\tau
      &\le
      \phi(U^{n-1}_\taueta)-
      \phi(U^n_\taueta),\\
       \label{eq:cap4:2a}
      (1+(\lambda-\eta)\tau)\frac{\Dist{U^n_\taueta}{U^{n-1}_\taueta}}{\tau}
    &\le \MetricSlope\phi{U^{n-1}_\taueta},\qquad
      (1-\lambda'\tau ) \MetricSlope\phi{U^n_\taueta}
    \le \MetricSlope\phi{U^{n-1}_\taueta} ,
    \end{align}
      \end{subequations}
      where $\lambda'\ge \eta(1+\tfrac 12 \tau \lambda_+)-\lambda$.
    %   In particular,
    % the sequence $(U^n_\taueta)_{n\in \N}$ satisfies the discrete
    % Energy-Dissipation inequality
    % \begin{equation}
    %   \label{eq:39}
    %   (1+\tfrac 12\betag\tau)\Big(\frac
    %   \tau2\frac{\DistSquare{U^n_\taueta}{U^{n-1}_\taueta}}{\tau^2}
    %   +\frac \tau2 \MetricSlope\phi{U^n_\taueta}^2\Big)\le \phi(U^{n-1}_\taueta)-
    %   \phi(U^n_\taueta)
    % \end{equation}
    % for every $\betag\le \lambda$ such that
    % \begin{equation}
    %   \label{eq:40}
    %   \eta(3+\tau\lambda_+)\le \lambda-\betag.
    % \end{equation}
\end{lemma}

Note that when $\eta=0$ we can choose $\lambda'=-\lambda$ in the above
estimates. On the other hand, when $\lambda>0$ we can always pick $\eta$ sufficiently
small (independently of $ \tau $ ranging in a bounded interval) so that $\lambda-\eta\ge 0$ and $\lambda'=0$.

\begin{proof}
Inequality \eqref{eq:5} follows directly from \eqref{eq:cap1:4}. In order to show \eqref{eq:6}, for every $ \vartheta,\eps \in (0,1) $ we take a $ (\vartheta,\eps) $-intermediate point $ \sfU^{n-1,n}_{\vartheta,\eps} $ between $ U^{n-1}_\taueta $ and $ U^n_\taueta $ such that
\begin{equation}\label{eq:app-conv-1}
    \phi(\sfU^{n-1,n}_{\vartheta,\eps})\le (1-\vartheta)\phi(U^{n-1}_\taueta)+\vartheta\phi(U^n_\taueta)-\frac{\lambda-\eps}2
    \vartheta(1-\vartheta)\DistSquare{U^{n-1}_\taueta}{U^{n}_\taueta} . 
\end{equation} 
Now note that by picking $ V=\sfU^{n-1,n}_{\vartheta,\eps} $ in \eqref{eq:cap1:67a}, setting $\tilde\eta:=\tfrac \eta2\Dist{U^n_\taueta}{U^{n-1}_\taueta}$, exploiting \eqref{eq:app-conv-1} and 
$$ \Dist{\sfU^{n-1,n}_{\vartheta,\eps}}{U^{n-1}_\taueta} \le \vartheta (1+\eps) \Dist{U^{n-1}_\taueta}{U^{n}_\taueta} , \qquad \Dist{\sfU^{n-1,n}_{\vartheta,\eps}}{U^{n}_\taueta} \le (1-\vartheta) (1+\eps) \Dist{U^{n-1}_\taueta}{U^{n}_\taueta} ,  $$
we end up with  
  \begin{subequations}
    \begin{align}
      \notag\frac
      1{2\tau}\DistSquare{U^n_\taueta}{U^{n-1}_\taueta}+\phi(U^n_\taueta)&\le
      \frac
      1{2\tau}\DistSquare{\sfU^{n-1,n}_{\vartheta,\eps}}{U^{n-1}_\taueta}+\phi(\sfU^{n-1,n}_{\vartheta,\eps})+
      \tilde\eta \Dist{\sfU^{n-1,n}_{\vartheta,\eps}}{U^n_\taueta} 
      \\\label{eq:7}&\le
      \frac
      {\vartheta^2(1+\eps)^2}{2\tau}\DistSquare{U^{n-1}_\taueta}{U^{n-1}_\taueta}+
      (1-\vartheta)\phi(U^{n-1}_\taueta)+\vartheta\phi(U^n_\taueta)
      \\\label{eq:8}&\quad-\frac
      {\lambda-\eps}{2}\vartheta(1-\vartheta)\DistSquare{U^{n-1}_\taueta}{U^n_\taueta}
      +\tilde\eta (1-\vartheta)(1+\eps)\Dist{U^{n-1}_\taueta}{U^n_\taueta} .
    \end{align}
  \end{subequations} 
First one lets $ \varepsilon \downarrow 0 $ and then observes that the corresponding second-order polynomial in $\vartheta\in [0,1]$ defined by the right-hand side \eqref{eq:7}--\eqref{eq:8} evaluated at $ \eps=0 $ attains a minimum at $\vartheta=1$; inequality 
  \eqref{eq:6} then follows by taking the (left) derivative with
  respect to $\vartheta$ at $\vartheta=1$.

Thanks to Proposition \ref{lem:app-conv-slope}, we know in particular that $ \MetricSlope{\phi}{U^{n-1}_\taueta} = \GSlope{\lambda}{\phi}{U^{n-1}_\taueta} $, whence  
  \begin{displaymath}
    \phi(U^{n-1}_\taueta)-\phi(U^n_\taueta)\le \MetricSlope\phi{U^{n-1}_\taueta}\Dist{U^{n-1}_\taueta}{U^n_\taueta}-\frac\lambda 2\DistSquare{U^{n-1}_\taueta}{U^n_\taueta}
  \end{displaymath}
  and therefore, by exploiting \eqref{eq:6} 
  \begin{equation}\label{eq:8-bis}
    \left(1+(\lambda-\tfrac 12\eta)\tau\right) \frac{\DistSquare{U^n_\taueta}{U^{n-1}_\taueta}}\tau
      \le \MetricSlope\phi{U^{n-1}_\taueta}\Dist{U^{n-1}_\taueta}{U^n_\taueta};
  \end{equation}
we thus get the first inequality of \eqref{eq:cap4:2a} since $ \eta \ge 0 $. The second inequality is
implied by \eqref{eq:5} and \eqref{eq:8-bis}, which yield
\begin{equation}\label{eqw}
  \left(1+(\lambda-\tfrac12 \eta)\tau\right) \left(1-\tfrac12
    \eta\tau\right)
  \MetricSlope{\phi}{U^{n}_\taueta}\le
  \MetricSlope{\phi}{U^{n-1}_\taueta},
\end{equation}
upon observing that $\lambda-\eta- \tfrac 12\eta(\lambda-\tfrac12\eta)\tau\ge -\lambda'$. Actually \eqref{eqw} may not hold if the left-hand sides of both \eqref{eq:5} and \eqref{eq:8-bis} are negative; however, the second inequality of \eqref{eq:cap4:2a} is trivially true.
%
% \GGG 
% \eqref{eq:39} then follows by combining \eqref{eq:5} and \eqref{eq:6},
% observing that
% \begin{align*}
%   (1-\tfrac 12\eta\tau)^2(1+\tfrac 12(\lambda-\eta)\tau)
%   &\ge
%   1-\eta\tau+\tfrac 12 (\lambda-\eta)\tau
%   -\tfrac 12 \eta(\lambda-\eta)\tau^2
%     \ge 1+\tfrac 12\tau(\lambda-3\eta-\eta\lambda\tau)
%     \\&
%   \ge 1+\tfrac 12\beta\tau.
% \end{align*}
\end{proof} 
The second step consists in providing the fundamental \emph{local error estimate}
between $ (\tau,\eta) $-minimizing movements and the
$\EVIshort\lambda$-gradient flow, just upon assuming
a relaxed version of \eqref{eq:5} and \eqref{eq:6}. Here the refined
local asymptotic expansion \eqref{eq:3} turns out to be the crucial ingredient of
the estimate. From now on, for simplicity we will only consider the case $\lambda\le 0$: this is not
restrictive, since a solution of $\EVIshort\lambda$ for $\lambda>0$
also solves $\EVIshort0$. 
{\renewcommand{\taueta}{{}}
\begin{lemma}
%[Local error estimates]
  \label{le:local-estimate}
  Let $u\in {\rm Lip}([0,+\infty);\Dom\phi)$ be a solution of the
  Evolution Variational Inequality
  $\EVI\AmbientSpace\DistName\phi\lambda$, $\lambda\le 0$,
  with $u_0\in \DomainSlope\phi$
  and let $U_\taueta\in \Dom{|\partial\phi|}$ satisfy the estimates
  \begin{equation} \label{eq:10}
    \begin{aligned}
      % (1+\tfrac 12\betag\tau)\Big(\frac
      % \tau2\frac{\DistSquare{U_\taueta}{u_0}}{\tau^2}
      %   +\frac \tau2 \MetricSlope\phi{U_\taueta}^2\Big)\le \phi(u_0)-
      %   \phi(U_\taueta)+\eps
\tau (1-\tfrac 12\eta\tau)^2 \MetricSlopeSquare\phi{U_\taueta} \le
    \frac{\DistSquare{U_\taueta}{u_0} } \tau +\varepsilon  , \qquad 
    (1-\tfrac 12\beta \tau)\frac{\DistSquare{U_\taueta}{u_0}}\tau
    \le 
    \phi(u_0)-
    \phi(U_\taueta) +\varepsilon 
  \end{aligned}
  \end{equation}
  for some $ \tau>0 $, $\eps\ge 0$, $\beta,\eta\in [0,2/\tau)$.
  If $ (\eta+\beta-2\lambda)\tau<1$, then
  for every $\alpha \le 0 $ complying with $2\alpha \le \tau^{-1}\log(1-(\eta+\beta-2\lambda)\tau)$ there holds
  \begin{equation}
    \label{eq:11}
    \rme^{2\lambda \tau}\DistSquare{U_\taueta}{u_\tau}\le
    \tau^2\Big(\MetricSlopeSquare\phi{u_0}-\rme^{2\alpha \tau}\MetricSlopeSquare\phi{U_\taueta}\Big) +3\varepsilon\tau .
    \end{equation}
  \end{lemma}
Before proceeding to prove the lemma, let us point out that estimate \eqref{eq:11} looks particularly simple in the case $\lambda=\eta=\beta=0$:
\begin{equation*}
  \label{eq:12}
  \DistSquare{U_\taueta}{u_\tau}\le
  \tau^2\Big(\MetricSlopeSquare\phi{u_0}-\MetricSlopeSquare\phi{U_\taueta}\Big) + 3\varepsilon\tau .  
\end{equation*}  
\begin{proof}
  The local expansion \eqref{eq:3} with $v:=U_\taueta$ yields
  \begin{equation}
    \label{eq:13}
    \frac{\rme^{2\lambda\tau}}2\DistSquare{u_\tau}{U_\taueta}-
    \frac 12\DistSquare{u_0}{U_\taueta}+\sfE{2\lambda}\tau\Big(\phi(u_0)-\phi(U_\taueta)\Big)\le
    \frac{\tau^2}2\MetricSlopeSquare\phi{u_0} .
  \end{equation}
  Multiplying the second inequality of \eqref{eq:10} by $2\sfE{2\lambda}\tau$ and summing it to \eqref{eq:13} we obtain
  \begin{equation}
    \label{eq:14}
    \rme^{2\lambda\tau}\DistSquare{u_\tau}{U_\taueta}\le \tau^2\MetricSlopeSquare\phi{u_0}-c_1 \DistSquare{u_0}{U_\taueta}  + 2\sfE{2\lambda} \tau \, \varepsilon 
    %+2\eta\sfE{2\lambda}\tau\Dist{u_0}{U_\taueta}.
 \end{equation}
 where $c_1:=2\tau^{-1}(1-\tfrac 12\beta \tau)\sfE{2\lambda}\tau-1$.
 Using the elementary inequalities
 \begin{displaymath}
   \tau(1+\lambda\tau)\le \sfE{2\lambda}\tau\le \tau,\quad
    1 \ge c_1\ge 1 + (2\lambda-\beta)\tau,
   %2\eta\sfE{2\lambda}\tau\Dist{u_0}{U_\taueta}\le
   %\eta\tau\DistSquare{u_0}{U_\taueta}+\eta\tau
 \end{displaymath}
 % we get $$ and
%  \begin{equation}
%    \label{eq:15}
%     \rme^{2\lambda\tau}\DistSquare{u_\tau}{U_\taueta}\le \tau^2\MetricSlopeSquare\phi{u_0}-c_2 \DistSquare{u_0}{U_\taueta}
%     +\eta\tau,\quad
%     c_2:=1+(3\lambda-\eta)\tau.
%  \end{equation}
%  We then use 
and the first inequality of \eqref{eq:10}, which yields
 \begin{equation*}
   \label{eq:16}
   \DistSquare{U_\taueta}{u_0} \ge (1-\eta\tau)\tau^2\MetricSlopeSquare\phi{U_\taueta} -\varepsilon \tau ,
 \end{equation*}
 from \eqref{eq:14} we obtain \eqref{eq:11}.
\end{proof}
}
%\Comment{Probabilmente sara' molto comodo poi assumere $ 2\beta\tau \le 1-e^{-1} $, sfruttando anche il fatto che cosi' facendo $ \sftau{\tau}{-\beta}t \le 2t $, oppure $ \sftau{\tau}{-\beta}t \le 2 \sfE{-2\beta}{t} $. Utile anche la stima $ \sfE{-\beta}{t} \ge e^{-\beta t} \, t $. Forse puo' servire nel seguito usare il fatto che $ \sfE{-\beta}{t} $ decresce al crescere di $ \beta $, e uno puo' sempre assumere che $ -\lambda $ e $ \eta $ siano sufficientemente grandi, se serve.}

\subsection[Error estimates and convergence for initial data in
  $\DomainSlope\phi$]{Error estimates and convergence for initial data in
  $\boldsymbol{\DomainSlope\phi}$}\label{s1s}

In the next result we establish uniform error estimates between $
\varepsilon $-relaxed $ (\tau,\eta) $-minimizing movements (in the
sense that they are assumed to satisfy only \eqref{eq:10} at each
discrete time step) and the gradient flow.
Here we consider the case of \emph{regular} discrete initial data, namely we suppose that $U_{\taueta}^0\in \Dom{|\partial\phi|}$.

For convenience, let us introduce the following notations:
\begin{equation}
  \label{eq:46}
  \gamma:=2\eta-3\lambda,\qquad
  t_\tau:=\min\{k\tau:k\in \N,\ k\tau\ge t\}\quad\text{for every }t\ge0.
\end{equation}

\begin{theorem}[Uniform error estimate for regular data]\label{lem:Error-est-approx}
Let the following assumptions hold:
  \begin{enumerate}[\rm 1.]
  \item the \EVIshort\lambda-gradient flow $(\lambda\le 0)$ $\FlowName_t$ of $\phi$ exists in
    $D\subset \overline{\Dom\phi}$;
  \item for some $\tau>0$, $\eta\ge0$ and $ \varepsilon\ge0 $
    such that $ 4\gamma\tau \le  1 $,
    the sequence $(U^n_\taueta)_{n\in \N} \subset D \cap \DomainSlope{\phi}  $ satisfies for all $ n \in \mathbb{N} $ the estimates 
  \begin{equation} \label{eq:10-est}
  %\begin{aligned}
    \tau (1-\tfrac 12\eta\tau)^2 \MetricSlopeSquare\phi{U_\taueta^n} \! \le \!
    \frac{\DistSquare{U_\taueta^n}{U_\taueta^{n-1}} } \tau +\varepsilon  , \quad
    \big(1-\tfrac {\eta-\lambda}2 \tau\big)\frac{\DistSquare{U_\taueta^n}{U_\taueta^{n-1}}} \tau \!
    \le \!
    \phi(U_\taueta^{n-1})-
    \phi(U_\taueta^{n}) +\varepsilon  .
  %\end{aligned}
  \end{equation}
  \end{enumerate} 
Then, for all $ u_0 \in D $, the following estimate holds for every $ t>0 $:
   \begin{equation} \label{eq:24:epsilon}
\Dist{\FlowName_{t}(u_0)}{\Pc U\taueta(t)}
%     \le \! \rme^{-\lambda(t+\tau)} \left[ \tau \MetricSlope{\phi}{u_0} +  \rme^{2\beta(t+\tau)} \big( \Dist{u_0}{U^0_\taueta} + \sqrt {(t+\tau)\tau} \MetricSlope\phi{U^0_\taueta} \big) \right]  \quad \forevery t > 0 . \\
\le \rme^{-\lambda t} \Dist{u_0}{U^0_\taueta}  + \left( \sqrt{\tau
       t_\tau} + t_\tau - t \right) \rme^{\gamma t_\tau} \,
     \MetricSlope\phi{U^0_\taueta} +
     2\sqrt{\frac{\varepsilon}{\tau} t_\tau \, \sfE{2 \gamma}{t_\tau} } .
     % {[1-(2\eta-3\lambda)\tau]^{\frac{2\eta-3\lambda}{2}(t+\tau)}}  
   \end{equation}
\end{theorem}
\begin{proof}
  Thanks to \eqref{eq:10-est}, by applying Lemma
  \ref{le:local-estimate} we obtain the validity of estimate
  \eqref{eq:11} with $ u_0 $ replaced by $ U^{n-1}_\taueta $, $
  U $ replaced by $ U^n_\taueta $, $ u_\tau $ replaced by $ \FlowName_\tau(U^{n-1}_\taueta) $  (for all $  n \in \N $) and
  $ \beta = \eta - \lambda $. Let us set $\gamma_\tau:=-\tau^{-1}\log(1-\gamma\tau)$. Upon multiplying \eqref{eq:11} by $
  \rme^{2\alpha(n-1)\tau} $ and noticing that there holds $ 2\alpha\le
  -\gamma_\tau\le- \gamma% \tau^{-1}\log (1+(3\lambda-2\eta)\tau)
  \le 2\lambda $, we end up with  
  \begin{equation*} 
  \begin{gathered}
    \rme^{2\alpha n\tau}\DistSquare{U^{\n}_\taueta}{\FlowName_{\tau}(U^{n-1}_\taueta)} \le
    \tau^2 ( A^{\n} + \varepsilon B^{\n} ) , \\
   \text{where} \quad  A^{\n} := 
    \rme^{2\alpha (n-1)\tau}\MetricSlopeSquare\phi{U^{\n-1}_\taueta}-
    \rme^{2\alpha n\tau}\MetricSlopeSquare\phi{U^{\n}_\taueta} \quad \text{and} \quad B^{\n} := \frac{3 \rme^{2\alpha(n-1)\tau}}{\tau} .
    \end{gathered}
  \end{equation*}
On the other hand, since $\alpha\le \lambda$ and it is not restrictive to assume $\alpha \le 0$; if we set $E^n:=\rme^{\alpha n\tau}\Dist{U^n_\taueta}{\FlowName_{n\tau}(U^{0}_\taueta)}$, we obtain: 
  \begin{align*} 
    E^n&\le
    \rme^{\alpha n\tau}\Dist{U^n_\taueta}{\FlowName_\tau(U^{n-1}_\taueta)}
    +\rme^{\alpha n\tau}\Dist{\FlowName_\tau(U^{n-1}_\taueta)}{\FlowName_{\tau}(\FlowName_{(n-1)\tau}(U^{0}_\taueta))}
    \\& \le
    \rme^{\alpha n\tau}\Dist{U^n_\taueta}{\FlowName_\tau(U^{n-1}_\taueta)}
    +\rme^{\alpha(n-1)\tau}\Dist{U^{n-1}_\taueta}{\FlowName_{(n-1)\tau}(U^{0}_\taueta)}
    \\& \le
    \rme^{\alpha n\tau}\Dist{U^n_\taueta}{\FlowName_\tau(U^{n-1}_\taueta)}
    +E^{n-1} ,
  \end{align*}
  so that if $  m \in \N $ is such that $ t \in ((m-1)\tau,m\tau] $, or equivalently $t_\tau=m\tau$, there holds (note that $ E^0 = 0 $)
  \begin{equation*}\label{est-E:epsilon}
  \begin{aligned}
    E^m \le & \sum_{n=1}^m \rme^{\alpha n\tau}\Dist{U^n_\taueta}{\FlowName_\tau(U^{n-1}_\taueta)}\le
    \tau \sum_{n=1}^m( A^n + \eps B^n )^{1/2}
    \le  \tau \sum_{n=1}^m ( A^n )^{1/2} +  \tau \sqrt{\varepsilon} \sum_{n=1}^m ( B^n )^{1/2}   \\
    \le &
    \tau \sqrt {m} \Big(\sum_{n=1}^m A^n  \Big)^{1/2} + \sqrt {3 m \varepsilon \tau}  \Big(\sum_{n=1}^m \rme^{2\alpha(n-1)\tau}  \Big)^{1/2} 
    \le
    \sqrt {\tau t_\tau}\MetricSlope\phi{U^0_\taueta} + \sqrt{3  \varepsilon t_\tau \, \frac{1-\rme^{2\alpha t_\tau}}{1-\rme^{2\alpha\tau}} } , 
    \end{aligned}
  \end{equation*}
whence recalling \eqref{eq:cap1:81} and \eqref{eq:3}, 
\begin{equation*}\label{est-E2:epsilon} 
\begin{aligned}
\Dist{ \Pc{U}{\taueta}(t)}{\FlowName_{t}(u_0)} \le & \Dist{\FlowName_t(U^0_\taueta)}{\FlowName_t(u_0)} + \rme^{-\lambda t} \Dist{\FlowName_{m\tau-t}(U^0_\taueta)}{U^0_\taueta} + \Dist{U^m_\taueta}{\FlowName_{m\tau}(U^0_\taueta)}  \\
\le & 
\rme^{-\lambda t} \Dist{u_0}{U^0_\taueta} + (t_\tau-t) \, \rme^{-\lambda
  t_\tau} \MetricSlope{\phi}{U^0_\taueta} + \rme^{-\alpha t_\tau}
E^m \\
\le & 
\rme^{-\lambda t} \Dist{u_0}{U^0_\taueta} + 
% \tau \rme^{-\lambda
  % \tau} \MetricSlope{\phi}{U^0_\taueta} 
 \left( \sqrt{\tau t_\tau} + t_\tau -t \right) \rme^{-\alpha t_\tau}
 \MetricSlope\phi{U^0_\taueta}
 + \sqrt{3\varepsilon t_\tau\, \frac{\rme^{-2\alpha t_\tau}-1}{1-\rme^{2\alpha\tau}}} ,
\end{aligned}
\end{equation*}
namely \eqref{eq:24:epsilon} upon choosing $ 2\alpha = -\gamma_\tau$,
exploiting the inequality
\begin{align*}
  \frac{1-\rme^{2\alpha\tau}}{2\alpha \tau}=
  \frac 1\tau\int_0^\tau \rme^{2\alpha r}\,\d r\ge
  \rme^{2\alpha \tau}=\rme^{-\tau \gamma_\tau}\ge
  (1-\gamma\tau)\ge 3/4
\end{align*}
and observing that (we use \eqref{eq:est-elm-new} below)
$$
\gamma_\tau = \log {(1-\gamma\tau)^{-\frac 1 \tau}}  \le \log \left[ \rme^\gamma (1-\gamma\tau)^{-\gamma} \right] \le 2\gamma
$$
since $ 1-\gamma \tau \ge \rme^{-1} $.
% \eqref{eq:est-elm} with $ n=1 $, $ x=(2\eta-3\lambda)\tau $ and using the assumption $ (2\eta-3\lambda)\tau \le 1-\rme^{-1} $.
\end{proof}

Let us make explicit two important consequences of the previous
result, keeping in mind the notations \eqref{eq:46} for $\gamma$ and $t_\tau$.

\begin{corollary}[Error estimate for the Minimizing Movement scheme]
  \label{cor:Error_Estimate}
  Let us suppose that $\Dom\phi$ is an approximate
  length subset of $X$, and 
  the \EVIshort\lambda-gradient flow $(\lambda\le 0)$ $\FlowName_t$ of $\phi$ exists in 
  $\overline{\Dom\phi}$.
If, for some $\tau>0$ and $\eta\ge0$ satisfying $ 4\gamma\tau\le 1$,
the sequence $(U^n_\taueta)_{n\in \N}\subset \Dom{\phi} $ is a
$(\tau,\eta)$-discrete Minimizing Movement, according to Definition
\ref{def:MMS}, then for all $ u_0 \in \overline{\Dom\phi} $ the following estimate holds:
  \begin{equation} \label{eq:24-1}
   \begin{aligned}
     \Dist{\FlowName_t(u_0)}{\Pc U\taueta(t)}
%     \le & \rme^{-\lambda(t+\tau)} \left[ \tau \MetricSlope{\phi}{u_0} +  \rme^{2\beta(t+\tau)} \big( \Dist{u_0}{U^0_\taueta} + \sqrt {(t+\tau)\tau} \MetricSlope\phi{U^0_\taueta} \big) \right]  \quad \forevery t > 0 \\
     \le \rme^{-\lambda t} \Dist{u_0}{U^0_\taueta}  + 
     \left( \sqrt{\tau t_\tau} + t_\tau - t \right) \rme^{\gamma t_\tau} \MetricSlope\phi{U^0_\taueta} \quad \forevery t > 0.
     % {[1-(2\eta-3\lambda)\tau]^{\frac{2\eta-3\lambda}{2}(t+\tau)}}  
   \end{aligned} 
 \end{equation}
\end{corollary}
\begin{proof}
  Since $\Dom\phi$ is an approximate length subset, 
  we know by Theorem \ref{thm:Daneri} that
  $\phi$ is approximately $\lambda$-convex.
  The sequence $(U^n_\taueta)_{n\in \N}$ thus satisfies
  the \emph{a priori} estimates of Lemma \ref{lem:disc-stab}, so that we can apply 
  Theorem \ref{lem:Error-est-approx} with $ \varepsilon = 0 $.
\end{proof}

In order to appreciate the strength of \eqref{eq:24-1}, let us consider
the case where $\lambda=0$ and $\eta=0$, so that the sequence
$n\mapsto U^n_\tau=U^n_{\tau,0}$ is in
fact a solution of the usual Minimizing Movement scheme.
For a fixed final time $t$, we choose a uniform partition of step size
$\tau=t/n$ and as initial datum $U^0_\tau=u_0\in
\Dom{|\partial\phi|}$, obtaining
\begin{equation*}
  \label{eq:47}
  \Dist{\FlowName_t(u_0)}{ U^n_\tau}
  =\Dist{\FlowName_t(u_0)}{\Pc U\tau(t)}
  \le \frac{t}{\sqrt{n}} 
  \, \MetricSlope\phi{u_0},
\end{equation*}
which reproduces (with a better constant) the celebrated
Crandall-Liggett estimate for the generation of
contraction semigroups in Banach spaces governed by $m$-accretive
operators.

As a second consequence, we are able to compare the $\EVIshort\lambda$-formulation with the
notion of Minimizing Movements recalled in Definition \ref{def:MM}.

\begin{corollary}[Existence of Minimizing Movements]
  \label{cor:GMM=EVI}
  Let us suppose that $\phi$ has complete sublevels,
  $\Dom\phi$ is an approximate length subset of $X$, and 
  the \EVIshort\lambda-gradient flow $(\lambda\le 0)$ $\FlowName_t$ of $\phi$ exists in 
  $\overline{\Dom\phi}$.
  Then for every $u_0\in \Dom{|\partial\phi|}$ the sets
  $\GMM{\Syst;u_0}$ and $\contMM{\Syst;u_0}$ coincide and contain as a
  unique element the curve $(\FlowName_t (u_0))_{t\ge0}$.
\end{corollary}
\begin{proof}
  As in the previous corollary, we know that $\phi$ is approximately
  $\lambda$-convex; since $\phi$ has complete sublevels,
  it is quadratically bounded from below by Theorem \ref{le:bfb}
  and the set of
  $(\tau,\eta)$-Minimizing Movements is surely not empty
  thanks to Theorem \ref{le:slope_by_Ekeland}, at east if $\eta>0$
  and $\tau$ is sufficiently small.
 Estimate \eqref{eq:24-1} then shows that $\lim_{\tau\down0}\Dist{\Pc
    U\taueta(t)}{\FlowName_t(u_0)}=0$
  for every $t\ge0$, so that the curve $t\mapsto \FlowName_t(u_0)$ is
  the unique element of $\contMM{\Syst;u_0}$ and $\GMM{\Syst;u_0}$.
\end{proof}
We point out that one could also drop the length assumption on $\Dom\phi$
in the previous Corollaries \ref{cor:Error_Estimate} and \ref{cor:GMM=EVI} (provided $ u_0 \in \overline{\Dom\phi}^{\lDistName} $), by replacing the minimizing movements generated by $\DistName$ 
with the corresponding ones generated by $\lDistName$ in $ X=\overline{\Dom\phi}^{\lDistName} $: it is enough to apply Theorems \ref{thm:self-improvement} and \ref{thm:Daneri}. Note that the sublevels of $ \phi $ stay complete also w.r.t.~$ \lDistName $ if they were w.r.t.~$ \DistName $.

\subsection[Error estimates and convergence for initial data in
  $\Dom\phi$]{Error estimates and convergence for initial data in
  $\boldsymbol{\Dom\phi}$}
\label{subsec:refined-estimates}
The error estimates we obtained in the previous Theorem \ref{lem:Error-est-approx}
involve the slope of the initial datum.
One can expect that for less regular initial data, namely belonging
only to $ \Dom \phi $,
a lower-order error estimate should still be available.
Usually, such weaker estimates can be derived by combining the stronger ones
with suitable contraction properties of the discrete scheme, through
an interpolation technique.
In our situation, however, only the continuous semigroup exhibits
contraction properties of the type of \eqref{eq:cap1:81}, but we do not
know if the Minimizing Movement scheme shares an analogous good
dependence on perturbations of the initial condition. In this regard, we are only able to prove estimate \eqref{mm-c-2} below.

To overcome this difficulty, we adopt a different approach, which will allow us to prove Theorem \ref{thm:Error_Estimate} below: because calculations are much more complicated with respect to Section \ref{s1s}, here we do not aim at finding optimal constants. First of all, we establish the following refined version of Lemma \ref{lem:disc-stab}, where the rate of approximation depends on the \emph{Moreau-Yosida} regularization \eqref{eq:1} of $\phi$ and the times at which the minimizing movements are considered are not necessarily consecutive (compare with \emph{a priori} estimates and asymptotic expansions of Theorem \ref{thm:main1}, see also Remark \ref{rem-apriori}). 

\begin{proposition}[Refined continuous stability estimates]\label{lem:cont-stab}
  Let the assumptions of Lemma \ref{lem:disc-stab} be fulfilled
  with $\lambda\le 0$ and $\beta:=\lambda'=\eta-\lambda$.
  Suppose in addition that $ \phi $ is quadratically bounded from
  below for all $ \kappa_o > \beta $. If $  4 \beta \tau \le 1 $ then for
  every $t,s>0$ the
  following estimates hold:
\begin{align}\label{mm-c-1}
\frac{1}{2} \DistSquare{ \Pc U\taueta(t)}{ U^0_\taueta } &\le \rme^{2\beta t_{\tau,\beta}} \, \sfE{-\beta}{t_{\tau,\beta}} \, \big[ \phi(U^0_\taueta) - \phi_{ \sfE{-\beta}{t_{\tau,\beta}} }(U^0_\taueta) \big],  \\
  \label{mm-c-2}
\frac{1}{2} \DistSquare{\Pcshift{U}{\taueta}{s}(t)}{ \Pc U\taueta(t) } 
&\le \rme^{2\beta(t_{\tau,\beta}+s_{\tau,\beta})+\eta t_{\tau,\beta}} \, \sfE{-\beta}{s_{\tau,\beta}}  \, \big[ \phi(U^0_\taueta) - \phi_{ \sfE{-\beta}{s_{\tau,\beta}} }(U^0_\taueta) \big], \\% \quad \forevery t,s > 0 ,
\label{mm-c-3}
\frac{1}{2} \MetricSlopeSquare{\phi}{\Pc U \taueta(t) } & \le  \left( 1+2\eta\tau \right) \frac{\rme^{2\beta t_{\tau,\beta}}}{\sfE{-\beta}{t_\tau}} \, \big[ \phi(U^0_\taueta) - \phi_{ \sfE{-\beta}{t_{\tau,\beta}} }(U^0_\taueta) \big], 
\end{align}
where by $ \Pcshift{U}{\taueta}{s}(\cdot) $ we denote the piecewise constant interpolant of the $(\tau,\eta)$-discrete Minimizing Movement $ \big(U_\taueta^{n+h}\big)_{n\in \N} $ starting from $ \Pc U\taueta(s) $ (let $ h := s_\tau/\tau $) and for every $ t>0 $ we set
\begin{equation*}\label{eq:def-sftau}
t_{\tau,\beta}:= (1+4\beta\tau) t_\tau .
\end{equation*}
%$$
%A_{\tau,\beta}(t) := \frac{\rme^{2\beta
%    \tau}}{(1-2\beta\tau)^{2\beta(t_\tau)}} , \qquad
%C_{\tau,\beta,\eta}(t):= \frac{\rme^{2\beta \tau}}{(1-2\beta\tau)^{2\beta(t_\tau)} \, (1-\tfrac 1 2 \eta \tau )^2} ,
%$$
%$$
%B_{\tau,\beta,\eta}(t,s) := \frac{e^{4\beta\tau+2\eta\tau}}{(1-\beta \tau)^{2\beta(t_\tau)} \, (1-\tfrac 1 2 \eta \tau )^{2\eta(t_\tau)} \, (1-2\beta\tau)^{2\beta(s_\tau)}} .
%$$
\end{proposition}
\begin{proof}
%The key point is to estimate from above the quantity
%$$ -\sum_{k=1}^n (1-\beta \tau)^k \phi(U^k_\taueta) . $$
We can suppose with no loss of generality that $ \beta>0 $: the estimates corresponding to $ \beta=0 $ just follow by letting $ \beta \to 0 $ (i.e.\ $ \eta,\lambda \to 0 $) in \eqref{mm-c-1}--\eqref{mm-c-3}. So, to begin with, for all $ k \in \N $ let us write 
\begin{equation}\label{eq:est-disc-1}
 - (1-\beta \tau)^k \phi(U^k_\taueta) + (1-\beta \tau)^{k-1} \phi(U^{k-1}_\taueta) = \frac{\beta\tau}{1-\beta\tau} (1-\beta \tau)^k \phi(U^k_\taueta) + (1-\beta\tau)^{k-1} \big[ \phi(U^{k-1}_\taueta) - \phi(U^{k}_\taueta) \big] .
\end{equation}
Now note that estimate \eqref{eq:6} entails
\begin{equation}\label{eq:est-disc-2}
\begin{aligned}
& (1-\beta\tau)^{k-1} \big[ \phi(U^{k-1}_\taueta) - \phi(U^{k}_\taueta) \big] \\
\ge & \frac1{2\tau} (1-\tfrac12 \beta\tau) (1-\beta\tau)^{k-1} \DistSquare{U^k_\taueta}{U^{k-1}_\taueta} + \frac12 (1-\beta\tau)^{k-1} \big[ \phi(U^{k-1}_\taueta) - \phi(U^{k}_\taueta) \big] ,
\end{aligned}
\end{equation}
so that by combining \eqref{eq:est-disc-1}, \eqref{eq:est-disc-2} and summing up we obtain for all $ n \in \N $
\begin{equation}\label{eq:est-disc-3}
\begin{aligned}
& \phi(U^0_\taueta) - (1-\beta \tau)^n \phi(U^n_\taueta) \\
\ge & \frac{\beta\tau}{1-\beta\tau} \sum_{k=1}^n (1-\beta \tau)^k \phi(U^k_\taueta) \\
   & + \frac1{2\tau} (1-\tfrac12 \beta\tau) \sum_{k=1}^n (1-\beta\tau)^{k-1} \DistSquare{U^k_\taueta}{U^{k-1}_\taueta} + \frac12 \sum_{k=1}^n (1-\beta\tau)^{k-1} \big[ \phi(U^{k-1}_\taueta) - \phi(U^{k}_\taueta) \big] .
\end{aligned}
\end{equation}
It is then direct to check that there holds (discrete integration by parts)
\begin{equation}\label{eq:est-disc-4}
\sum_{k=1}^n (1-\beta\tau)^{k-1} \big[ \phi(U^{k-1}_\taueta) - \phi(U^{k}_\taueta) \big] = \phi(U^0_\taueta) - (1-\beta \tau)^n \phi(U^n_\taueta) - \frac{\beta\tau}{1-\beta\tau} \sum_{k=1}^n (1-\beta \tau)^k \phi(U^k_\taueta) .
\end{equation}
Moreover, 
\begin{equation}\label{eq:est-disc-5}
\DistSquare{U^n_\taueta}{U^0_\taueta} \le \left( \sum_{k=1}^n \Dist{U^k_\taueta}{U^{k-1}_\taueta} \right)^2 \le \frac{1-\beta\tau}{\beta\tau} \big[ (1-\beta\tau)^{-n}-1 \big] \sum_{k=1}^n (1-\beta\tau)^{k-1} \DistSquare{U^k_\taueta}{U^{k-1}_\taueta} .
\end{equation}
Hence \eqref{eq:est-disc-3}--\eqref{eq:est-disc-5} yield 
\begin{equation}\label{eq:est-disc-6}
\begin{aligned}
(1-\beta\tau)^n\left[ \frac{\beta}{2[1-(1-\beta\tau)^n]} \DistSquare{U^n_\taueta}{U^0_\taueta}  + \phi(U^n_{\taueta}) \right] & \\
- \frac{(1-\beta \tau)^n}{2} \phi(U^n_\taueta) + \frac{\beta\tau}{2(1-\beta\tau)} \sum_{k=1}^n (1-\beta \tau)^k \phi(U^k_\taueta) & \le \frac{1}{2} \phi(U^0_\taueta) ;
\end{aligned}
\end{equation}
recalling the definition of {Moreau-Yosida} approximation and setting
$$ \Phi_n:= - \sum_{k=1}^n (1-\beta \tau)^k \phi(U^k_\taueta) , \quad \Phi_0:=0 , $$
from \eqref{eq:est-disc-6} we can therefore deduce
\begin{equation}\label{eq:est-disc-7}
\Phi_n-\Phi_{n-1} - \frac{\beta\tau}{1-\beta\tau} \Phi_n \le \phi(U^0_\taueta)-2(1-\beta\tau)^n \phi_{\frac{1-(1-\beta\tau)^n}{\beta}}(U^0_\taueta) \quad \forevery n \in \N \setminus \{ 0 \} ;
\end{equation} 
by iterating \eqref{eq:est-disc-7} and using the fact that $ \tau \mapsto \phi_\tau(U^0_\taueta) $ is nonincreasing, we end up with
\begin{equation}\label{eq:est-disc-8}
\begin{aligned}
 \frac{\beta\tau}{1-\beta\tau} \Phi_n 
\le  \big[ (1-\beta\tau)^n-1 \big] \phi(U^0_\taueta) + \left[\left( \frac{1-\beta\tau}{1-2\beta\tau} \right)^n - (1-\beta\tau)^n \right] \left[ \phi(U^0_\taueta) - \phi_{\frac{1-(1-\beta\tau)^n}{\beta}}(U^0_\taueta) \right] .
\end{aligned}
\end{equation}
We point out that the r.h.s.~of \eqref{eq:est-disc-8} is finite thanks to the assumptions on $ \beta\tau $ and $ \phi $. Now we reconsider \eqref{eq:est-disc-3} and still apply \eqref{eq:6} to the last term in the r.h.s.\ (together with \eqref{eq:est-disc-5}), to get 
\begin{equation*}\label{eq:est-disc-9}
\begin{aligned}
& \frac{\beta (1-\beta\tau)^n}{2[1-(1-\beta\tau)^n]} \DistSquare{U^n_\taueta}{U^0_\taueta} \\
\le & \phi(U^0_\taueta) - (1-\beta\tau)^n\left[ \frac{\beta}{2[1-(1-\beta\tau)^n]} \DistSquare{U^n_\taueta}{U^0_\taueta}  + \phi(U^n_{\taueta}) \right] + \frac{\beta\tau}{1-\beta\tau} \Phi_n ;
 \end{aligned}
\end{equation*}
estimate \eqref{eq:est-disc-8} plus once again the definition {Moreau-Yosida} approximation then give rise to 
\begin{equation}\label{eq:est-disc-10}
\frac12 \DistSquare{U^n_\taueta}{U^0_\taueta} \le \frac{1-(1-\beta\tau)^n}{\beta (1-2\beta\tau)^n} \left[ \phi(U^0_\taueta) - \phi_{\frac{1-(1-\beta\tau)^n}{\beta}}(U^0_\taueta) \right] .
\end{equation}
The validity of \eqref{mm-c-1} is ensured upon recalling \eqref{eq:46}, the definition \eqref{eq:cap1:66} of piecewise constant interpolant $ \Pc U \taueta (t) $ and the elementary inequality
\begin{equation}\label{eq:est-elm-new}
(1-x)^{-1} \le \rme^{\frac{x}{1-x}} \quad \forevery  x \in (0,1) ,
\end{equation}
which implies 
\begin{equation}\label{eq:est-elm-new-bis}
(1-2\beta\tau)^{-n} \le \rme^{2\beta(1+4\beta \tau)t_\tau} \quad \text{and} \quad (1-\beta\tau)^{-n} \le \rme^{\beta\left(1+ \frac 4 3 \beta \tau\right)t_\tau} \qquad \text{provided }\,  4\beta\tau \le 1 .
\end{equation}
In order to prove \eqref{mm-c-2}, let $ n,h \in \N $ be such that $ t \in ((n-1)\tau,n\tau] $ and $ s \in ((h-1)\tau,h\tau] $, namely $ n = t_\tau/\tau $ and $ h = s_\tau/\tau $. Similarly to \eqref{eq:est-disc-5}, we have:
\begin{equation*}\label{eq:est-disc-5-bis}
\DistSquare{U^{n+h}_\taueta}{U^n_\taueta} \le \frac{1-\beta\tau}{\beta\tau} \big[ (1-\beta\tau)^{-h}-1 \big] \sum_{k=1}^h (1-\beta\tau)^{k-1} \DistSquare{U^{k+n}_\taueta}{U^{k+n-1}_\taueta} .
\end{equation*}
On the other hand, \eqref{eq:5} and the left-hand inequality in \eqref{eq:cap4:2a} yield (upon iteration)
\begin{equation*}\label{eq:est-disc-11}
\Dist{U^{k+n}_\taueta}{U^{k+n-1}_\taueta} \le (1-\beta\tau)^{-n} (1-\tfrac{1}{2}\eta\tau)^{-n} \Dist{U^k_\taueta}{U^{k-1}_\taueta} , 
\end{equation*}
whence 
\begin{equation}\label{eq:est-disc-5-ter}
\DistSquare{U^{n+h}_\taueta}{U^n_\taueta} \le \frac{1-\beta\tau}{\beta\tau} \frac{(1-\beta\tau)^{-h}-1}{(1-\beta\tau)^{2n} (1-\tfrac{1}{2}\eta\tau)^{2n}} \sum_{k=1}^h (1-\beta\tau)^{k-1} \DistSquare{U^{k}_\taueta}{U^{k-1}_\taueta} .
\end{equation} 
Now we observe that, by the above method of proof, estimate \eqref{eq:est-disc-10} still holds if we replace $ \DistSquare{U^n_\taueta}{U^0_\taueta} $ with the r.h.s.\ of \eqref{eq:est-disc-5}, so that this refined information plus \eqref{eq:est-disc-5-ter} entail
\begin{equation}\label{eq:est-disc-12}
\frac 12 \DistSquare{U^{n+h}_\taueta}{U^n_\taueta} \le \frac{1-(1-\beta\tau)^h}{\beta(1-2\beta\tau)^h (1-\beta\tau)^{2n} (1-\tfrac{1}{2}\eta\tau)^{2n}} \left[ \phi(U^0_\taueta) - \phi_{\frac{1-(1-\beta\tau)^h}{\beta}}(U^0_\taueta) \right] .
\end{equation}
Estimate \eqref{mm-c-2} is therefore a consequence of \eqref{eq:est-disc-12} up to exploiting again \eqref{eq:est-elm-new} as we did in \eqref{eq:est-elm-new-bis} (now with $ x=\beta\tau $, $ x=2\beta\tau $ and $ x=\eta\tau/2 $), noting that $ U^n_\taueta = \Pc U \taueta (t) $ and $ U^{n+h}_\taueta = \Pcshift{U}{\taueta}{s}(t) $ by the definition of piecewise constant interpolant. 

We are then left with establishing \eqref{mm-c-3}. To this aim, first of all note that by virtue of the right-hand inequality in \eqref{eq:cap4:2a} we know that the map $ n \mapsto (1-\beta\tau)^n \MetricSlope{\phi}{U^n_\taueta} $ is not increasing. Furthermore, \eqref{eq:5} and \eqref{eq:6} yield 
\begin{equation}\label{eq:est-disc-13}
\tau (1-\tfrac{1}{2}\eta\tau)^2 (1-\tfrac{1}{2}\beta\tau) \MetricSlopeSquare\phi{U^n_\taueta} \le \phi(U^{n-1}_\taueta)-\phi(U^n_\taueta) ; \end{equation}
if we plug \eqref{eq:est-disc-13} into \eqref{eq:est-disc-3} and proceed exactly as above to estimate the remaining terms, we obtain:
\begin{equation}\label{eq:est-disc-14}
\begin{aligned}
\frac{(1-\beta\tau)^n[1-(1-\beta\tau)^n]}{2\beta} \MetricSlopeSquare\phi{U^n_\taueta} = & \frac{\tau}{2} \sum_{k=1}^n (1-\beta\tau)^{-k} \big[ (1-\beta\tau)^{n} \MetricSlope\phi{U^n_\taueta} \big]^{2} \\
\le & \frac{\tau}{2} (1-\tfrac{1}{2}\beta\tau) \sum_{k=1}^n (1-\beta\tau)^{k-1} \MetricSlopeSquare\phi{U^k_\taueta} \\
\le & (1-\tfrac{1}{2}\eta\tau)^{-2} \left( \frac{1-\beta\tau}{1-2\beta\tau} \right)^n \left[ \phi(U^0_\taueta) - \phi_{\frac{1-(1-\beta\tau)^n}{\beta}}(U^0_\taueta) \right] .
\end{aligned}
\end{equation}
Estimate \eqref{mm-c-3} finally follows from \eqref{eq:est-disc-14}, by recalling \eqref{eq:46} along with the definition of piecewise constant interpolant and using \eqref{eq:est-elm-new-bis}, the inequality $ (1-x) \le e^{-x} $, valid for all $ x \in (0,1) $, applied to $ x=\beta \tau $ and the elementary estimate
$$
(1-\tfrac{1}{2}\eta\tau)^{-2} \le 1+2\eta \tau \qquad \text{ensured by } 4\eta\tau \le 4\beta\tau \le 1 .
$$
\end{proof}

We are now able to establish the analogue of Theorem \ref{lem:Error-est-approx} for initial (discrete) data that merely belong to $ \Dom\phi $. For simplicity, here we treat the case $ \eps=0 $ only.

\begin{theorem}[Uniform error estimate for data in $\Dom\phi$]
  \label{thm:Error_Estimate}
  Let the following assumptions hold:
  \begin{enumerate}[\rm 1.]
  \item the \EVIshort\lambda-gradient flow $(\lambda\le 0)$ $\FlowName_t$ of $\phi$ exists in
    $D\subset \overline{\Dom\phi}$;
  \item for some $\tau \in (0,1) $ and $\eta\ge0$ such that $ 4\gamma\tau \le 1 $ (let $ \gamma $ be as in \eqref{eq:46}), the sequence $(U^n_\taueta)_{n\in \N}\subset D \cap \Dom{\phi} $ is a $(\tau,\eta)$-discrete Minimizing Movement, according to Definition \ref{def:MMS}; 
  \item $\phi$ is approximately $ \lambda $-convex in $D$.
  \end{enumerate}
Then, for all $ u_0 \in D $, the following estimate holds for every $ t>0 $:
  \begin{equation} \label{eq:24-bis}  
 \Dist{\FlowName_t(u_0)}{\Pc U\taueta(t)} \le  \rme^{-\lambda t} \Dist{u_0}{U^0_\taueta} + 10 \, \sqrt[4]{\tau t_\tau} \, \rme^{2 \gamma t_\tau} \sqrt{\phi(U^0_\taueta)-\phi_{\sfE{\lambda-\eta}{3\sqrt{\tau t_\tau}}}(U^0_\taueta)} .
  \end{equation}
%where $ C $ is a positive numerical constant that can be taken equal to
%$$
%\GR{
%C= 10 \, \sqrt{1+\frac{\sqrt{\tau}+2\tau}{t+1}} \, \rme^{4(\eta-\lambda)(\sqrt{\tau}+2\tau)} .
%}
%$$ 
\end{theorem}
Note that once condition 1.\ holds, then by virtue of Theorem \ref{thm:Daneri} condition 3.\ is also satisfied as long as $ D $ is an approximate length subset. 
% condition 3.\ is satisfied if \emph{e.g.}~$\GeoCon D{U^{n-1}_\taueta}{U^n_\taueta}\neq\emptyset$.
\begin{proof}

In order to establish \eqref{eq:24-bis}, the idea is to start from the trivial inequality 
\begin{equation}\label{eq: triv-ineq}
\Dist{\FlowName_t(U^0_\taueta)}{\Pc{U}{\taueta}(t)} \le \Dist{\FlowName_t(U^0_\taueta)}{\Pcshift{U}{\taueta}{s}(t)} + \Dist{\Pcshift{U}{\taueta}{s}(t)}{\Pc{U}{\taueta}(t)} ,
\end{equation}
valid for all $ s>0 $. If we apply \eqref{eq:24:epsilon} with $  \Pc{U}{\taueta} $ replaced by $ \Pcshift{U}{\taueta}{s} $, $ u_0=U^0_\taueta $ and $ \varepsilon=0 $, 
%together with the estimate 
%$$ \rme^{-\lambda t} \Dist{\FlowName_{m\tau-t}(U^0_\taueta)}{U^0_\taueta} \le 2\sqrt{\tau} \, \rme^{-\lambda(t+\tau)} \, \sqrt{\phi(U^0_\taueta)-\phi_{\sfE{\lambda}{\tau}}(U^0_\taueta)} \, , $$ 
%direct consequence of \eqref{eq:3-lessreg}, 
we obtain:
  \begin{equation} \label{eq:24-shift}
  \begin{aligned}
%     \Dist{\FlowName_{t}(U^0_\taueta)}{\Pcshift{U}{\taueta}{s}(t)}
%%     \le & \rme^{-\lambda(t+\tau)} \left[ \tau \MetricSlope{\phi}{u_0} +  \rme^{2\beta(t+\tau)} \big( \Dist{u_0}{U^0_\taueta} + \sqrt {(t+\tau)\tau} \MetricSlope\phi{U^0_\taueta} \big) \right]  \quad \forevery t > 0 . \\
%     \le & \rme^{(2\eta-3\lambda)(t+\tau)} \Dist{u_0}{U^0_\taueta} + \rme^{(2\eta-3\lambda)(t+\tau)} \Dist{\Pc{U}{\taueta}(s)}{U^0_\taueta} \\ & + 2\sqrt{\tau} \, \rme^{-\lambda(t+\tau)} \, \sqrt{\phi(u_0)-\phi_{\sfE{\lambda}{\tau}}(u_0)} +  \sqrt{(t+\tau)\tau} \, \rme^{(2\eta-3\lambda)(t+\tau)} \, \MetricSlope{\phi}{\Pc{U}{\taueta}(s)} .  \\
    \Dist{\FlowName_{t}(U^0_\taueta)}{\Pcshift{U}{\taueta}{s}(t)} 
    \le  \rme^{-\lambda t} \Dist{U^0_\taueta}{\Pc{U}{\taueta}(s)}  +
     \left( \sqrt{\tau
       t_\tau} + t_\tau - t \right) \rme^{\gamma t_\tau} \, \MetricSlope\phi{\Pc{U}{\taueta}(s)} \quad \forevery t > 0 . 
   \end{aligned}  
  \end{equation} 
Hence by exploiting \eqref{eq:cap1:81}, \eqref{eq: triv-ineq}, \eqref{eq:24-shift} and \eqref{mm-c-1}--\eqref{mm-c-3}, we deduce the estimate
\begin{equation}\label{eq:almost-final}
\begin{aligned}
\Dist{\FlowName_t(u_0)}{\Pc U\taueta(t)} \le \, &  \rme^{-\lambda t} \Dist{u_0}{U^0_\taueta} + \rme^{\beta(t_{\tau,\beta}+s_{\tau,\beta})+\frac \eta 2 t_{\tau,\beta}} \, \sqrt{2 \, \sfE{-\beta}{s_{\tau,\beta}}  \, \big[ \phi(U^0_\taueta) - \phi_{\sfE{-\beta}{s_{\tau,\beta}} }(U^0_\taueta) \big]} \\
& + \rme^{-\lambda t + \beta s_{\tau,\beta}} \, \sqrt{2 \, \sfE{-\beta}{s_{\tau,\beta}}  \, \big[ \phi(U^0_\taueta) - \phi_{\sfE{-\beta}{s_{\tau,\beta}} }(U^0_\taueta) \big]} \\
& + \left( \sqrt{\tau t_\tau} + t_\tau - t \right) \rme^{\gamma t_\tau + \beta s_{\tau,\beta} } \, \sqrt{\frac{2+4\eta\tau}{\sfE{-\beta}{s_\tau}} \, \big[ \phi(U^0_\taueta) - \phi_{ \sfE{-\beta}{s_{\tau,\beta}} }(U^0_\taueta) \big]} .
\end{aligned}
\end{equation}
Now we make the choice $ s=\sqrt{\tau t_\tau} $ and try to simplify \eqref{eq:almost-final} as much as possible by means of the following elementary inequalities, valid under the running assumptions on the parameters $ \tau $, $\eta$ and $ \lambda $:
\begin{equation}\label{simp-1}
s_{\tau,\beta} \le 3 \sqrt{\tau t_\tau}, \quad \sfE{-\beta}{s_{\tau,\beta}} \le 3 \sqrt{\tau t_\tau}, \quad \sfE{-\beta}{s_\tau} \ge \sqrt{\tau t_\tau} \, \rme^{-\frac{\beta}{2}t_\tau - \frac{1}{16}}, \quad \sqrt{\tau t_\tau} + t_\tau - t \le 2 \sqrt{\tau t_\tau}  , 
\end{equation}
and
\begin{equation}\label{simp-2}
\quad 2+4\eta\tau \le \frac{5}{2} , \quad \rme^{\beta(t_{\tau,\beta}+s_{\tau,\beta})+\frac \eta 2 t_{\tau,\beta}} \le \rme^{2\gamma t_\tau + \frac{3}{16}} , \quad \rme^{-\lambda t + \beta s_{\tau,\beta}} \le \rme^{\gamma t_\tau + \frac{3}{16}}, \quad  \rme^{\gamma t_\tau + \beta s_{\tau,\beta} } \le   \rme^{\gamma t_\tau + \frac{3}{2} \beta t_\tau + \frac{3}{16}} .
\end{equation}
% together with the elementary inequality $ \sfE{-\beta}{t} \ge \rme^{-\beta t}t $ and the assumption $ (2\eta-3\lambda)\tau \le 1-\rme^{-1}  $, which in particular implies $ \sftau{\tau}{-\beta}{t} \le 2 \sfE{-2\beta}{t} \le 2t $, we end up with  
%  \begin{equation} \label{eq:25-shift}
%  \begin{aligned}
%     \Dist{u_t}{\Pc{U}{\taueta}(t)}
%%     \le & \rme^{-\lambda(t+\tau)} \left[ \tau \MetricSlope{\phi}{u_0} +  \rme^{2\beta(t+\tau)} \big( \Dist{u_0}{U^0_\taueta} + \sqrt {(t+\tau)\tau} \MetricSlope\phi{U^0_\taueta} \big) \right]  \quad \forevery t > 0 . \\
%     \le &  \rme^{-\lambda t} \Dist{u_0}{U^0_\taueta} + 2 \sqrt{s+\tau} \, \rme^{2\beta(s+\tau)-\lambda t} \,  
%     \sqrt{\phi(U^0_\taueta)-\phi_{2\sfE{-2\beta}{s+\tau}}(U^0_\taueta)}   \\ 
%      & + 2 \sqrt{\frac{\tau(t+2\tau)}{s}} \, \rme^{2\beta(t+s+2\tau)-\lambda(t+\tau)+\beta s/2+1} \, \sqrt{\phi(U^0_\taueta)-\phi_{2\sfE{-2\beta}{s+\tau}}(U^0_\taueta)}                \\
%     & + 2 \sqrt{s+\tau} \, \rme^{2\beta(t+s+2\tau)+2\eta(t+\tau)} \, \sqrt{\phi(U^0_\taueta)-\phi_{2\sfE{-2\beta}{s+\tau}}(U^0_\taueta)}  . 
%   \end{aligned}
%  \end{equation} 
The validity of \eqref{eq:24-bis} is then ensured by \eqref{eq:almost-final}--\eqref{simp-2} up to some further trivial numerical inequalities. Note that the right-hand side is surely finite: indeed, by virtue of Theorem \ref{thm:main1} we know that $ \phi $ is quadratically bounded from below for all $ \kappa_o > -\lambda $, and in the case where $ \lambda<0 $ there holds $ \sfE{-\beta}{3\sqrt{\tau t_\tau}} < 1 / \beta \le -1/\lambda $.
\end{proof}

Let us point out that in the simplified case $\lambda=\eta=0$ with $\inf_\AmbientSpace \phi>-\infty$, if we set $n\mapsto U^n_\tau=U^n_{\tau,0}$, choose a uniform partition of step size $\tau=t/n$ and an initial datum $ U^0_\tau=u_0 \in \Dom \phi $, we obtain the following uniform error estimate of order $1/4$:
\begin{equation*}
  \label{eq:48}
  \Dist{\FlowName_t(u_0)}{U^n_\tau} = \Dist{\FlowName_t(u_0)}{\Pc U\tau(t)} \le 10 \, \frac{\sqrt t}{\sqrt[4]{n}} \, \sqrt{\phi(u_0)-\inf_\AmbientSpace \phi} .
\end{equation*}
Moreover, if $U^0_\taueta\in \Dom{|\partial\phi|}$ the last term in \eqref{eq:24-bis} involving the Moreau-Yosida regularization is of order $ \sqrt{\tau}$ (recall \eqref{eq:duality-slope}), so that the latter reproduces the same convergence rate as \eqref{eq:24-1}, up to different multiplicative constants (in order to end up with a much more readable estimate we have given up ``optimal'' constants in \eqref{eq:24-bis}).

\appendix
\section{Appendix: right Dini derivatives}\label{app}
% \subsection{Right Dini derivatives} 

For every real function $\zeta:[a,b)\to \R$ and $t\in [a,b)$
we consider the lower and upper right \emph{Dini derivatives}
\begin{equation*}
  \label{eq:cap1:8}
  \Lrd\zeta(t):=\liminf_{h\downarrow0}\frac
  {\zeta(t+h)-\zeta(t)}h,\quad
  \Urd\zeta(t):=\limsup_{h\downarrow0}\frac {\zeta(t+h)-\zeta(t)}h.
\end{equation*}
In the beginning of Section \ref{sec: gf} we take advantage of the following basic lemma (see e.g.~\cite{Gal57}).
\begin{lemma}
  \label{le:monotonicity}
  Let $\zeta,\eta: (a,b)\to \R$ be lower semicontinuous functions. If
  \begin{equation}
    \label{eq:cap1:9}
    \Lrd\zeta(t)+\eta(t)\le 0\quad \forevery t\in (a,b),
  \end{equation}
  then $\eta$ is locally integrable in $(a,b)$ and for every $ a_0 \in (a,b) $ the function
  \begin{equation*}
    \label{eq:34}
    \tilde \zeta(t):=\zeta(t)+\int_{a_0}^t \eta(s)\,\d s\quad
    \text{is nonincreasing and right continuous in $ (a,b) $.}
  \end{equation*}
In particular, \eqref{eq:cap1:9} is equivalent to 
  \begin{equation*}
    \label{eq:32}
    \Urd\zeta(t)+\eta(t)\le 0 \quad \forevery t\in (a,b)
 \end{equation*}
 and implies the distributional inequality $  \frac \d{\dt}\zeta+\eta\le 0 $, i.e.
 \begin{equation} \label{eq:33}
 \zeta,\eta\in L^1_{\rm loc}((a,b)) \quad \text{and} \quad
       \int_a^b \left(-\zeta\psi'+\eta\psi\right) \dt\le 0 
        \quad \text{for every nonnegative }\psi\in C^\infty_{\rmc}((a,b)),
 \end{equation}
which is in turn equivalent to \eqref{eq:cap1:9} under the additional assumption that $ \zeta $ is right continuous.
\end{lemma}
\begin{proof}
  Let us first consider the case $\eta=0$.
  If $a < t_0<t_0+\tau<b $ existed with
  $\delta:=\tau^{-1}\big(\zeta(t_0+\tau)-\zeta(t_0)\big)>0$, then a
  minimum point $\bar t\in [t_0,t_0+\tau)$ of $ t\mapsto
  \zeta_\delta(t):=\zeta(t)-\zeta(t_0)-\delta(t-t_0)$ would satisfy
  \begin{displaymath}
    \liminf_{h\downarrow 0}
    \frac{\zeta_\delta(\bar t+h)-\zeta_\delta(\bar t)}h=
    \lrd\zeta(\bar t)-\delta\ge
    0,\qquad
    \text{which contradicts \eqref{eq:cap1:9}}.
  \end{displaymath}
The right continuity is then a trivial consequence of the lower semicontinuity. In the general case, since $\eta$ is lower semicontinuous it admits a minimum $m(c,d)$ in every interval
  $[c,d]\subset (a,b)$. It follows that the function
  $\zeta_{c,d}(t):=\zeta(t)+m(c,d)t$ satisfies
  \begin{displaymath}
    \lrd\zeta_{c,d}(t)\le -(\eta(t)-m(c,d))\le 0 \quad \forevery t\in (c,d),
  \end{displaymath} 
  whence $\zeta_{c,d}$ is nonincreasing in $(c,d)$ thanks to the first part of the proof. Thus $\zeta_{c,d}$ is right continuous in $ (c,d) $, is differentiable $\Leb 1$-a.e.~in $(c,d)$,
  its pointwise derivative $\dot\zeta_{c,d}$ coincides with $\lrd$ up to a $\Leb 1$-negligible set and $\dot\zeta_{c,d}\in L^1_{\rm loc}((c,d))$; since $\eta \le -\lrd \zeta_{c,d} + m(c,d) $ and $ \eta $ is locally bounded from below, it follows that also $\eta\in L^1_{\rm loc}((c,d))$. All of the just proved properties being independent of $ a<c<d<b $, we have established that $ \eta \in L^1_{\rm loc}((a,b)) $ and that $ \zeta $ is right continuous. 

We can finally introduce the primitive function $H(t):=\int_{a_0}^t\eta(s)\,\d s$. Because $\frac\d{\dt}\zeta_{c,d}\le \dot\zeta_{c,d}$ in
  $\DD'((c,d))$, there holds
  \begin{equation*}
    \label{eq:35}
    \frac\d{\dt}\tilde\zeta=\frac\d{\dt}(\zeta+H)=
    \frac\d{\dt}(\zeta_{c,d}+H-m(c,d)t)\le
    \dot\zeta_{c,d}+\eta-m(c,d)=\Lrd\zeta+\eta\le 0,
  \end{equation*}
still in the sense of distributions in $(c,d)$. On the other hand, since $ a<c<d<b $ are arbitrary and $\tilde\zeta=\zeta+H$ is right continuous, it follows that it is nonincreasing in $ (a,b) $. The last assertions concerning \eqref{eq:33} are direct consequences of the previous results.
\end{proof} 

\bibliographystyle{siam}
\bibliography{bibliografia2015}

\end{document}